\documentclass[a4paper,reqno, 11pt]{amsart} 
\usepackage{latexsym,amsthm,amsmath,amssymb}
\usepackage[usenames,dvipsnames,svgnames,table]{xcolor}
\usepackage[utf8x]{inputenc}

\usepackage{enumerate}
\usepackage[normalem]{ulem}
\usepackage{pgf, tikz}
\usetikzlibrary{automata}
\usetikzlibrary{shapes}
\usetikzlibrary{calc}
\usetikzlibrary{decorations}
\usetikzlibrary{decorations.pathmorphing}
\usetikzlibrary{decorations.markings}

\newtheorem{theorem}{Theorem}
\newtheorem{corollary}[theorem]{Corollary}
\newtheorem{lemma}[theorem]{Lemma}

\newtheorem{innercustomgeneric}{\customgenericname}
\providecommand{\customgenericname}{}
\newcommand{\newcustomtheorem}[2]{%
  \newenvironment{#1}[1]
  {%
   \renewcommand\customgenericname{#2}%
   \renewcommand\theinnercustomgeneric{##1}%
   \innercustomgeneric
  }
  {\endinnercustomgeneric}
}

\newcustomtheorem{customthm}{Theorem}

\colorlet{mylinkcolor}{violet}
\colorlet{mycitecolor}{YellowOrange}
\colorlet{myurlcolor}{Aquamarine}
\usepackage[unicode=true]{hyperref}
\hypersetup{
    colorlinks,
    linkcolor={red!50!black},
    citecolor={blue!50!black},
    urlcolor={blue!80!black}
}

\newcommand{\RR}{\mathbb{R}}
\newcommand{\NN}{\mathbb{N}}
\newcommand{\np}{\textsf{NP}}
\DeclareMathOperator{\delete}{\setminus}

\newcommand{\ks}{\mathsf{S}} 
\newcommand{\kf}{\mathsf{F}} 
\newcommand{\kp}{\mathsf{P}} 
\newcommand{\kn}{\mathsf{N}} 
\newcommand{\kall}{\mathcal{U}_\infty}
\newcommand{\kl}{\mathsf{L}} 

\newcommand{\ddir}[3]{\langle #1, #2; #3 \rangle} 
\newcommand{\ip}[2]{\left\langle #1, #2 \right\rangle}
\newcommand{\tw}{\mathop{\mathrm{tw}}}
\newcommand{\polylog}{\mathop{\mathrm{poly log}}}

\begin{document}

\title[Unavoidable minors for graphs with large $\ell_p$-dimension] {Unavoidable minors for graphs with large $\ell_p$-dimension}

\author[S.~Fiorini]{Samuel Fiorini}
\address[S.~Fiorini]{\newline D\'epartement de Math\'ematique
\newline Universit\'e Libre de Bruxelles
\newline Brussels, Belgium}
\email{sfiorini@ulb.ac.be}

\author[T.~Huynh]{Tony Huynh}
\address[T.~Huynh]{\newline D\'epartement de Math\'ematique
\newline Universit\'e Libre de Bruxelles
\newline Brussels, Belgium}
\email{tony.bourbaki@gmail.com}

\author[G.~Joret]{Gwena\"el Joret}
\address[G.~Joret]{\newline D\'epartement d'Informatique
\newline Universit\'e Libre de Bruxelles
\newline Brussels, Belgium}
\email{gjoret@ulb.ac.be}

\author[C.~Muller]{Carole Muller}
\address[C.~Muller]{\newline D\'epartement de Math\'ematique
\newline Universit\'e Libre de Bruxelles
\newline Brussels, Belgium}
\email{camuller@ulb.ac.be}

\thanks{S. Fiorini and T. Huynh are supported by ERC Consolidator Grant 615640-ForEFront. G. Joret is supported by an ARC grant from the Wallonia-Brussels Federation of Belgium. C. Muller is supported by the Luxembourg National Research Fund (FNR) Grant Nr. 11628910.}

\maketitle 

\begin{abstract}
    A \emph{metric graph} is a pair $(G,d)$, where $G$ is a graph and $d:E(G) \to\mathbb{R}_{\geq0}$ is a distance function. Let $p \in [1,\infty]$ be fixed. An \emph{isometric embedding} of the metric graph $(G,d)$ in $\ell_p^k = (\mathbb{R}^k, d_p)$ is a map $\phi : V(G) \to \mathbb{R}^k$ such that $d_p(\phi(v), \phi(w)) = d(vw)$ for all edges $vw\in E(G)$. The \emph{$\ell_p$-dimension} of $G$ is the least integer $k$ such that there exists an isometric embedding of $(G,d)$ in $\ell_p^k$ for all distance functions $d$ such that $(G,d)$ has an isometric embedding in $\ell_p^K$ for some $K$.
    
    It is easy to show that $\ell_p$-dimension is a minor-monotone property.  
    In this paper, we characterize the minor-closed graph classes $\mathcal{C}$ with bounded $\ell_p$-dimension, for $p \in \{2,\infty\}$. 
    For $p=2$, we give a simple proof that $\mathcal{C}$ has bounded $\ell_2$-dimension if and only if $\mathcal{C}$ has bounded treewidth. 
    In this sense, the $\ell_2$-dimension of a graph is `tied' to its treewidth. 
    
    For $p=\infty$, the situation is completely different. 
    Our main result states that a minor-closed class $\mathcal{C}$ has bounded $\ell_\infty$-dimension if and only if $\mathcal{C}$ excludes a graph obtained by joining copies of $K_4$ using the $2$-sum operation, or excludes a M\"obius ladder with one `horizontal edge' removed.  
\end{abstract}

\section{Introduction}
    
    In this paper, we consider isometric embeddings of metric graphs in metric spaces. Recall that a \emph{metric space $(X,d)$} consists of a \emph{set of points} $X$ and a \emph{metric} $d: X \times X \to \RR_{\geq 0}$.  That is, for all $x,y,z\in X$, (i)~$d(x,y) = d(y,x)$, (ii)~$d(x,y) = 0$ if and only if $x = y$, and (iii)~$d(x,y)\leq d(x,z)+d(z,y)$. 
    Here, we only consider the metric spaces ${\ell_p^k = (\RR^k, d_p)}$, focusing mainly on the cases $p \in \{2, \infty\}$. We let $\NN$ denote the set of positive integers, and for $k \in \NN$, $[k] = \{1, \dots, k\}$. Recall that $\lVert x \rVert_p = (\sum_{i = 1}^k \lvert x \rvert^p)^{1/p}$ if $p\in [1, \infty)$ and $\lVert x \rVert_\infty = \max_{i \in [k]}{\lvert x_i \rvert}$. We set $d_p(x,y) = \lVert x-y \rVert_p$ for all $p\in [1, \infty]$.
    
    Comparing different metric spaces is a ubiquitous theme throughout mathematics. One way to do so is by means of \emph{isometric embeddings}, which are functions $\phi: X\to X'$ such that $d(x,y) = d'(\phi(x), \phi(y))$ for all $x,y\in X$. As these are quite restrictive, other approaches have been developed. For instance, Bourgain \cite{Bourgain1985} has shown that every $n$-point metric space can be embedded into an $\ell_p^{O(\log^2 n)}$ space with $O(\log n)$ distortion. (The upper bound on the dimension was subsequently reduced to $O(\log n)$, see~\cite{ABN11}.)
    
    Another point of view is to require only a subset of distances to be preserved, which is the perspective we take in this paper.  
    Our methods are mostly graph theoretical, although similar problems have been studied using techniques from rigidity theory~\cite{KNS18, Schulze19, SW15}.  
    
    All graphs in this paper are finite and do not contain loops or parallel edges, unless otherwise stated. 
    A graph $H$ is a \emph{minor} of a graph $G$ if $H$ can be obtained from a subgraph of $G$ by contracting some edges. When taking minors we remove parallel edges and loops resulting from edge contractions.
    
    A \emph{metric graph} $(G,d)$ is a pair consisting of a graph $G$ and a function ${d:E(G)\to\RR_{\geq 0}}$ satisfying $d(vw)\leq d(P) = \sum_{i = 1}^{r}d(v_{i-1}v_i)$ for all edges $vw\in E(G)$ and all paths $P = v_0v_1 \cdots v_r$ with $v_0 = v$ and $v_r = w$.  Such a function $d$ is called a \emph{distance function} on $G$. An \emph{isometric embedding} of a metric graph $(G, d)$ in $\ell_p^k$ is a map $\phi : V(G) \to \RR^k$ such that $d_p(\phi(v), \phi(w)) = d(vw)$ for all edges $vw\in E(G)$.
    
    For each $p\in [1,\infty]$ and graph $G$, a distance function ${d:E(G)\to\RR_{\geq 0}}$ is \emph{$\ell_p$-realizable} if it has an isometric embedding in $\ell_p^K$ for some $K$. If $d$ is $\ell_p$-realizable, we define the parameter $f_p(G,d)$ to be the least integer $k$ such that $(G,d)$ can be isometrically embedded in $\ell_p^k$.   The \emph{$\ell_p$-dimension} of $G$ is defined to be $f_p(G) = \sup_d f_p(G,d)$, where the supremum is over all $\ell_p$-realizable distance functions $d$ on $G$. 
    We remark that in the special case $p=\infty$, the supremum is taken over all distance functions on $G$, since it is well-known that every $n$-point metric space can be isometrically embedded into $\ell_\infty^{n-1}$.
    It is known that $\ell_p$-dimension is always at most $|V(G)| \choose 2$, see \cite{B90} and \cite[Proposition 11.2.3]{DL97}. The $\ell_2$-dimension is also referred to as \emph{Euclidean dimension}.
    
    It is easy to see that every minor $H$ of $G$ satisfies $f_p(H)\leq f_p(G)$ for all $p\in [1,\infty]$. Hence the property \emph{$f_p(G)\leq k$} is closed under taking minors. By the Graph Minor Theorem of Robertson and Seymour \cite{RS04}, for each $k$, there are only a finite number of minor-minimal graphs satisfying $f_p(G) > k$. Formally, an \emph{excluded minor} for $f_p(G)\leq k$ is a graph $H$ such that $f_p(H)> k$ and every proper minor $H'$ of $H$ satisfies $f_p(H')\leq k$. 
    
    The complete sets of excluded minors are known in the Euclidean case $p = 2$ for dimensions $k = 1,2,3$. Belk and Connelly \cite{Belk2007, BelkConnelly2007} have shown that $\{K_3\}$, $\{K_4\}$, $\{K_5, K_{2,2,2}\}$ are the respective sets of excluded minors. Furthermore, note that $\ell_p^1 = \ell_q^1$ for all $p,q \in [1,\infty]$.  Therefore, for all $p \in [1, \infty]$, $K_3$ is the only excluded minor for $f_p(G) \leq 1$. Fiorini, Huynh, Joret, and Varvitsiotis \cite{FHJV17} determined that $W_4$, the wheel on $5$ vertices, and the graph $K_4+_eK_4$ (see Figure~\ref{fig:excluded_minors_plane}) are the only excluded minors for $f_\infty(G) \leq 2$ and for $f_1(G) \leq 2$.  As far as we know, the complete set of excluded minors for $f_p(G) \leq k$ is unknown for all other values of $p$ and $k$.
    
    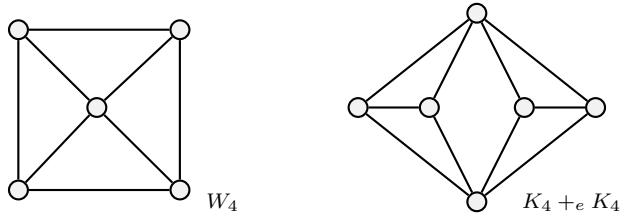
\begin{figure}
        \centering
        \begin{tikzpicture}[scale=1.25,inner sep=2.5pt]
            \tikzstyle{vtx} = [circle,draw,thick,fill=black!5]
            \begin{scriptsize}
            \node[vtx] (0) at (0,0) {};
            \node[vtx] (1) at (0,2) {};
            \node[vtx] (2) at (-0.5,1) {};
            \node[vtx] (3) at (-1.25,1) {};
            \node[vtx] (4) at (0.5,1) {};
            \node[vtx] (5) at (1.25,1) {};
            \node at (1,0) {$K_4 +_e K_4$};
            \node[vtx] (a) at (-3.125,0.125) {};
            \node[vtx] (b) at (-4.825,0.125) {};
            \node[vtx] (c) at (-4.825,1.825) {};
            \node[vtx] (d) at (-3.125,1.825) {};
            \node[vtx] (e) at (-4,1) {};
            \node at (-2.675,0) {$W_4$};
            \draw[thick] (0)--(2) (1)--(2) (0)--(3) (1)--(3) (2)--(3) (0)--(4) (1)--(4) (0)--(5) (1)--(5) (4)--(5) (a)--(b)--(c)--(d)--(a) (e)--(a) (e)--(b) (e)--(c) (e)--(d);
            \end{scriptsize}
        \end{tikzpicture}
        \caption{The excluded minors for $f_{\infty}(G) \leq 2$.}
        \label{fig:excluded_minors_plane}
    \end{figure}
    
    It is plausible that determining any further set of excluded minors will require significant effort, especially in dimension $3$ or higher (see~\cite{Muller2017}). Therefore, instead of obtaining exact characterizations of the graphs with $f_p(G) \leq k$, we take a different approach and seek collections of \emph{unavoidable minors}. That is, for each $k \in \mathbb N$, we look for a finite collection of graphs $\mathcal{U}_p^k$ and an integer $c_p(k)$, such that every graph $H \in \mathcal{U}_p^k$ has $f_p(H) > k$, and every graph $G$ with $f_p(G) > c_p(k)$ has a minor in $\mathcal{U}_p^k$. 
    
    For the case $p = 2$, we show that grids are unavoidable minors, see Theorem~\ref{thm:main2} in Section~\ref{sec:euclidean}. Most of the paper is devoted to the case $p = \infty$, which turns out to be much more challenging. Our main result is Theorem~\ref{thm:main} that gives unavoidable minors for $p = \infty$. 
    
    Now, we introduce the four graphs $\ks_k$, $\kp_k$, $\kf_k$ and $\kn_k$ that form $\kall^k$ for each $k \in \mathbb N$. Examples of all four graphs are given in Figure~\ref{fig:families}. The first three graphs are obtained by gluing together $k$ copies of $K_4$ in a certain way, and then deleting each edge that is common to at least two copies. The graph
    $\ks_k$ is obtained by gluing the $k$ copies of $K_4$ along one common edge. The graph $\kp_k$ 
    is obtained by picking a perfect matching $\{e_i, f_i\}$ in each copy of $K_4$, and identifying $f_i$ 
    and $e_{i+1}$ for all $i \in [k-1]$. The graph $\kf_k$ is constructed in a similar way, except that we 
    take $e_i$ and $f_i$ to be incident edges. Edges are identified in such a way that the common end of $e_i$ and
    $f_i$ is identified to the common end of $e_{i+1}$ and $f_{i+1}$ for all $i \in [k-1]$. The notation for these first three
    families reflect the fact that the corresponding copies of $K_4$ are arranged as a star, path, and fan, respectively. Notice that $\ks_2 = \kp_2 = \kf_2 = K_4+_eK_4$, which is one of the excluded minors for 
    $f_\infty(G)\leq 2$. Next, we define our final family of graphs. The graph $\kn_k$ is the graph with $V(\kn_k) = \{v_0, \dots, v_k\} \cup \{w_0, \dots, w_k\}$ and 
        \[
        E(\kn_k) = \{v_{i-1}v_i,v_{i}w_{i},v_{i-1}w_i,w_{i-1}w_i \mid i \in [k]\} \cup \{v_0w_0, w_0v_k\}.
        \]
    
    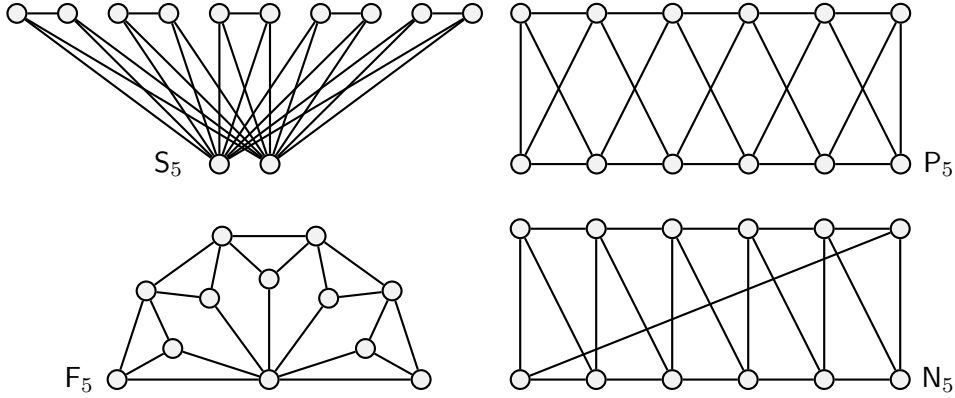
\begin{figure}
        \centering
        \begin{tabular}{c c}
		    \begin{tikzpicture}[scale=.333, inner sep=2.5pt]
		        \tikzstyle{vtx}=[circle,draw,thick,fill=black!5]
		        \begin{scriptsize}
		        \node[vtx](p0) at (1,-6){};
		        \node[vtx](p1) at (3,-6){};
		        \node[vtx](p2) at (-7,0){};
		        \node[vtx](p3) at (-5,0){};
		        \node[vtx](p4) at (-3,0){};
		        \node[vtx](p5) at (-1,0){};
		        \node[vtx](p6) at (1, 0){};
		        \node[vtx](p7) at (3,0){};
		        \node[vtx](p8) at (5, 0){};
		        \node[vtx](p9) at (7,0){};
		        \node[vtx](p10) at (9,0){};
		        \node[vtx](p11) at (11,0){};
		        \end{scriptsize}
		        \node at (-1,-6) {$\ks_5$};
		        \draw[thick] (p0)--(p2);
		        \draw[thick] (p0)--(p3);
		        \draw[thick] (p0)--(p4) ;
		        \draw[thick] (p0)--(p5) ;
		        \draw[thick] (p0)--(p6) ;
		        \draw[thick] (p0)--(p7) ;
		        \draw[thick] (p0)--(p8) ;
		        \draw[thick] (p0)--(p9);
		        \draw[thick] (p0)--(p10);
		        \draw[thick] (p0)--(p11);
		        \draw[thick] (p1)--(p2) ;
		        \draw[thick] (p1)--(p3) ;
		        \draw[thick] (p1)--(p4) ;
		        \draw[thick] (p1)--(p5) ;
		        \draw[thick] (p1)--(p6) ;
		        \draw[thick] (p1)--(p7);
		        \draw[thick] (p1)--(p8) ;
		        \draw[thick] (p1)--(p9) ;
		        \draw[thick] (p1)--(p10) ;
		        \draw[thick] (p1)--(p11) ;
		        \draw[thick] (p2)--(p3) ;
		        \draw[thick] (p4)--(p5) ;
		        \draw[thick] (p6)--(p7) ;
		        \draw[thick] (p8)--(p9) ;
		        \draw[thick] (p10)--(p11) ;
		    \end{tikzpicture}
		    &
		    \begin{tikzpicture}[x=0.5cm, y=1cm, inner sep=2.5pt]
                \tikzstyle{vtx}=[circle,draw,thick,fill=black!5]
                \begin{scriptsize}
                \node[vtx] (v0) at (0,0){};
                \node[vtx] (v1) at (0,2){};
                \node[vtx] (v2) at (2,0){};
                \node[vtx] (v3) at (2,2){};
                \node[vtx] (v4) at (4,0){};
                \node[vtx] (v5) at (4,2){};
                \node[vtx] (v6) at (6,0){};
                \node[vtx] (v7) at (6,2){};
                \node[vtx] (v8) at (8,0){};
                \node[vtx] (v9) at (8,2){};
                \node[vtx] (v10) at (10,0){};
                \node[vtx] (v11) at (10,2){};
                \end{scriptsize}
                
                \node at (11,0) {$\kp_5$};
   
                \draw[thick] (v0)--(v1) ;
                \draw[thick] (v0)--(v2) ;
                \draw[thick] (v0)--(v3);
                \draw[thick] (v1)--(v2);
                \draw[thick] (v1)--(v3) ;
                \draw[thick] (v2)--(v4) ;
                \draw[thick] (v2)--(v5);
                \draw[thick] (v3)--(v4);
                \draw[thick] (v3)--(v5) ;
                \draw[thick] (v4)--(v6) ;
                \draw[thick] (v4)--(v7);
                \draw[thick] (v5)--(v6);
                \draw[thick] (v5)--(v7) ;
                \draw[thick] (v6)--(v8) ;
                \draw[thick] (v6)--(v9) ;
                \draw[thick] (v7)--(v8) ;
                \draw[thick] (v7)--(v9) ;
                \draw[thick] (v8)--(v10) ;
                \draw[thick] (v8)--(v11);
                \draw[thick] (v9)--(v10);
                \draw[thick] (v9)--(v11) ;
                \draw[thick] (v10)--(v11);
            \end{tikzpicture}
            \\[2ex]
            \begin{tikzpicture}[inner sep=2.5pt]
		        \tikzstyle{vtx}=[circle,draw,thick,fill=black!5]
		        \begin{scriptsize}
                \node[vtx] (v0) at (0,0){};
                \node[vtx] (v1) at (180:2cm){};
                \node[vtx] (v2) at (162:1.333cm){};
                \node[vtx] (v3) at (144:2cm){};
                \node[vtx] (v4) at (126:1.333cm){};
                \node[vtx] (v5) at (108:2cm){};
                \node[vtx] (v6) at (90:1.333cm){};
                \node[vtx] (v7) at (72:2cm){};
                \node[vtx] (v8) at (54:1.333cm){};
                \node[vtx] (v9) at (36:2cm){};
                \node[vtx] (v10) at (18:1.333cm){};
                \node[vtx] (v11) at (0:2cm){};
                \end{scriptsize}
                \node at (-2.5,0) {$\kf_5$};
                \draw[thick] (v0)--(v1);
                \draw[thick] (v0)--(v2);
                \draw[thick] (v0)--(v4);
                \draw[thick] (v0)--(v6);
                \draw[thick] (v0)--(v8);
                \draw[thick] (v0)--(v10) ;
                \draw[thick] (v0)--(v11);
                \draw[thick] (v2)--(v1) ;
                \draw[thick] (v3)--(v1) ;
                \draw[thick] (v2)--(v3) ;
                \draw[thick] (v4)--(v3) ;
                \draw[thick] (v4)--(v5) ;
                \draw[thick] (v3)--(v5) ;
                \draw[thick] (v5)--(v6);
                \draw[thick] (v7)--(v6) ;
                \draw[thick] (v5)--(v7);
                \draw[thick] (v7)--(v8);
                \draw[thick] (v8)--(v9) ;
                \draw[thick] (v7)--(v9);
                \draw[thick] (v9)--(v11);
                \draw[thick] (v10)--(v11) ;
                \draw[thick] (v9)--(v10);
            \end{tikzpicture}
            &
            \begin{tikzpicture}[x=0.5cm, y=1.cm, inner sep=2.5pt]
                \tikzstyle{vtx}=[circle,draw,thick,fill=black!5]
                \begin{scriptsize}
                \node[vtx] (v0) at (0,0){};
                \node[vtx] (v1) at (0,2){};
                \node[vtx] (v2) at (2,0){};
                \node[vtx] (v3) at (2,2){};
                \node[vtx] (v4) at (4,0){};
                \node[vtx] (v5) at (4,2){};
                \node[vtx] (v6) at (6,0){};
                \node[vtx] (v7) at (6,2){};
                \node[vtx] (v8) at (8,0){};
                \node[vtx] (v9) at (8,2){};
                \node[vtx] (v10) at (10,0){};
                \node[vtx] (v11) at (10,2){};
                \end{scriptsize}
            
                \node at (11,0) {$\kn_5$};
   
                \draw[thick] (v0)--(v1);
                \draw[thick] (v2)--(v3);
                \draw[thick] (v4)--(v5);
                \draw[thick] (v6)--(v7);
                \draw[thick] (v8)--(v9);
                \draw[thick] (v10)--(v11);
            
                \draw[thick] (v0)--(v2) ;
                \draw[thick] (v1)--(v3) ;
                \draw[thick] (v2)--(v4) ;
                \draw[thick] (v3)--(v5) ;
                \draw[thick] (v4)--(v6) ;
                \draw[thick] (v5)--(v7) ;
                \draw[thick] (v6)--(v8) ;
                \draw[thick] (v7)--(v9) ;
                \draw[thick] (v8)--(v10) ;
                \draw[thick] (v9)--(v11) ;
            
                \draw[thick] (v1)--(v2) ;
                \draw[thick] (v3)--(v4) ;
                \draw[thick] (v5)--(v6) ;
                \draw[thick] (v7)--(v8) ;
                \draw[thick] (v9)--(v10) ;
                \draw[thick] (v11)--(v0) ;
            \end{tikzpicture}
        \end{tabular}
        
        \caption{The graphs $\ks_5$, $\kp_5$, $\kf_5$ and $\kn_5$.\label{fig:families}}
    \end{figure}
    
    For each $k \in \mathbb N$, we let $\kall^k = \{\ks_k, \kp_k, \kf_k, \kn_k\}$. We say that a graph $G$ \emph{contains a $\kall^k$ minor} if it contains $\ks_k, \kf_k, \kp_k$ or $\kn_k$ as a minor.  Our main theorem shows that if $f_\infty(G)$ is large, then $G$ necessarily contains a $\kall^k$ minor.

      \begin{theorem} \label{thm:main}
            There exists a computable function $g_{\ref{thm:main}}: \NN \to \RR$ such that for every $k\in \NN$, every
            graph $G$ with $f_\infty(G) > g_{\ref{thm:main}}(k)$ contains a $\kall^k$ minor. Moreover, every graph $G$ that contains a $\kall^k$ minor has $f_\infty(G) > k$.
        \end{theorem}

    Let $\mathcal{S}=\bigcup_k \{\ks_k\}, \mathcal{F} = \bigcup_k \{\kf_k\}, \mathcal{P} = \bigcup_k \{\kp_k\}$, and $\mathcal{N} = \bigcup_k \{\kn_k\}$.  For a class of graphs $\mathcal{C}$ and $p \in [1, \infty]$, we let $f_p(\mathcal C) = \max \{f_p(G) \mid G \in \mathcal{C}\}$, if this number is finite, and $f_p (\mathcal C) = \infty$, otherwise. 
    As an immediate corollary, our main theorem gives an exact characterization of all minor-closed classes $\mathcal{C}$ with $f_\infty (\mathcal C)=\infty$.  
    
    \begin{corollary}
    For all minor-closed classes of graphs $\mathcal{C}$, $f_\infty(\mathcal C) = \infty$ if and only if $\mathcal{S} \subseteq \mathcal{C}$ or  $\mathcal{F} \subseteq \mathcal{C}$ or  $\mathcal{P} \subseteq \mathcal{C}$ or  $\mathcal{N} \subseteq \mathcal{C}$.
    \end{corollary}
    
    The rest of the paper is organized as follows.  In Section~\ref{sec:euclidean}, we establish that grids are unavoidable minors for large $\ell_2$-dimension. 
    In Section~\ref{sec:alternative}, we give a more combinatorial definition of $\ell_\infty$-dimension.  In Section~\ref{sec:metrictools}, we establish some lemmas on $\ell_\infty$-dimension to be used later. 
    
    We establish the second part of our main result, Theorem~\ref{thm:main}, in Section~\ref{sec:certificates}, by constructing on each graph $G \in \kall^k$ a distance function $d$ that allows us to show $f_\infty(G,d) > k$ in a simple, combinatorial way. 
    
   In order to prove the first part of Theorem~\ref{thm:main}, we consider a graph $G$ without a $\kall^k$ minor and set out to prove that we can upper bound $f_\infty(G)$ by some integer $g_{\ref{thm:main}}(k)$. 
    
    It is straightforward to show that the $\ell_\infty$-dimension of a graph is the maximum $\ell_\infty$-dimension of one of its blocks (see Lemma~\ref{lem:gluing}). Therefore, we may assume that $G$ is $2$-connected. In Section~\ref{sec:reduction}, we prove that we can essentially assume that $G$ is $3$-connected. This part relies on SPQR trees. 
    
    The $3$-connected case is the part of the proof requiring most of the work. 
    The proof techniques here are mostly graph-theoretic, and may be of independent interest. 
    This is done in Section~\ref{sec:3connected} and Section~\ref{sec:finish}.

    \section{The Euclidean case} \label{sec:euclidean}
    
    The goal of this section is to establish that grids are a collection of unavoidable minors for large Euclidean dimension, which is the analogue of Theorem~\ref{thm:main} for $\ell_2$-dimension. 
    
     Let $r \in \NN$.  Recall that the \emph{square grid} graph $\Box_r$ is the graph with vertex set $[r] \times [r]$, where $(i,j)$ is adjacent to $(i',j')$ if and only if $|i-i'|+|j-j'|=1$. The \emph{triangular grid} graph $\triangle_r$ has vertex set $V(\triangle_r) = \{v_{i,j} \mid i, j \in [r],\  i \leq j\}$ and edge set $E(\triangle_r) = \{v_{i,j}v_{k,\ell} \mid v_{i,j}, v_{k,\ell} \in V(\triangle_r),\ (i-k,j-\ell) \in \{\pm(1,0),\pm(0,1),\pm (1,1)\}\}$. 
    
     Let $G$ and $H$ be graphs such that $H$ is a minor of $G$. Then $G$ contains an \emph{$H$-model}, that is, a collection $\{X_v \mid v \in V(H)\}$ of disjoint subsets $X_v \subseteq V(G)$ each inducing a connected subgraph of $G$ such that for every edge $vw \in E(H)$ there is an edge of $G$ with one end in $X_v$ and the other in $X_w$. The sets $X_v$ are called the \emph{vertex images}. 
    The following is the main result of this section.  
    
    \begin{theorem} \label{thm:main2}
        There exists a function $g_{\ref{thm:main2}}(k)=O(k^{9}\polylog(k))$ such that every
        graph $G$ with $f_2(G) > g_{\ref{thm:main2}}(k)$ contains a $\triangle_{k+2}$ minor.
        Moreover, every graph $G$ that contains a $\triangle_{k+2}$ minor has $f_2(G) > k$.
    \end{theorem}
    
    In order to prove the first part of Theorem~\ref{thm:main2}, we use the by now standard notion of \emph{treewidth} (see \cite{Diestel} for the definition). We let $\tw(G)$ denote the treewidth of a graph $G$.  As observed by Belk and Connelly \cite{BelkConnelly2007}, $f_2(G) \leq \tw(G)$ holds for all graphs $G$. Thus if $f_2(G) > c$, then $\tw(G) > c$. 
    
    By the grid theorem \cite{RS86}, there is a function $\gamma(k)$ such  that every graph $G$ with $\tw(G) \geq \gamma(k)$ contains $\Box_k$ as a minor. In fact, one can take $\gamma(k) = O(k^{9}\polylog(k))$ by very recent results \cite{CT19} (see  \cite{ChekuriChuzhoi16} for the original polynomial grid theorem). Furthermore, it is easy to check that $\Box_{2k+2}$ has a  $\triangle_{k+2}$ minor, for all $k \in \NN$. Figure~\ref{fig:triangular_grid} illustrates this for $k = 4$. Therefore, in Theorem~\ref{thm:main2}, we may take $g_{\ref{thm:main2}}(k)=\gamma(2k+2)$. This proves the first part of the theorem. Notice that for all $r \in \NN$, $\triangle_r$ has $\Box_m$ as a subgraph, where $m = \lfloor \frac{r-1}{2} \rfloor$. Thus, excluding triangular grids is equivalent to excluding rectangular grids within a factor of $2$. 
    
    \begin{figure}[ht]
        \centering
        \begin{tikzpicture}[x=0.9cm, y=-0.9cm]
            \draw (5,0)--(0,0)--(0,5);
            \foreach \n in {1,...,5}{
            \draw[blue] (\n, 0)--(0, \n);
            \draw (\n, 5-\n)--(\n, 0);
            \draw (5-\n, \n)--(0, \n);
            }
            \foreach \n in {0,...,5}{
            \foreach \m in {0,...,\n}{
            \draw[red, fill=red] (\m, 5-\n) circle (2pt);
            }
            }
        \end{tikzpicture}
        \hspace{5mm}
        \begin{tikzpicture}[x=0.5cm, y=-0.5cm]
            \foreach \n in {0,...,9}{
            \draw[lightgray] (0, \n)--(9, \n);
            \draw[lightgray] (\n, 0)--(\n, 9);
            }
            \foreach \n in {1,...,4}{
            \draw[red, very thick] (2*\n-1, 0)--(2*\n, 0);
            \draw[red, very thick] (0, 2*\n-1)--(0, 2*\n);
            \foreach \m in {1,...,\n}{
            \draw[red, very thick] (2*\m-1, 10- 2*\n)--(2*\m, 10- 2*\n);
            \draw[red, very thick] (2*\m, 9-2*\n)--(2*\m, 10-2*\n);
            \draw[thick] (2*\m-2, 10-2*\n)--(2*\m-1, 10-2*\n);
            \draw[thick] (2*\m, 8-2*\n)--(2*\m, 9-2*\n);
            }
            }
            \foreach \n in {0,...,4}{
            \draw[thick] (2*\n, 0)--(2*\n+1, 0);
            \draw[thick] (0, 2*\n)--(0, 2*\n+1);
            \foreach \m in {0,...,\n}{
            \draw[blue, thick] (2*\m, 9-2*\n)--(2*\m+1, 9-2*\n);
            \draw[red, very thick] (2*\m+1, 8-2*\n)--(2*\m+1, 9-2*\n);
            }
            }
            \draw[red, fill=red] (0,0) circle (2pt);
            \draw[red, fill=red] (0,9) circle (2pt);
        \end{tikzpicture}
        \caption{On the left is $\triangle_6$. On the right is a $\triangle_6$-model in $\Box_{10}$. Vertex images are displayed in red, and edges between the vertex images in black or blue. \label{fig:triangular_grid}}
    \end{figure}
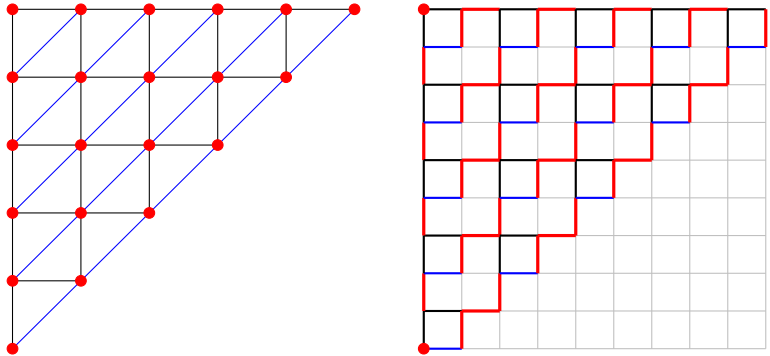  
    
    We now prove the second part of Theorem~\ref{thm:main2}, see Lemma~\ref{lem:triangular_grid} below. We remark that Eisenberg-Nagy, Laurent and Varvitsiotis~\cite{Eisenberg-NagyLV14} prove a similar result for a related invariant called \emph{extreme Gram dimension}. This is a variant of the \emph{Gram dimension} of a graph, that is studied and compared to the Euclidean dimension in Laurent and Varvitsiotis~\cite{LaurentV14}. The idea of considering a triangular grid instead of a rectangular one comes from~\cite{Eisenberg-NagyLV14}, and our induction-based proof is inspired by their proof. However, to our knowledge, the results of \cite{LaurentV14} and \cite{Eisenberg-NagyLV14} do not imply our next lemma.
    
\begin{lemma} \label{lem:triangular_grid}
    For all $r \in \NN$, $f_2(\triangle_r) \geq r-1$. 
\end{lemma}
\begin{proof}
Let $e_1, \ldots, e_r$ be the $r$ standard basis vectors in $\mathbb{R}^r$. We recursively define an embedding $\phi : V(\triangle_r) \to \mathbb{R}^r$ by $\phi(v_{1,j}) = e_j$ for all $j \in [r]$ and
$
\phi(v_{i,j}) = \frac{1}{2} \phi(v_{i-1,j-1}) + \frac{1}{2} \phi(v_{i-1,j})
$ 
for all $2 \leq i \leq j$. We define an $\ell_2$-realizable distance function $d : E(\triangle_r) \to \RR_+$ from the embedding $\phi$, by letting $d(vv') = ||\phi(v)-\phi(v')||_2$ for each $vv' \in E(\triangle_r)$. 

Now consider an arbitrary isometric embedding $\psi$ of $(\triangle_r,d)$ in some Euclidean space $\mathbb E$. By our choice of the distance function, $\psi(v_{i,j})$ is the midpoint of $\psi(v_{i-1,j-1})$ and $\psi(v_{i-1,j})$ for every $i \geq 2$. Hence, the whole embedding $\psi$ is entirely determined by the $r$ points $q_j = \psi(v_{1,j})$, and lies in the affine hull of $q_1$, \ldots, $q_r$. By applying an appropriate isometry, we may assume that $\mathbb{E} = \{x \in \RR^r \mid \sum_{i} x_i = 1\}$.
We claim that $||q_i-q_j||_2= \sqrt{2}$ for all distinct $i,j \in [r]$. Hence, these $r$ points are the vertices of a regular simplex, which implies $f_2(G,d) \geq r-1$.  

The proof is by induction on $r$. Since the statement is clear for $r = 2$, we may assume that $r \geq 3$. Observe that the induced subgraphs $\triangle_r - \{v_{i,r} \mid i \in [r]\}$ and $\triangle_r - \{v_{i,i} \mid i \in [r]\}$ are both isomorphic to $\triangle_{r-1}$. By the inductive hypothesis, this implies that $q_1$, \ldots, $q_{r-1}$ are equidistant, and $q_2$, \ldots, $q_r$ are equidistant. Thus, it remains to show $||q_1-q_r||_2=\sqrt{2}$.

Since $||q_i-q_j||_2=\sqrt{2}$ for all distinct $i,j \in [r-1]$, by applying an appropriate isometry we may assume that $q_k = e_k$ for all $k \in [r-1]$. 


 Let $x_1, \ldots, x_r \in \RR$ denote the coordinates of $q_r$ in $\RR^r$. The following constraints hold:

\begin{align}
\sum_{i} x_i &= 1\,,\\
\label {eq:squarenorm} \sum_{i} x^2_i &= 1 + 2 x_k \quad \forall 2 \leq k \leq r-1\,. 
\end{align}

The first constraint is due to the fact that $q_r \in \mathbb{E}$, and the second is equivalent to $||\psi(v_{1,r}) - \psi(v_{1,k})||_2^2 = ||\phi(v_{1,r}) - \phi(v_{1,k})||_2^2$ (for $2 \leq k \leq r-1$), which holds by induction. Notice that $x_2 = x_3 = \cdots = x_{r-1}$ follows from \eqref{eq:squarenorm}. Since $v_{r-1,r-1}v_{r-1,r}$ is an edge of $\triangle_r$, 
\begin{equation}
\label{eq:lastdist}
||\psi(v_{r-1,r-1})-\psi(v_{r-1,r})||_2^2 = ||\phi(v_{r-1,r-1})-\phi(v_{r-1,r})||_2^2\,.
\end{equation}

Since $\psi(v_{1,j}) = \phi(v_{1,j})$ for all $j \in [r-1]$,  $\psi(v_{i,j}) = \phi(v_{i,j})$ for all $i \leq j \leq r-1$. Hence, we can rewrite the left-hand side of \eqref{eq:lastdist} as
\begin{align*}
||\psi(v_{r-1,r-1})-\psi(v_{r-1,r})||_2^2
&= ||\phi(v_{r-1,r-1})-\psi(v_{r-1,r})||_2^2\\
&= ||(\phi(v_{r-1,r-1})-\phi(v_{r-1,r}))-(\psi(v_{r-1,r})-\phi(v_{r-1,r}))||_2^2\\
\end{align*}

Thus, \eqref{eq:lastdist} holds if and only if 
\begin{equation}
\label{eq:lastdist_bis}
||\psi(v_{r-1,r})-\phi(v_{r-1,r})||_2^2 = 2 \ip{\phi(v_{r-1,r})-\phi(v_{r-1,r})}{\psi(v_{r-1,r})-\phi(v_{r-1,r})}\,.
\end{equation}

By induction, we see that, for all $i \in [r-1]$,
$$
\psi(v_{i,r}) - \phi(v_{i,r}) = \frac{1}{2^{i-1}} (\psi(v_{1,r}) - \phi(v_{1,r}))
= \frac{1}{2^{i-1}} (q_r - e_r)\,.
$$
Using this, we can rewrite the left-hand side of \eqref{eq:lastdist_bis}:
\begin{align*}
||\psi(v_{r-1,r})-\phi(v_{r-1,r})||_2^2 
&= \left(\frac{1}{2^{r-2}}\right)^2 ||q_r - e_r||_2^2\\
&= \frac{1}{2^{2r-4}} (||q_r||_2^2 + ||e_r||_2^2 - 2 \ip{q_r}{e_r})\\
&= \frac{1}{2^{2r-4}} (1 - 2x_2 + 1 - 2 x_r)\,.
\end{align*}

Notice that, since $x_2 = x_3 = \ldots = x_{r-1}$, 
$$
q_r - e_r = x_2 \mathbf{1} + (x_1-x_2) e_1 + (x_r-x_2-1) e_r\,,
$$
where $\mathbf{1}$ is the all-ones vector. Also, an easy induction on $i$ shows that
$$
\ip{\phi(v_{i,i})}{e_1} = \frac{1}{2^{i-1}} = \ip{\phi(v_{i,r})}{e_r}\,,
$$
and thus
\begin{align*}
\ip{\phi(v_{i,i}) - \phi(v_{i,r})}{e_1} &= \frac{1}{2^{i-1}},\text{ and}\\
\ip{\phi(v_{i,i}) - \phi(v_{i,r})}{e_r} &= - \frac{1}{2^{i-1}}\,. 
\end{align*}

Now, we can rewrite the right-hand side of \eqref{eq:lastdist_bis} as
\begin{align*}
\frac{1}{2^{r-3}} &\ip{\phi(v_{r-1,r})-\phi(v_{r-1,r})}{q_r-e_r}\\
&= \frac{1}{2^{r-3}} \ip{\phi(v_{r-1,r})-\phi(v_{r-1,r})}{x_2 \mathbf{1} + (x_1-x_2) e_1 + (x_r-x_2-1) e_r}\\
&= \frac{1}{2^{r-3}} \left(0 + \frac{1}{2^{r-2}} (x_1-x_2) - \frac{1}{2^{r-2}} (x_r-x_2-1) \right).
\end{align*}

Hence, \eqref{eq:lastdist_bis} can be rewritten 
$$
\frac{1}{2^{2r-4}} (1-2x_2+1-2x_r) = \frac{1}{2^{r-3}} \left(\frac{1}{2^{r-2}} (x_1-x_2) - \frac{1}{2^{r-2}} (x_r-x_2-1) \right)
\iff x_2 = -x_1\,.
$$
Now, 
$$
||q_r - q_1||_2^2 = ||q_r - e_1||_2^2 = \sum_i x_i^2 + 1 - 2x_1 
= (1 - 2x_2) + 1 - 2x_1 = (1 + 2x_1) + 1 - 2x_1 = 2 \,. \qedhere
$$
\end{proof}

It is easy to check that $\tw(\triangle_r) \leq r-1$ for all $r \geq 3$.  Thus, Lemma~\ref{lem:triangular_grid} implies that $f_2(\triangle_r)=r-1$ for all $r \geq 3$.  Moreover, since every planar graph is a minor of a sufficiently large triangular grid, Theorem~\ref{thm:main2} immediately yields the following corollary. 

 \begin{corollary}
    For all minor-closed classes of graphs $\mathcal{C}$, $f_2(\mathcal C) = \infty$ if and only if $\mathcal{C}$ contains all planar graphs.  
    \end{corollary}

    \section{Alternative view of $\ell_\infty$-dimension} \label{sec:alternative}
    
    In this section, we provide a more combinatorial definition of $\ell_\infty$-dimension.  The equivalence follows by considering potentials on a weighted auxilliary digraph.  
    
    Let $D$ be a digraph with edge weights $l: A(D)\to \RR$. A \emph{potential on $(D,l)$} is a function $p:V(D) \to \RR$ such that $p(w)-p(v)\leq l(v,w)$ for all arcs $(v,w)\in A(D)$.
    
    Now consider a metric graph $(G,d)$. Let $(D,l)$ be the (edge)-weighted digraph obtained from $(G,d)$ by bidirecting all edges and setting $l(v,w) = l(w,v) = d(vw)$ for all edges $vw\in E(G)$. Note that $p:V(D)\to \RR$ is a potential on $(D,l)$ if and only if $\lvert p(w)-p(v)\rvert \leq d(vw)$ for all edges $vw\in E(G)$.
    
    For convenience, we let $D(G)$ and $l(d)$ denote the digraph and edge weights defined above, respectively. Thus the weighted digraph $(D,l)$ we are considering can also be denoted $(D(G),l(d))$ when more precision is required.
    
    Recall that distances in $\ell_\infty^k$ are given by $d_\infty(x,y) = \max_{i \in [k]} \lvert x_i-y_i \rvert$. Hence $d_\infty(x,y) = \delta$ if and only if $|x_i - y_i| \leq \delta$ for all $i \in [k]$ and there exists some index $j \in [k]$ for which $\lvert x_j - y_j\rvert = \delta$. Therefore, $(G,d)$ has an isometric embedding $\phi$ in $\ell_\infty^k$ if and only if there exist $k$ potentials $p_i : V(G) \to \RR$ on $(D,l)$ such that for each edge $vw$ there is at least one index $j \in [k]$ with $\lvert p_j(w)-p_j(v)\rvert = d(vw)$. This can be seen by taking $p_i(v)$ to be the $i$-th coordinate of $\phi(v)$, for all $i \in [k]$ and $v \in V(G)$.
    
    We say that a set of arcs $F\subseteq A(D)$ is a \emph{flat set} of $(G,d)$ if there exists a potential $p : V \to \RR$ on $(D,l)$ such that $p(w) - p(v) = -d(vw) \iff p(v) - p(w) = d(vw)$ for all arcs $(v,w) \in F$. Given a set $F \subseteq A(D)$, consider the modified edge weights $l_F : A(D) \to \RR$ such that 
    $$
    l_F(v,w) = \begin{cases}
         d(vw) & \text{if } (v,w) \notin F\\
        -d(vw) & \text{if } (v,w) \in F\,.
    \end{cases}
    $$
    When necessary, we denote these edge weights by $l_F(d)$. Then $F \subseteq A(D)$ is a flat set of $(G,d)$ if and only if $(D,l_F) = (D(G),l_F(d))$ admits a potential. By the well-known characterization of the existence of potentials, this is equivalent to the non-existence of a negative weight directed cycle in $(D,l_F)$. That is, $F \subseteq A(D)$ is a flat set if and only if $(D,l_F)$ does not contain a negative directed cycle. In proofs, we will often use the notation $\ddir{G}{d}{F}$ to denote $(D(G),l_F(d))$. Notice that $F$ is a flat set if and only if $F' = \{(w,v) \mid (v,w) \in F\}$ is a flat set. 
    
    We say that a flat set $F \subseteq A(D)$ \emph{covers} an edge $vw \in E(G)$ if $F$ contains $(v,w)$ or $(w,v)$. A \emph{flat covering} of $(G,d)$ is a collection $\mathcal{F} = \{F_1, \ldots, F_k\}$ of flat sets such that every edge $vw \in E(G)$ is covered by at least one $F_i$. Then, $(G,d)$ has an isometric embedding into $\ell^k_\infty$ if and only if $(G,d)$ has a flat covering of size at most $k$. To construct an embedding given a flat covering, we pick a potential $p_i$ on $\ddir{G}{d}{F_i}$ for each flat set $F_i$, and use these potentials to define the embedding coordinatewise. That is, each potential $p_i$ associated to $F_i$ gives us the $i$-th coordinate of the vertices in the embedding. Notice that the potentials respect the maximum differences given by the distance function $d$. Furthermore, because each edge is covered by some potential, the vertices of this edge are at exact distance in the corresponding coordinate. Hence we get an embedding of $(G,d)$.
    For the other direction, it is sufficient to realize that each coordinate of an embedding defines a potential. Furthermore, for each edge at least one of the potentials defined by the coordinates is such that the distance between the vertices is attained with equality, that is the edge is covered by this potential. Thus, the coordinates define a flat covering of size $k$.
    
    In our terminology, the $\ell_\infty$-dimension $f_\infty(G)$ is the least integer $k$ such that for each distance function $d$, the metric graph $(G,d)$ has a flat covering of size at most $k$.

    \section{Metric tools} \label{sec:metrictools}
    
    In this section, we present several general results related to distance functions and flat coverings. 
    
    Given a vertex $v$ of a graph $G$, we let $N(v) = \{w \in V(G) \mid vw \in E(G)\}$ denote the 
    neighborhood of $v$ in $G$. 
        
    \begin{lemma}  \label{lem:flatstar} 
        Let $(G,d)$ be a metric graph and let $v \in V(G)$. The set $F = \{(v,w) \mid w \in N(v)\}$ is a flat set of $(G,d)$. 
    \end{lemma}
    \begin{proof}
        Let $C$ be an arbitrary directed cycle in $\ddir{G}{d}{F}$.
        The cycle $C$ uses at most one arc of $F$. Thus at most one arc of $C$ has negative weight in $\ddir{G}{d}{F}$, and all other arcs of $C$ have non-negative weight.  Since $d$ is a distance function, it follows that $C$ has non-negative weight in $\ddir{G}{d}{F}$.  Thus, $F$ is a flat set of $(G,d)$, as required.
    \end{proof}
    
    A \emph{vertex cover} of a graph $G$ is a set of vertices $X \subseteq V(G)$ such that every edge of $G$ is incident with some vertex in $X$.  The \emph{vertex cover number} of $G$, denoted $\tau(G)$, is the size of a smallest vertex cover of $G$.  
    By Lemma~\ref{lem:flatstar}, $f_\infty(G)$ is at most the vertex cover number of $G$.
    
    \begin{lemma}[\cite{FHJV17}, Lemma 9] \label{lem:vertexcover}
    For every graph $G$, $f_\infty(G) \leq \tau(G)$.
    \end{lemma}
    
    Clearly, if $d$ is a distance function on $G$, and $H$ is a subgraph of $G$, then the restriction of $d$ to $E(H)$ is a distance function on $H$. We denote it by $d|_H$. Conversely, sometimes we can define a distance function on a graph from distance functions on certain subgraphs, see Lemma~\ref{lem:distancefunction} below.
    
    A \emph{$k$-sum} is a graph $G$ obtained by gluing two graphs $G_1$ and $G_2$ along a common clique $K$ of size $k$ and then possibly deleting some edges of $K$. We use the following notation for $1$-sums and $2$-sums. We write $G = G_1 +_v G_2$ if $G = G_1 \cup G_2$ with $V(G_1) \cap V(G_2) = \{v\}$. Now let $e = vw$ be an edge. We write $G = G_1\oplus_{e} G_2$ if $G = G_1 \cup G_2$ with $V(G_1) \cap V(G_2) = \{v,w\}$ and $e \in E(G_1) \cap E(G_2)$. Also, we denote by $G_1 +_{e} G_2$ the graph $G_1 \oplus_{e} G_2$ minus the edge $e$. 
    
    \begin{lemma}\label{lem:distancefunction}
        Let $G = G_1 \oplus_f G_2$.
        For $i \in [2]$, let $d_i$ be a distance function on $G_i$. If $d_1(f) = d_2(f)$, then the function $d : E(G) \to \RR_{\ge 0}$ defined by $d(e) = d_i(e)$ if $e\in E(G_i)$ is a distance function on $G$. 
    \end{lemma}
    \begin{proof}
        Let $vw$ be any edge of $G$. Without loss of generality, we may suppose $vw \in E(G_1)$. Let $P$ be a $v$--$w$ path in $G$. 
        If $P$ is contained in $G_1$ then $d(P) = d_1(P) \geq d_1(vw) =  d(vw)$. Otherwise, $P$ uses both ends of $f$ and we may decompose $P$ into a path $P_1$ from $v$ to an end of $f$ with $E(P_1) \subseteq E(G_1)$, a path $P_2$ between the two ends of $f$ with $E(P_2) \subseteq E(G_2)$ and a path $P'_1$ from the other end of $f$ to $w$ with $E(P'_1) \subseteq E(G_1)$. Then we get $d(P) = d(P_1) + d(P_2) + d(P'_1) \geq d(P_1) + d(f) + d(P'_1) \geq d(vw)$, where the first inequality uses that $d_2$ is a distance function, and the second inequality uses that $d_1$ is a distance function.
    \end{proof}
    
    Similarly, every subset of a flat set is flat, and if $F$ is a flat set of $(G,d)$, then $F$ is also a flat set of $(H, d|_H)$, for all subgraphs $H$ of $G$ with $F \subseteq A(D(H))$. The following lemma gives conditions under which a flat set of a subgraph is a flat set of the entire graph.  
    
    \begin{lemma}\label{lem:flatset}
        Let $G$ be a graph obtained by gluing two graphs $G_1$ and $G_2$ along a common clique $K$. Let $d$ be a distance function on $G$ and $d_i = d|_{G_i}$ its restriction to $G_i$, where $i \in [2]$. If $F$ is a flat set of $(G_j, d_j)$ for some $j \in [2]$, then $F$ is also a flat set of $(G,d)$. Conversely, if $F$ is a flat set of $(G,d)$ then $F_i = F \cap A(D(G_i))$ is a flat set of $(G_i,d_i)$ for all $i \in [2]$. 
    \end{lemma}
    \begin{proof}
        For the first part, it suffices to show that $\ddir{G}{d}{F}$ does not contain a negative weight directed cycle. 
        Let $C$ be a minimum weight directed cycle in $\ddir{G}{d}{F}$ such that $V(C)$ is inclusion-wise minimal. We may assume that $C$ contains some arc of $F$, since otherwise $C$ is disjoint from $F$ and has non-negative weight. Thus $C$ intersects $A(D(G_j))$.
        
        We claim that $C$ must be fully contained in $D(G_j)$. Otherwise, $C$ contains a directed path $P$ from $v$ to $w$, where $v, w \in K$, that is internally disjoint from $D(G_j)$. By replacing $P$ with the arc $(v,w)$ we obtain a new directed cycle $C'$ in $\ddir{G}{d}{F}$ whose weight is at most that of $C$ and such that $V(C') \subsetneq V(C)$, a contradiction.
        
        Since $C$ is contained in $D(G_j)$ and $F$ is a flat set of $(G_j,d_j)$, $C$ has non-negative weight in $\ddir{G_j}{d_j}{F}$ and thus in $\ddir{G}{d}{F}$. 
        
        For the second part, notice that $F_i$ is a flat set of $(G,d)$ because $F_i \subseteq F$ and $F$ is a flat set of $(G,d)$. Since $G_i$ is a subgraph of $G$, $F_i$ is also clearly a flat set of $(G_i, d_i)$. 
    \end{proof}

    \begin{lemma}\label{lem:negPaths2directions}
        Let $F$ be a flat set of a metric graph $(G, d)$ and $u$ and $v$ be vertices of $G$. Let $P_1$ be a directed path from $u$ to $v$ and let $P_2$ be a directed path from $v$ to $u$. Then at least one of $P_1$ and $P_2$ has non-negative weight in $\ddir{G}{d}{F}$.
    \end{lemma}
    \begin{proof}
        Consider the directed closed walk obtained by concatenating $P_1$ and $P_2$.  This directed closed walk decomposes into directed cycles. If $P_1$ and $P_2$ both have negative weight in $\ddir{G}{d}{F}$, then at least one of these directed cycles has negative weight in $\ddir{G}{d}{F}$. But this contradicts the fact that $F$ is a flat set. 
    \end{proof}
    
    In \cite{FHJV17}, the following result is proved. 
    
    \begin{lemma}[\cite{FHJV17}]\label{lem:suppressDeg2}
    For every graph $G$ with $f_\infty(G) \geq 2$ and every edge $e\in E(G)$,  
    \[
    f_\infty(G) = f_\infty(G +_e K_3) = f_\infty(G \oplus_e K_3).
    \]
    \end{lemma}

    Hence, deleting a degree-$2$ vertex $v$ and adding a new edge between the neighbors of $v$ (if there was none) does not change $f_\infty(G)$, provided the resulting graph is not a forest. We will refer to this operation as \emph{suppressing a degree-$2$ vertex}. It follows that for all $k \geq 2$, the excluded minors for $f_\infty(G) \leq k$ have minimum degree at least $3$. 

    We will use the following bounds on $f_\infty(G)$ when $G$ is a $k$-sum.
    
    \begin{lemma}\label{lem:gluing} 
        For all graphs $G_1$ and $G_2$ (for which the $k$-sums below exist),
        \begin{equation}
        \label{eq:1-sum}
        f_\infty(G_1+_v G_2) = \max\{f_\infty(G_1), f_\infty(G_2)\}
        \end{equation}
        and
        \begin{equation}
        \label{eq:2-sum}
            f_\infty(G_1+_{vw}G_2)\leq f_\infty(G_1\oplus_{vw}G_2) \leq f_\infty(G_1) + f_\infty(G_2)-1\,.
        \end{equation}
        Moreover,
        \begin{equation}
        \label{eq:k-sum}
        f_\infty(G)\leq f_\infty(G_1) + f_\infty(G_2)
        \end{equation}
        whenever $G$ is a $k$-sum of $G_1$ and $G_2$.
    \end{lemma}
    \begin{proof}
        Observe that \eqref{eq:k-sum} follows from Lemma~\ref{lem:flatset}.
        Next, we prove \eqref{eq:1-sum}. Let $k = \max \{f_\infty(G_1), f_\infty(G_2)\}$. Since $f_\infty$ is minor-monotone, it is clear that $f_\infty(G_1+_v G_2)$ is at least $k$. The next paragraph proves that it is at most $k$.  
        
        Let $d$ be a distance function on $G_1+_vG_2$. For $i \in [2]$, let $d_i = d|_{G_i}$. Then $d_i$ is a distance function on $G_i$. For $i \in [2]$, let $\phi_i$ be any isometric embedding of $(G_i,d_i)$ into $\ell_\infty^k$. After translating one of the embeddings if necessary, we may assume that $\phi_1(v) = \phi_2(v)$. It is easy to see that the function $\phi : V(G_1 +_v G_2) \to \RR^{k}$ obtained by setting $\phi(w) = \phi_i(w)$ if $w \in V(G_i)$ for $i \in [2]$ is an isometric embedding of $(G_1+_vG_2,d)$ into $\ell_\infty^{k}$. 
            
        Finally, we prove \eqref{eq:2-sum}. The first inequality in \eqref{eq:2-sum} is trivial since $G_1+_{vw}G_2$ is a minor of $G_1\oplus_{vw}G_2$. To prove the second inequality, consider a distance function $d$ on $G$. For $i \in [2]$, let $d_i = d|_{G_i}$ be the corresponding distance function of $G_i$.
        
        Let $\mathcal{F}_i$ be a minimum size flat covering of $(G_i,d_i)$. By Lemma~\ref{lem:flatset}, each set in $\mathcal{F}_1 \cup \mathcal{F}_2$ is flat in $(G,d)$. For $i \in [2]$, let $F_i$ be a flat set in $\mathcal{F}_i$ covering $vw$. By reversing arcs if necessary, we may assume both $F_1$ and $F_2$ contain $(v,w)$. We may also 
        assume that neither $F_1$ nor $F_2$ contains $(w,v)$, since otherwise we get $d(vw) = 0$. In this case, we can contract the edge $vw$ and use \eqref{eq:1-sum}.
    
        We claim that $F_1 \cup F_2$ is a flat set of $(G,d)$. Let $C$ be an arbitrary directed cycle in $\ddir{G}{d}{F_1\cup F_2}$. For $i \in [2]$, let $C_i$ be the directed cycle obtained by restricting $C$ to $D(G_i)$ and possibly adding $(v,w)$ or $(w,v)$ (possibly $C_i = \emptyset$). Let $l = l_{F_1 \cup F_2}(d)$ be the edge weights on $\ddir{G}{d}{F_1\cup F_2}$ and $l_i = l_{F_i}(d_i)$ be the edge weights on $\ddir{G_i}{d_i}{F_i}$. Notice that $l(v,w) = -d(vw)$ and $l(w,v) = d(vw)$. Then $l(C) = l(C_1)+l(C_2) = l_1(C_1)+l_2(C_2)\geq 0+0 = 0$ since $l_i$ is the restriction of $l$ to $A(D(G_i))$ and $F_i$ is flat in $(G_i, d_i)$. Thus, $C$ has non-negative weight and $F_1 \cup F_2$ is a flat set of $(G, d)$, as claimed. 
    
        Now $\mathcal{F} = \{F_1 \cup F_2\} \cup (\mathcal{F}_1 \cup \mathcal{F}_2) \delete \{F_1,F_2\}$ is a flat covering of $(G,d)$ of size at most $|\mathcal{F}_1| + |\mathcal{F}_2| - 1 \le f_\infty(G_1) + f_\infty(G_2) -1$. 
    \end{proof}
    
    Let $(G, d)$ be a metric graph.  We say that two edges $e$ and $f$ of $G$ are \emph{incompatible}, if there is no flat set of $(G, d)$ that covers both of them. 
    Note that two such edges are necessarily independent, by Lemma~\ref{lem:flatstar}. 
    A simple but crucial observation is that if $(G,d)$ contains $k$ pairwise incompatible edges, then $f_\infty(G) \ge k$. The following lemma provides sufficient conditions under which two edges are incompatible.  
    
    \begin{lemma}\label{lem:incompatible}
        Let $(G,d)$ be a metric graph and let $v_1v_2, w_1w_2$ be two independent edges of $G$. 
        If for all $i,j \in [2]$, there exist paths $P_{i,j}$ between $v_i$ and $w_j$ such that $d(P_{1,1})+d(P_{2,2})<d(v_1v_2)+d(w_1w_2)$ and $d(P_{1,2})+d(P_{2,1})< d(v_1v_2)+d(w_1w_2)$, then 
        $v_1v_2$ and $w_1w_2$ are incompatible. 
    \end{lemma}
    \begin{proof} 
        Suppose $F$ is a flat set covering $v_1v_2$ and $w_1w_2$. 
        Suppose first $(v_1,v_2), (w_1, w_2)\in F$. 
        Consider the closed directed walk $W$ that starts at $v_1$, takes $(v_1,v_2)$, follows $P_{2,1}$ to $w_1$, takes $(w_1,w_2)$ and then follows $P_{1,2}$ back to $v_1$. 
        The weight of $W$ in $\ddir{G}{d}{F}$ is at most $d(P_{1,2})+d(P_{2,1}) - d(v_1v_2)-d(w_1w_2) < 0$. Thus, $W$ contains a negative weight directed cycle, which contradicts that $F$ is flat.  
        
        By symmetry the remaining case is $(v_1,v_2), (w_2,w_1)\in F$. 
        Again it is easy to find a negative weight directed walk $W$ in $\ddir{G}{d}{F}$ using the fact that $d(P_{1,1})+d(P_{2,2})<d(v_1v_2)+d(w_1w_2)$.
        Hence, $F$ cannot simultaneously cover the edges $v_1v_2$ and $w_1w_2$, as claimed.
    \end{proof}
    
    Finally, we also need the fact that $f_\infty(K_4) = 2$. 
    
    \begin{lemma}[\cite{WITSENHAUSEN1986}, 4.2] \label{lem:K4}
        $f_\infty(K_4) = 2$. 
    \end{lemma}
    
    In order to illustrate the concepts introduced in the last two sections, we briefly describe a polynomial reduction from computing the chromatic number of a graph $H$ to computing $f_\infty(G,d)$ given a metric graph $(G,d)$. 
    This proves that the latter problem is \np-hard. We remark that there is a different reduction using the \textsc{Partition} problem which shows that the problem of deciding if $f_\infty(G,d) \leq 1$ given a metric graph $(G, d)$ is \np-complete (see~\cite{Saxe79}).
    
    Let $H$ be a graph. We construct a metric graph $(G,d)$ by replacing each vertex $v \in V(H)$ by two adjacent vertices $v_1, v_2 \in V(G)$, and each edge $vw \in E(H)$ by a $K_{2,2}$ in $G$ with edge set $\{v_i w_j \mid i \in [2],\ j \in [2]\}$. 
    The distance function $d$ is defined by $d(v_1v_2) = 2$ for all $v \in V(H)$ and $d(v_iw_j) = 1$ for all $vw \in E(H)$, $i \in [2]$ and $j \in [2]$. We claim that $f_\infty(G,d) = \chi(H)$.
    
    To see that $f_\infty(G,d) \ge \chi(H)$, notice that edges $v_1v_2$ and $w_1w_2$ are incompatible whenever $vw \in E(H)$. Thus every size-$k$ flat covering of $(G,d)$ gives a $k$-coloring of $H$.
    
    Finally, $f_\infty(G,d) \le \chi(H)$, since for every stable set $S$ in $G$, $\{(v_1,v_2) \mid v \in S\} \cup \{(u_i,v_1) \mid i \in [2],\ uv \in E(H),\ v \in S\} \cup \{(v_2,w_j) \mid j \in [2],\ vw \in E(H),\ v \in S\}$ is a flat set of $(G,d)$. Hence, every
    $k$-coloring of $H$ gives a size-$k$ flat covering of $(G,d)$.

    \section{Certificates of large $\ell_\infty$-dimension} \label{sec:certificates}

    In this section, we show that if $H \in \kall^k = \{\ks_k, \kp_k, \kf_k, \kn_k\}$, then $f_\infty(H) > k$.  It follows that if a graph $G$ contains a $\kall^k$ minor, then $f_\infty(G) > k$.  Therefore, the existence of one of these four minors is a certificate that $f_\infty (G) > k $. Conversely, our main theorem shows that if $f_\infty(G) \geq g_{\ref{thm:main}}(k)$, then $G$ necessarily contains one of these four minors.  
    We also prove that $\ks_k, \kp_k$, and $\kf_k$ are excluded minors for the property $f_\infty(G) \leq k$, that is, all their proper minors have $\ell_\infty$-dimension at most $k$.  
    
    We begin by proving that for each $H \in \{\ks_k, \kp_k,\kf_k\}$, $f_\infty(H) = k+1$. We first prove the upper bound.   
    
    \begin{lemma} \label{lem:upperbound}
        For all $k \in \NN$ and all $H \in \{\ks_k, \kp_k, \kf_k\}$, $f_\infty(H) \leq k+1$.
    \end{lemma}
    \begin{proof}
        We proceed by induction on $k$.  The base case follows by Lemma~\ref{lem:K4}, since $\ks_1 = \kp_1 = \kf_1 = K_4$. Next note that $\ks_k = \ks_{k-1}+_e K_4,\kp_k = \kp_{k-1}+_e K_4$, and $\kf_k = \kf_{k-1}+_e K_4$.  Therefore, we are done by induction and Lemmas~\ref{lem:gluing} and \ref{lem:K4}. 
    \end{proof}
		
	\begin{theorem}\label{thm:K4Star}
		For all $k \in \NN$, $f_\infty(\ks_k) = k+1$.
	\end{theorem}
	\begin{proof}
		By Lemma~\ref{lem:upperbound}, it suffices to show $f_\infty(\ks_k)\geq k+1$. Since $S_1 = K_4$, by Lemma~\ref{lem:K4}, we may assume $k \geq 2$. We now give a distance function $d$ on $\ks_k$, which is illustrated in Figure~\ref{fig:K4Star}, such that there are $k+1$ incompatible edges in $(\ks_k, d)$. 
		    
		Let $V(\ks_k) = \{v,w\}  \cup \{v_1, w_1, \dots, v_k, w_k\}$ where $v,w,v_i,w_i$ are the vertices of the $i$th copy of $K_4$. We define $d$ as follows:
		    
		\begin{align*}
		    d(vv_1) = d(ww_1)& = 4k\,,\\
			d(vv_i) = d(ww_i) &= 2(k+i-1) &\text{ for all } i \in [k],\ i \neq 1 \,,\\
			d(wv_i) = d(vw_i) &= k+i-1 &\text{ for all } i \in [k]\,,\\
			d(v_iw_i) &= 3(k+i-1) &\text{ for all } i \in [k]\,.\\
		\end{align*}

        	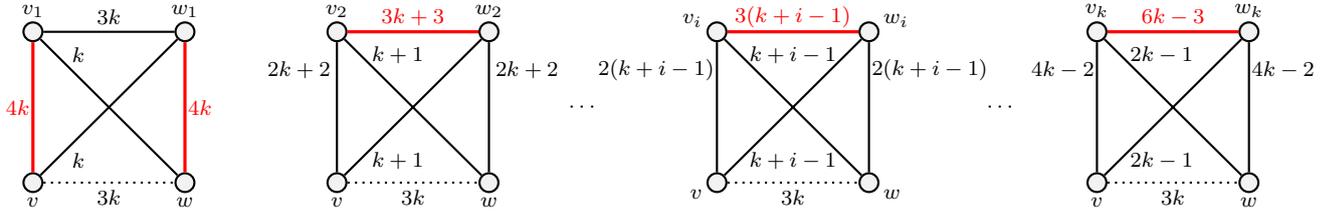
\begin{figure}[hb]
			\centering
			\begin{tikzpicture}[scale=1, inner sep=2.5pt]
                \begin{scriptsize}
                    \tikzstyle{vtx}=[circle,draw,thick,fill=black!5]
                    \node[vtx] (a1) at (0,0) {};
                    \node[vtx] (b1) at (2,0) {};
                    \node[vtx] (v1) at (0,2) {};
                    \node[vtx] (w1) at (2,2) {};
			        \node[below] at (0,-.1) {$v$};
			        \node[below] at (2,-.1) {$w$};
			        \node[above] at (0,2.1) {$v_1$};
			        \node[above] at (2,2.1) {$w_1$};
			        
			        \node[vtx] (a2) at (4,0) {};
			        \node[vtx] (b2) at (6,0) {};
			        \node[vtx] (v2) at (4,2) {};
			        \node[vtx] (w2) at (6,2) {};
			        \node[below] at (4,-.1) {$v$};
			        \node[below] at (6,-.1) {$w$};
			        \node[above] at (4,2.1) {$v_2$};
			        \node[above] at (6,2.1) {$w_2$};
			        
			        \node[vtx] (a3) at (9,0) {};
			        \node[vtx] (b3) at (11,0) {};
			        \node[vtx] (v3) at (9,2) {};
			        \node[vtx] (w3) at (11,2) {};
			        \node[below,left] at (8.9,-.15) {$v$};
			        \node[below,right] at (11.1,-.15) {$w$};
			        \node[above,left] at (8.9,2.15) {$v_i$};
			        \node[above,right] at (11.1,2.15) {$w_i$};
			        
			        \node[vtx] (a4) at (14,0) {};
			        \node[vtx] (b4) at (16,0) {};
			        \node[vtx] (v4) at (14,2) {};
			        \node[vtx] (w4) at (16,2) {};
			        \node[below] at (14,-.1) {$v$};
			        \node[below] at (16,-.1) {$w$};
			        \node[above] at (14,2.1) {$v_k$};
			        \node[above] at (16,2.1) {$w_k$};
			        
			        \node at (7.25,1) {$\cdots$};
			        \node at (12.75,1) {$\cdots$};
			        
			        \node at (1,-.2) {$3k$};
			        \node at (1,2.2) {$3k$};
			        \node[red] at (-0.2,1) {$4k$};
			        \node[red] at (2.2,1) {$4k$};
			        \node at (0.6,1.7) {$k$};
			        \node at (0.6,0.3) {$k$};
			        
			        \node at (5,-.2) {$3k$};
			        \node[red] at (5,2.2) {$3k+3$};
			        \node at (3.5,1.5) {$2k+2$};
			        \node at (6.5,1.5) {$2k+2$};
			        \node at (4.8,1.7) {$k+1$};
			        \node at (4.8,0.3) {$k+1$};
			        
			        \node at (10,-.2) {$3k$};
			        \node[red] at (10,2.2) {$3(k+i-1)$};
			        \node at (8.2,1.5) {$2(k+i-1)$};
			        \node at (11.8,1.5) {$2(k+i-1)$};
			        \node at (10.0,1.7) {$k+i-1$};
			        \node at (10.0,0.3) {$k+i-1$};
			        
			        \node at (15,-.2) {$3k$};
			        \node[red] at (15,2.2) {$6k-3$};
			        \node at (13.55,1.5) {$4k-2$};
			        \node at (16.45,1.5) {$4k-2$};
			        \node at (14.85,1.7) {$2k-1$};
			        \node at (14.85,0.3) {$2k-1$};
			    \end{scriptsize}
			    
			    \draw[thick] (v1)--(w1) (a1)--(w1) (b1)--(v1);
			    \draw[very thick, red] (a1)--(v1) (b1)--(w1);
			    \draw[thick, dotted] (a1)--(b1);
			    
			    \draw[thick] (a2)--(v2) (a2)--(w2) (b2)--(v2) (b2)--(w2);
			    \draw[very thick, red] (v2)--(w2);
			    \draw[thick, dotted] (a2)--(b2);
			    
			    \draw[thick] (a3)--(v3) (a3)--(w3) (b3)--(v3) (b3)--(w3);
			    \draw[very thick, red] (v3)--(w3);
			    \draw[thick, dotted] (a3)--(b3);
			    
			    \draw[thick] (a4)--(v4) (a4)--(w4) (b4)--(v4) (b4)--(w4);
			    \draw[very thick, red] (v4)--(w4);
			    \draw[thick, dotted] (a4)--(b4);
			\end{tikzpicture}
			\caption{\label{fig:K4Star}$(\ks_k,d)$ as in the proof of Theorem~\ref{thm:K4Star}. The red edges are pairwise incompatible. Vertices with the same label are identified.}
		\end{figure}

		First, we show that $d$ is  a distance function. For this, let $(G,d')$ be obtained from $(\ks_k,d)$ by adding the edge $vw$ of length $d'(vw) = 3k$. Observe that 
		\[
		G = K_4 \oplus_{vw} K_4 \oplus_{vw} \dots \oplus_{vw} K_4,
		\] 
		where $K_4$ appears $k$ times in the righthand side.  It is easy to see that the restriction of $d'$ to each $K_4$ subgraph of $G$ is a distance function. Therefore, by Lemma~\ref{lem:distancefunction}, $d'$ is a distance function on $G$. Since $d$ is a restriction of $d'$ to $\ks_k$ it follows that $d$ is a distance function on $\ks_k$.

		We now show that the $k+1$ edges $vv_1, ww_1, v_2w_2, v_3w_3, \dots, v_kw_k$ are pairwise incompatible. For this, we make repeated use of Lemma~\ref{lem:incompatible}.
			
		First, consider $vv_1$ and $ww_1$.  Observe that $d(vv_1)+d(ww_1) = 8k$. However, $d(vw_1)+d(wv_1) = 2k < 8k$ and $d(v_1w_1)+d(vv_2w) = 6k+3<8k$, since $k \geq 2$. By Lemma~\ref{lem:incompatible}, $vv_1$ and $ww_1$ are incompatible. 
			
		Next, consider $vv_1$ and $v_iw_i$ with $i\in \{2,\ldots, k\}$. Observe that $d(vv_1)+d(v_iw_i) =  7k+3i-3$.  However, 		$d(vv_i)+d(w_iwv_1) = 5k+2i-2 < 7k + 3i-3$ and $d(vw_i)+d(v_iwv_1) = 3k+2i-2 < 7k+3i-3$. Hence, by Lemma~\ref{lem:incompatible}, $vv_1$ and $v_iw_i$ are incompatible. 
		
		By symmetry, $ww_1$ and $v_iw_i$ are also incompatible for each $i \in \{2,\ldots, k\}$.
			
		Finally, consider $v_iw_i$ and $v_jw_j$ for $2 \leq i< j \leq k$.  Observe that $d(v_iw_i)+d(v_jw_j) = 6k+3i+3j-6$.  However, $d(v_iwv_j)+d(w_ivw_j) = 4k+2i+2j-4 < 6k+3i+3j-6$, and $d(v_ivw_j)+d(w_iwv_j) = 6k+4i+2j-6 < 6k+3i+3j-6$ since $i < j$.  Hence, by Lemma~\ref{lem:incompatible}, $v_iw_i$ and $v_jw_j$ are incompatible, which completes the proof. 
	\end{proof}

	\begin{theorem}\label{thm:K4Path}
		For all $k \in \NN$, $f_\infty(\kp_k) =  k+1$. 
	\end{theorem}
	\begin{proof}
        Again, $f_\infty(\kp_k)\leq k+1$ follows from Lemma~\ref{lem:upperbound}.  We label the vertices of the topmost path of $\kp_k$ as $v_0, v_1, \dots , v_{k}$ and the vertices of the bottommost path of $\kp_k$ as $w_0, w_1, \dots , w_{k}$. Thus $V(\kp_k) = \{v_0,v_1,\ldots,v_{k}\} \cup \{w_0,w_1,\ldots,w_{k}\}$ and $E(\kp_k) = \{v_0w_0,v_{k}w_{k}\} \cup \{v_{i-1}v_{i},v_{i-1}w_{i},w_{i-1}v_{i},w_{i-1}w_{i} \mid i \in [k]\}$.
		For the lower bound, consider the following distance function $d$, which is illustrated in Figure~\ref{fig:nK4Path} (we take $i \in [k]$):
		\begin{align*}
		    d(v_0w_0) = d(v_{k}w_{k}) &= 2^k\,, \\
		    d(v_{i-1}v_{i}) = d(w_{i-1}w_{i}) &= 2^k+1 &\text{ if $ i \equiv 1 \pmod 2$}\,,\\
		    d(v_{i-1}v_{i}) = d(w_{i-1}w_{i}) &= 2^k-1 &\text{ if $ i \equiv 2 \pmod 4$}\,,\\
            d(v_{i-1}v_{i}) = d(w_{i-1}w_{i}) &= 2^k - 2^{1+i/2} &\text{ if $i \equiv 0 \pmod 4$}\,,\\
		    d(v_{i-1}w_{i}) = d(w_{i-1}v_{i}) &= 2^{1+i/2} &\text{ if $i \equiv 0 \pmod 4$}\,,\\
            d(v_{i-1}w_{i}) = d(w_{i-1}v_{i}) &= 1 \, &\text{ if $i \not\equiv 0 \pmod 4$}\,.
        \end{align*}
		    
		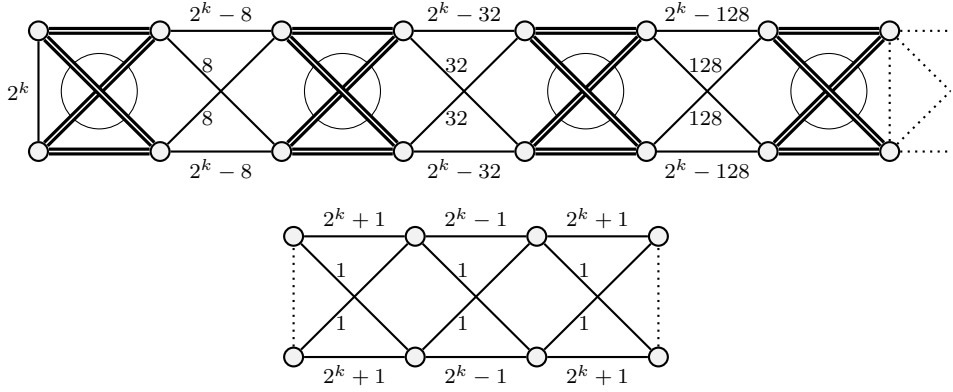
\begin{figure}[ht]
		    \centering
		    \begin{tikzpicture}[x=0.8cm, y=0.8cm, inner sep=2.5pt]
		        \tikzstyle{vtx}=[circle,draw,thick,fill=black!5]
		        \begin{scriptsize}
		            \node[vtx] (v0) at (0,0){};
		            \node[vtx] (v1) at (0,2){};
		            \node[vtx] (v2) at (2,0){};
		            \node[vtx] (v3) at (2,2){};
		            \node[vtx] (v4) at (4,0){};
		            \node[vtx] (v5) at (4,2){};
		            \node[vtx] (v6) at (6,0){};
		            \node[vtx] (v7) at (6,2){};
		            \node[vtx] (v8) at (8,0){};
		            \node[vtx] (v9) at (8,2){};
		            \node[vtx] (v10) at (10,0){};
		            \node[vtx] (v11) at (10,2){};
		            \node[vtx] (v12) at (12,0){};
		            \node[vtx] (v13) at (12,2){};
		            \node[vtx] (v14) at (14,0){};
		            \node[vtx] (v15) at (14,2){};
		          
		            \draw (1,1) circle(0.5cm);
		            \draw (5,1) circle(0.5cm);
		            \draw (9,1) circle(0.5cm);
		            \draw (13,1) circle(0.5cm);
		   
		            \draw[thick] (v0)--(v1) node[midway, left]{$2^k$};
		            \draw[very thick, double] (v0)--(v2);
		            \draw[very thick, double] (v0)--(v3);
		            \draw[very thick, double] (v1)--(v2);
		            \draw[very thick, double] (v1)--(v3);
		            \draw[thick] (v2)--(v4) node[midway, below]{$2^k- 8$};
		            \draw[thick] (v2)--(v5) node[near start, right]{$8$};
		            \draw[thick] (v3)--(v4) node[near start, right]{$8$};
		            \draw[thick] (v3)--(v5) node[midway, above]{$2^k- 8$};
		            \draw[very thick, double] (v4)--(v6);
		            \draw[very thick, double] (v4)--(v7);
		            \draw[very thick, double] (v5)--(v6);
		            \draw[very thick, double] (v5)--(v7);
		            \draw[thick] (v6)--(v8) node[midway, below]{$2^k- 32$};
		            \draw[thick] (v6)--(v9) node[near start, right]{$32$};
		            \draw[thick] (v7)--(v8) node[near start, right]{$32$};
		            \draw[thick] (v7)--(v9) node[midway, above]{$2^k- 32$};
		            \draw[very thick, double] (v8)--(v10);
		            \draw[very thick, double] (v8)--(v11);
		            \draw[very thick, double] (v9)--(v10);
		            \draw[very thick, double] (v9)--(v11);
		            \draw[thick] (v10)--(v12) node[midway, left, below]{$2^k- 128$};
		            \draw[thick] (v10)--(v13) node[near start, right]{$128$};
		            \draw[thick] (v11)--(v12) node[near start, right]{$128$};
		            \draw[thick] (v11)--(v13) node[midway, above]{$2^k- 128$};
		            \draw[very thick, double] (v12)--(v14);
		            \draw[very thick, double] (v12)--(v15);
		            \draw[very thick, double] (v13)--(v14);
		            \draw[very thick, double] (v13)--(v15);
		            \draw[thick, dotted] (v14)--(v15);
		            \draw[thick, dotted] (v14)--(15,0);
		            \draw[thick, dotted] (v14)--(15,1);
		            \draw[thick, dotted] (v15)--(15,1);
		            \draw[thick, dotted] (v15)--(15,2);
		        \end{scriptsize}
		  \end{tikzpicture}\medskip
		        
		  \begin{tikzpicture}[x=0.8cm, y=0.8cm, inner sep=2.5pt]
		        \tikzstyle{vtx}=[circle,draw,thick,fill=black!5]
		        \begin{scriptsize}
		            \node[vtx] (v0) at (0,0){};
		            \node[vtx] (v1) at (0,2){};
		            \node[vtx] (v2) at (2,0){};
		            \node[vtx] (v3) at (2,2){};
		            \node[vtx] (v4) at (4,0){};
		            \node[vtx] (v5) at (4,2){};
		            \node[vtx] (v6) at (6,0){};
		            \node[vtx] (v7) at (6,2){};
		            
		            \draw[thick] (v0)--(v2) node[midway, below]{$2^k+1$};
		            \draw[thick] (v0)--(v3) node[near start, right]{$1$};
		            \draw[thick] (v1)--(v2) node[near start, right]{$1$};
		            \draw[thick] (v1)--(v3)node[midway, above]{$2^k+1$};
		            \draw[thick] (v2)--(v4) node[midway, below]{$2^k-1$};
		            \draw[thick] (v2)--(v5) node[near start, right]{$1$};
		            \draw[thick] (v3)--(v4) node[near start, right]{$1$};
		            \draw[thick] (v3)--(v5) node[midway, above]{$2^k-1$};
		            \draw[thick] (v4)--(v6) node[midway, below]{$2^k+1$};
		            \draw[thick] (v4)--(v7) node[near start, right]{$1$};
		            \draw[thick] (v5)--(v6) node[near start, right]{$1$};
		            \draw[thick] (v5)--(v7)node[midway, above]{$2^k+1$};
		            \draw[thick, dotted] (v6)--(v7);
		            \draw[thick, dotted] (v0)--(v1);
            \end{scriptsize}
        \end{tikzpicture}
		  
		\caption{The top half of the figure depicts the distance function on $\kp_k$ used in the proof of Theorem~\ref{thm:K4Path}. The thick double crosses with a circle are each to be replaced with the metric graph shown in the bottom half of the figure.  
		\label{fig:nK4Path}}
        \end{figure}
		    
        Let $(G, d')$ be obtained from $(\kp_k, d)$ by adding edges $v_{i}w_{i}$ with $d'(v_{i}w_{i}) = 2^k$ for all $i \in [k-1]$. Notice that for all $i$, the length of a shortest path between $v_{i}$ and $w_{i}$ in $(\kp_k, d)$ is $2^k$. 
        Therefore, $(\kp_k, d)$ is a metric graph if and only if $(G, d')$ is a metric graph. Observe that the restriction of $d'$ to every $K_4$ subgraph of $G$ is a distance function. Therefore, $(G, d')$ and hence also $(\kp_k, d)$ is a metric graph by Lemma~\ref{lem:distancefunction}.
  
        Consider the matching $M = \{v_{i-1}v_{i}, w_{i-1}w_{i} \mid i \equiv 1 \pmod{2} \}$. If $k$ is even, then we also add the edge $v_{k}w_{k}$ to $M$. Thus $|M| = k+1$ always. We claim that the edges of $M$ are pairwise incompatible.
        To see this, let $e = xx'$ and $f = yy'$ be distinct edges of $M$. Let $P$ be a shortest $x$--$y$ path, and $P'$ be a shortest $x'$--$y'$ path. We claim that $d(P) + d(P') \leq 2 \cdot 2^{k}$ (see next paragraph for a proof). However, $d(e) + d(f) > 2 \cdot 2^{k}$ because $e, f \in M$. Therefore, by Lemma~\ref{lem:incompatible}, $e$ and $f$ are incompatible. Since $|M| = k+1$, $f_\infty(\kp_k) \geq k+1$, as required. 
        
        To prove the claim, we split the discussion into two cases. A \emph{segment} in $\kp_k$ is any subgraph induced by $\{v_{i}, w_{i} \mid i = 4q + r,\ r \in \{0,1,2,3\},\ i \leq k\}$ for some $q$. If $e$ and $f$ belong to the same segment, then it is easy to see that $d(P) + d(P') \leq 2 \cdot 2^k$. (Notice that sometimes $d(P) = 2^k+1$ and $d(P') = 2^k-1$.) Now if $a$ and $b$ are any two vertices in distinct segments (indexed by $q$ and $s$, with $q < s$), then there is a $a$--$b$ path $Q$ such that 
        \begin{align*}
        d(Q) &\leq 1 + 1 + 1 + 2^{2q+3} + 1 + 1 + 1 + \cdots + 2^{2s-1} + 1 + 1 + 1 + (2^k - 2^{2s+1}) + 1 + 1 + 1\\
        &\leq \underbrace{(3s + 3)}_{\leq 1+2+4+2^{2s}} + 2^{3} + 2^{5} + \cdots + 2^{2s-1} - 2^{2s+1} + 2^k
        \leq \sum_{i = 0}^{2s} 2^i - 2^{2s+1} + 2^k
        \leq 2^k\,.
        \end{align*}
        It follows that $d(P) + d(P') \leq 2 \cdot 2^k$ in this case too.
    \end{proof}

	\begin{theorem}\label{thm:K4fan}
		For all $k \in \NN$, $f_\infty(\kf_k) =  k +1$.
	\end{theorem}
	\begin{proof}
	    For all $i \in [k]$, we label the vertices of the $i$th copy of $K_4$ in $\kf_k$ as $v_0, v_{2i-1}, v_{2i}, v_{2i+1}$. Remember that in order to obtain $\kf_k$ we form the $2$-sum of these $k$ copies of $K_4$ and delete every edge that is in two consecutive copies. Thus $V(\kf_k) = \{v_j \mid j \in \{0,\ldots,2k+1\}\}$ and $E(\kf_k) = \{v_0v_1,v_0v_{2k+1}\} \cup \{v_0v_{2i},v_{2i-1}v_{2i},v_{2i-1}v_{2i+1},v_{2i}v_{2i+1}\}$.
	
        By Lemma~\ref{lem:upperbound}, it suffices to show $f_\infty(\kf_k) \geq k+1$. Consider the following distance function $d$ on $\kf_k$: 
        \begin{align*}
	        d(v_0v_1) &= 1\,, \\
		    d(v_0v_{2i}) &= 1 &\text{for $i \in [k]$}\,,\\
		    d(v_{2i-1}v_{2i+1}) &= 1 &\text{for $i \in [k]$}\,,\\
		    d(v_{2i}v_{2i+1}) &= i &\text{for $i \in [k]$}\,,\\
		    d(v_{2i}v_{2i-1}) &= i +1&\text{for $i \in [k]$}\,,\\
		    d(v_0v_{2k+1}) &= k+1\,. \\
	    \end{align*}
		    
        As before, by Lemma~\ref{lem:distancefunction}, we can prove that $d$ is a distance function. Notice that $v_0$ is at distance $i+1$ from $v_{2i+1}$ for each $i \in [k-1]$.
    
        Consider the matching $M = \{v_0v_{2k+1}\}\cup\{v_{2i}v_{2i-1} \mid i\in [k] \}$  in $(\kf_k, d)$.  See Figure~\ref{fig:K4fan} for an illustration of the distance function $d$ and the matching $M$ in $\kf_5$.  We let the reader verify, with the help of Lemma~\ref{lem:incompatible}, that all edges of $M$ are pairwise incompatible. Since $|M| = k+1, f_\infty(\kf_k) \geq k+1$ as required.
    \end{proof}
		        
 
    \begin{figure}[ht]
	    \centering
	    
		\begin{tikzpicture}[scale=2, inner sep=2.5pt]
		    \tikzstyle{vtx}=[circle,draw,thick,fill=black!5]
		    \clip (-1.7, -1.6) rectangle (1.7,1.6);
		    \begin{scriptsize}
		            \node[vtx] (v0) at (0,0){};
		            \node[vtx] (v1) at (180:1.5){};
		            \node[vtx] (v2) at (155:0.8){};
		            \node[vtx] (v3) at (130:1.5){};
		            \node[vtx] (v4) at (105:0.8){};
		            \node[vtx] (v5) at (80:1.5){};
		            \node[vtx] (v6) at (50:1.5){};
		            \node[vtx] (v7) at (25:0.8){};
		            \node[vtx] (v8) at (0:1.5){};
		            \node[vtx] (v9) at (-30 :1.5){};
		            \node[vtx] (v10) at (-55:0.8){};
		            \node[vtx] (v11) at (-80:1.5){};
		            
		            \draw[very thick, red] (v0)--(v1) node[midway,below]{$k+1$};
		            \draw[thick](v0)--(v2) node[midway, above]{$1$};
		            \draw[thick](v0)--(v4) node[midway,right]{$1$};
		            \draw[thick](v0)--(v7) node[midway, below]{$1$};
		            \draw[thick](v0)--(v10) node[midway,right]{$1$};
		            \draw[thick](v0)--(v11) node[midway,left]{$1$};
		            \draw[thick, loosely dotted] (v5)--(v6);
		            \draw[thick, loosely dotted] (v8)--(v9);
		            
		            \draw[thick](v2)--(v1) node[midway,above]{$k$};
		            \draw[thick](v3)--(v1) node[midway,left]{$1$};
		            \draw[very thick, red] (v2)--(v3) node[midway, right]{$k+1$};
		            
		            \draw[thick](v4)--(v3) node[midway, right]{$k-1$};
		            \draw[very thick, red] (v4)--(v5) node[midway, right]{$k$};
		            \draw[thick](v3)--(v5) node[midway,above]{$1$};
		            
		            \draw[thick](v7)--(v6) node[midway, left]{$i$};
		            \draw[very thick, red] (v7)--(v8) node[midway, below left]{$i+1$};
		            \draw[thick](v6)--(v8) node[midway,right]{$1$};
		            
		            \draw[thick](v10)--(v9) node[midway,below]{$1$};
		            \draw[very thick, red] (v10)--(v11) node[midway,right]{$2$};
		            \draw[thick](v9)--(v11) node[midway,below]{$1$};            
		            
		        \end{scriptsize}
		        \end{tikzpicture}
		        \hspace{5mm}
		\begin{tikzpicture}[scale=2, inner sep=2.5pt]
		    \clip (-1.7, -1.6) rectangle (1.7,1.6);
		    \tikzstyle{vtx}=[circle,draw,thick,fill=black!5]
		    \begin{scriptsize}
            \node[vtx] (v0) at (0,0){};
            \node[vtx] (v1) at (180:1.5){};
            \node[vtx] (v2) at (155:0.8 ){};
            \node[vtx] (v3) at (130:1.5){};
            \node[vtx] (v4) at (105:0.8){};
            \node[vtx] (v5) at (80:1.5){};
            \node[vtx] (v6) at (55:0.8){};
            \node[vtx] (v7) at (30:1.5){};
            \node[vtx] (v8) at (5:0.8){};
            \node[vtx] (v9) at (-20 :1.5){};
            \node[vtx] (v10) at (-45:0.8 ){};
            \node[vtx] (v11) at (-70:1.5){};
            
                \draw[very thick, red ] (v0)--(v1) node[midway, below]{$6$};
                \draw[thick](v0)--(v2)node[midway, above]{$1$};
                \draw[thick](v0)--(v4)node[midway, right]{$1$};
                \draw[thick](v0)--(v6)node[midway, above]{$1$};
                \draw[thick](v0)--(v8)node[midway, above]{$1$};
                \draw[thick](v0)--(v10) node[midway, right]{$1$};
                \draw[thick](v0)--(v11)node[midway, left]{$1$};
            
                \draw[thick](v2)--(v1) node[midway, above]{$5$};
                \draw[thick](v3)--(v1) node[midway, left]{$1$};
                \draw[very thick, red ] (v2)--(v3) node[midway, right]{$6$};
            
                \draw[thick](v4)--(v3) node[midway, below]{$4$};
                \draw[very thick, red ] (v4)--(v5) node[midway, right]{$5$};
                \draw[thick](v3)--(v5) node[midway, above]{$1$};
            
                \draw[thick](v5)--(v6)node[midway, left]{$3$};
                \draw[very thick, red ] (v7)--(v6) node[midway, below]{$4$};
                \draw[thick](v5)--(v7)node[midway, right]{$1$};
            
                \draw[thick](v7)--(v8)node[midway, left]{$2$};
                \draw[very thick, red ] (v8)--(v9) node[midway, left]{$3$};
                \draw[thick](v7)--(v9)node[midway, right]{$1$};
            
                \draw[thick](v9)--(v11)node[midway, right]{$1$};
                \draw[very thick, red ] (v10)--(v11) node[midway, left]{$2$};
                \draw[thick](v9)--(v10)node[midway, above]{$1$};
           \end{scriptsize}
        \end{tikzpicture}
        \caption{$(\kf_k, d)$  as in the proof of Theorem~\ref{thm:K4fan} and $(\kf_5,d)$. The red edges are pairwise incompatible. \label{fig:K4fan}}		        
    \end{figure}
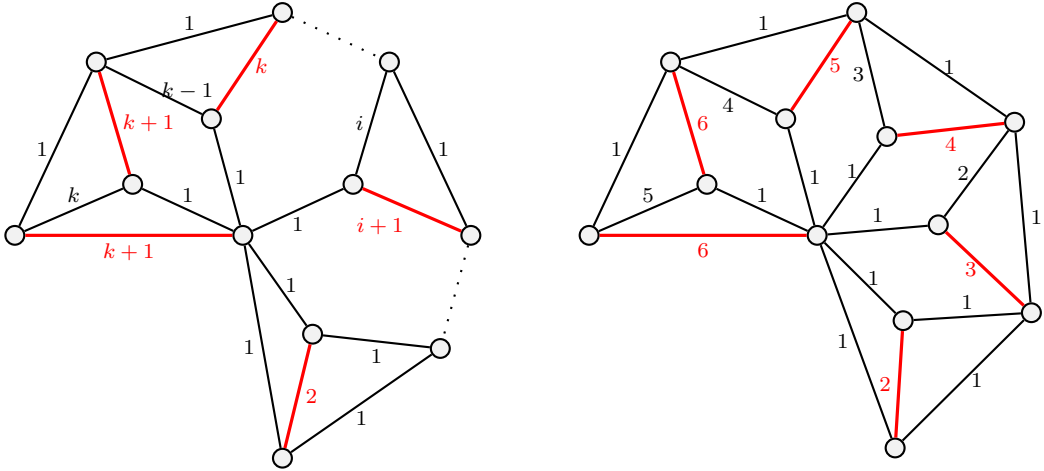

    \begin{theorem} \label{thm:excludedminors} 
    For all $k \geq 2$, $\ks_k, \kp_k, \kf_k$ are excluded minors for the property $f_\infty(G)\leq k$. 
    \end{theorem}
    \begin{proof}
        Let $H$ be one of $\ks_k, \kp_k, \kf_k$. By Theorems~\ref{thm:K4Star},~\ref{thm:K4Path}, and~\ref{thm:K4fan}, we know $f_\infty(H) >k$.
        
        When deleting or contracting an edge in $H$, we get a minor $H'$ which can be expressed as a $2$-sum of two graphs $H_1$, $H_2$ with the following properties. First, $H_1 \in \{\ks_\ell, \kp_\ell, \kf_\ell\}$ for some $\ell < k$ (and $H_1$ is of the same type as $H$). Second, $H_2$ has a degree-$2$ vertex and recursively suppressing the degree-$2$ vertices from $H_2$ results in a graph $H_2'$ such that $H_2' \in \{\ks_m, \kp_m, \kf_m\}$ for some $m \leq k-l-1$ (again $H_2'$ is of the same type as $H$), or $H_2'$ is a single edge (this corresponds to the case $m = 0$). 
        
        By Lemma~\ref{lem:gluing} and Lemma~\ref{lem:upperbound},
        \[
        f_\infty(H') \leq f_\infty(H_1)+f_\infty(H_2)-1 =  f_\infty(H_1)+f_\infty(H_2')-1 \leq (l+1)+(m+1)-1 \leq k.
        \]
        Thus, $H$ is an excluded minor for $f_\infty(G) \leq k$. 
    \end{proof}
    
    \begin{theorem}\label{thm:necklace}
        For all $k \in \NN$, $f_\infty(\kn_k)\geq k+1$.
    \end{theorem}
    \begin{proof}
        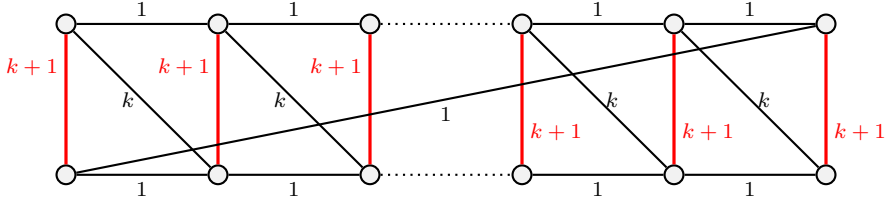
\begin{figure}[ht]
            \centering
        \begin{tikzpicture}[x=1cm, y=1.cm, inner sep=2.5pt]
            \tikzstyle{vtx}=[circle,draw,thick,fill=black!5]
        \begin{scriptsize}
            \node[vtx] (v0) at (0,0){};
            \node[vtx] (v1) at (0,2){};
            \node[vtx] (v2) at (2,0){};
            \node[vtx] (v3) at (2,2){};
            \node[vtx] (v4) at (4,0){};
            \node[vtx] (v5) at (4,2){};
            \node[vtx] (v6) at (6,0){};
            \node[vtx] (v7) at (6,2){};
            \node[vtx] (v8) at (8,0){};
            \node[vtx] (v9) at (8,2){};
            \node[vtx] (v10) at (10,0){};
            \node[vtx] (v11) at (10,2){};
   
            \draw[very thick, red] (v0)--(v1) node[midway, near end, left]{$k+1$};
            \draw[very thick, red] (v2)--(v3) node[midway, near end, left]{$k+1$};
            \draw[very thick, red] (v4)--(v5) node[midway, near end, left]{$k+1$};
            \draw[very thick, red] (v6)--(v7) node[midway, near start, right]{$k+1$};
            \draw[very thick, red] (v8)--(v9) node[midway, near start, right]{$k+1$};
            \draw[very thick, red] (v10)--(v11) node[midway, near start, right]{$k+1$};
            
            \draw[thick] (v0)--(v2) node[midway, below]{$1$};
            \draw[thick] (v1)--(v3) node[midway, above]{$1$};
            \draw[thick] (v2)--(v4) node[midway, below]{$1$};
            \draw[thick] (v3)--(v5) node[midway, above]{$1$};
            \draw[thick,dotted] (v4)--(v6);
            \draw[thick,dotted] (v5)--(v7);
            \draw[thick] (v6)--(v8) node[midway, below]{$1$};
            \draw[thick] (v7)--(v9) node[midway, above]{$1$};
            \draw[thick] (v8)--(v10) node[midway, below]{$1$};
            \draw[thick] (v9)--(v11) node[midway, above]{$1$};
            
            \draw[thick] (v1)--(v2) node[midway, left]{$k$};
            \draw[thick] (v3)--(v4) node[midway, left]{$k$};
            \draw[thick] (v7)--(v8) node[midway, right]{$k$};
            \draw[thick] (v9)--(v10) node[midway, right]{$k$};
            \draw[thick] (v11)--(v0) node[midway, below]{$1$};
        \end{scriptsize}
            \end{tikzpicture}
            \caption{$(\kn_k, d)$ as in the proof of Theorem~\ref{thm:necklace}.
            \label{fig:Nk}}
        \end{figure}
        
        Let $V(\kn_k) = \{v_0, \dots, v_k\} \cup \{w_0, \dots, w_k\}$ and 
        \[
        E(\kn_k) = \{v_{i-1}v_i,v_{i}w_{i},v_{i-1}w_i,w_{i-1}w_i \mid i \in [k]\} \cup \{v_0w_0, w_0v_k\}.
        \]
        
        Consider the distance function $d$ such that $d(w_0v_k) = d(v_{i-1}v_i) = d(w_{i-1}w_i) = 1$, $d(v_{i-1}w_i) = k$ for all $i \in [k]$ and $d(v_iw_i) = k+1$ for all $i=0, \dots, k$.  It is easy to check that $d$ is indeed a distance function.  
        Let $M = \{v_iw_i \mid i = 0, \dots, k\}$.  See Figure~\ref{fig:Nk} for an illustration of $(\kn_k, d)$ and $M$, where $v_0 \cdots v_k$ and $w_0 \cdots w_k$ are the topmost and bottommost paths, respectively. 
        
        We claim that the edges in $M$ are pairwise incompatible. 
        To see this, first observe that the shortest $v_i$--$v_j$ and $w_i$--$w_j$ paths both have weight $|j-i|\leq k$ since all edges in these paths have weight $1$, hence the cumulative weight of these paths is at most $2k$. If $i>j$, then 
        \[
        d(v_iv_{i+1}\cdots v_kw_0w_1 \cdots w_j)+d(v_jv_{j+1}\cdots v_{i-1}w_i) = (k-i+j+1)+ (i-j-1+k) = 2k.
        \]
        This shows that there exist a $v_i$--$w_j$ path and a $v_j$--$w_i$ path of cumulative weight $2k$. 
        Since $d(v_iw_i)+d(v_jw_j) = 2k+2$, the conditions of Lemma~\ref{lem:incompatible} are satisfied and we get that $v_iw_i$ and $v_jw_j$ are incompatible for all $i \neq j$. Hence, $f_\infty(\kn_k)\geq k+1$. 
    \end{proof}

Since $\kn_k$ is $3$-connected, it is difficult to adapt the proof of Theorem~{\ref{thm:excludedminors}} to show that $\kn_k$ is also an excluded minor for the property $f_\infty(G) \leq k$.  However, we conjecture that this is true.

    \section{$2$-connected graphs} \label{sec:reduction}
    
    In this section, we show that it is enough to prove our main theorem, Theorem~\ref{thm:main}, for $3$-connected graphs.  
    To do so, we introduce a variant of SPQR trees in Section~\ref{sec:contractedSPQR}. In section~\ref{sec:extendFlatSet}, we show that in a graph $G_1+_eG_2$ obtained as a $2$-sum of two graphs $G_1$ and $G_2$, we can merge flat sets from $G_1$ and $G_2$ under some conditions. In Section~\ref{sec:glumpkin}, we present several lemmas that show how to bound $f_\infty(H)$, where $H$ is obtained by gluing several $2$-connected graphs on a given graph. At the end of this section, we also show how to complete the proof of Theorem~\ref{thm:main} under some additional assumptions.
    
    \subsection{Contracted SPQR trees}\label{sec:contractedSPQR}
    In this context we need to consider {\em multigraphs} that are minors of a simple $2$-connected graph, that is, parallel edges resulting from edge contractions are kept. 
    (Loops on the other hand are not important for our purposes and thus can safely be discarded.) 
    SPQR trees were introduced in~\cite{DiBattista1996} as a way to decompose a $2$-connected graph across its  $2$-separations. They are defined as follows. 
    
    Let $G$ be a (simple) $2$-connected graph. 
    The \emph{SPQR tree $T_G$ of $G$} is a tree each of whose node $a\in V(T_G)$ is associated with a multigraph $H_a$ which is a minor of $G$. 
    Each vertex $x \in V(H_a)$ is a vertex of $G$, that is, $V(H_a) \subseteq V(G)$. 
    Each edge $e \in E(H_a)$ is classified either as a \emph{real} or \emph{virtual} edge. By the construction of an SPQR tree each edge $e\in E(G)$ appears in exactly one minor $H_a$ as a real edge, and each edge $e\in H_a$ which is classified real is an edge of $G$. 
    The SPQR tree $T_G$ is defined recursively as follows.
    \begin{enumerate}
        \item If $G$ is $3$-connected, then $T_G$ consists of a single \emph{$R$-node} $a$ for which we have $H_a = G$. 
        All edges of $H_a$ are real in this case.
        
        \item If $G$ is a cycle, then $T_G$ consists of a single \emph{$S$-node} for which $H_a = G$. 
        Again, all edges of $H_a$ are real in this case. 
        
        \item Otherwise $G$ has a cutset $\{x,y\}$ such that the vertices $x$ and $y$ have degree at least $3$. In this case we construct $T_G$ inductively. First we add a \emph{$P$-node} $a$ to $T_G$, for which $H_a$ is the graph consisting of the single edge $xy$. 
        The edge $xy$ of $H_a$ is real if $xy$ is an edge of $G$, and virtual otherwise.  
        Next we consider the connected components $C_1,\dots, C_r$ ($r\geq 2$) of $G - \{x,y\}$. Let $G_i$ be the graph $G[V(C_i)\cup \{x,y\}]$ with the additional edge $xy$ if it is not already there. Since we include the edge $xy$, each $G_i$ is $2$-connected and we can construct the corresponding SPQR tree $T_{G_i}$ by induction. Let $a_i$ be the (unique) node in $T_{G_i}$ for which $xy$ is a real edge in $H_{a_i}$. 
        In order to construct $T_G$, we make $xy$ a virtual edge in the node $a_i$,  and connect $a_i$ to $a$ in $T_G$. Finally, we add parallel virtual edges $xy$ to $H_a$ so that it has exactly $r$ virtual edges $xy$.
    \end{enumerate} 
    Notice that minors corresponding to $S$-nodes and $R$-nodes are simple graphs, whereas those corresponding to $P$-nodes are multigraphs consisting of two vertices linked by at least two virtual edges and  possibly a real one. 
    To each edge $ab$ of the SPQR tree $T_G$ corresponds a unique virtual edge $e \in E(H_a) \cap E(H_b)$ with ends $x, y \in V(G)$. Thus we can define a corresponding multigraph $H_{a,b}$ which is the minor of $G$ obtained by taking the $2$-sum of $H_a$ and $H_b$ in which the edge $e$ is deleted. 
    (To be precise, one virtual edge $xy$ from each of $H_a$ and $H_b$ is deleted in the operation, other copies of $xy$, if any, are kept in the resulting graph.) 
    Similarly, we can define a unique minor of $G$ for each \emph{subtree} of $T_G$ by performing one $2$-sum operation as described above for each edge of the subtree. 
    
    Let $G$ be a $2$-connected graph, and let $T_G$ be the SPQR tree of $G$. We define the \emph{contracted SPQR tree $T'_G$} as the tree obtained from $T_G$ by contracting every maximal connected subtree of $T_G$ each of whose nodes is either a $S$-node or a $P$-node, see Figure~\ref{fig:spqrSimplify} for an example. We call the new nodes resulting from the contraction \emph{$O$-nodes}. Each node $a$ of $T'_G$ has a unique corresponding minor $H_a$ of $G$. If $a$ is an $R$-node, then we keep the same minor as in $T_G$. Otherwise, $a$ is an $O$-node and $H_a$ is the minor of $G$ corresponding to the subtree of $T_G$ that was contracted to node $a$ of $T'_G$. 
    
    \begin{figure}
        \centering
        \includegraphics[width=\textwidth]{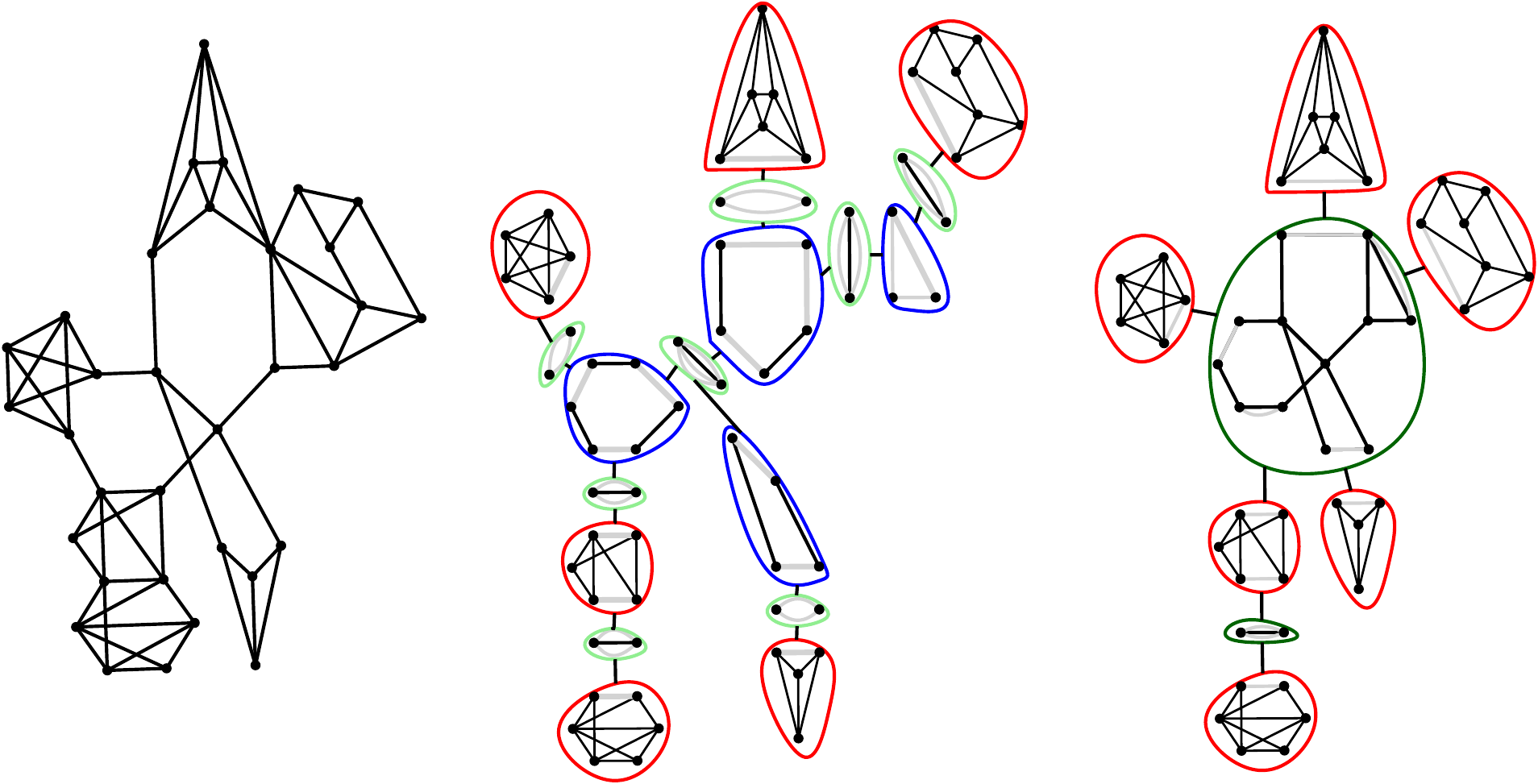}
        \caption{ An example of a $2$-connected graph $G$, its SPQR tree $T_G$, and the contracted SPQR tree $T'_G$.} \label{fig:spqrSimplify}
    \end{figure}
    
    We quickly give some standard terminology before stating our first result of the section. The \emph{length} of a path in $G$ is its number of edges. The \emph{diameter} of a graph $G$ is the maximum length of a shortest path between any two vertices.

    \begin{lemma}\label{lem:spqrSimplify}
        Let $G$ be a $2$-connected graph with minimum degree at least $3$. 
        \begin{enumerate} 
            \item Every $O$-node in $T'_G$ corresponds to a $2$-connected treewidth-$2$ graph.
            \item All leaves of $T'_G$ are $R$-nodes.
            \item If the diameter of $T'_G$ is at least $6k$, then $G$ contains $\kp_k$ or $\kf_k$ as a minor.
        \end{enumerate}
    \end{lemma}
    \begin{proof}
        (1) Let $o$ be an $O$-node of $T'_G$. Its corresponding minor $H_o$ is obtained by $2$-sums from cycles corresponding to $S$-nodes, and parallel edges corresponding to $P$-nodes. Hence $H_o$ is $2$-connected and has treewidth $2$.\medskip 
        
        (2) Suppose for a contradiction that some leaf $o$ of $T'_G$ is an $O$-node. Since a $P$-node cannot be a leaf in $T_G$,  the subtree corresponding to $o$ in $T_G$ has at least one leaf $s$ which is an $S$-node. Because $s$ is a leaf, $H_s$ contains exactly one virtual edge. Since $H_s$ is a cycle of length at least $3$, there is at least one degree-$2$ vertex in $G$, a contradiction. 
        \medskip
        
        (3) Let $P = a_0  \cdots a_{m}$ be a maximum length path in $T'_G$. By maximality, $P$ is a leaf-to-leaf path in $T'_G$, $a_i$ is an $R$-node for even $i$ and an $O$-node for odd $i$, and $m$ is even.  
        
        For $i \in [m-1]$, we let $x_i$ and $y_i$ be the ends of the virtual edge in $E(H_{a_i}) \cap E(H_{a_{i+1}})$. Since $H_{a_i}$ is $2$-connected, exchanging $x_i$ and $y_i$ if necessary we may assume that for each $i \in [m-1]$, $H_{a_i}$ contains an $x_{i-1}$--$x_i$ path $P_i$ and a $y_{i-1}$--$y_i$ path $Q_i$ such that $P_i$ and $Q_i$ are vertex-disjoint. 
        
        Let $i \in [m-1]$ with $i$ even. 
        Let us emphasize that the vertices $x_{i-1},x_i, y_{i-1},y_i$ are not necessarily all distinct. 
        We call a $K_4$-model in $H_{a_i}$ \emph{good} if the intersections of the four vertex images with these vertices 
        fall in one of the following cases: 
        \begin{itemize}
            \item $\{x_{i-1}\}, \{x_i\}, \{y_{i-1}\} , \{y_i\}$, or
            \item $\{x_{i-1}, x_i\}, \{y_{i-1}\} , \{y_i\}, \emptyset$ with $x_{i-1} \neq x_i$,  or
            \item $\{x_{i}\}, \{y_{i-1}\} , \{y_i\}, \emptyset$ with $x_{i-1} = x_i$,  or
            \item $\{x_{i-1}\}, \{x_i\}, \{y_{i-1}, y_i\}, \emptyset$ with $y_{i-1} \neq y_i$, or  
            \item $\{x_{i-1}\}, \{x_i\}, \{y_i\}, \emptyset$ with $y_{i-1} = y_i$. 
        \end{itemize}
        
        We claim that $H_{a_i}$ has a good $K_4$-model for each even $i \in [m-1]$. To see this, let $C_i = P_i+Q_i+x_{i-1}y_{i-1}+x_{i}y_{i}$. First suppose $V(C_i) = V(H_{a_i})$. Since $H_{a_i}$ is $3$-connected, there is an edge $e \in E(H_{a_i})$ distinct from $x_{i-1}y_{i-1}$ and $x_iy_i$ between $V(P_i)$ and $V(Q_i)$, and another edge $f$ such that $C_i \cup \{e,f\}$ is a subdivision of $K_4$. Then $C_i + e + f$ contains a good $K_4$-model. Assume now that $V(C_i) \subsetneq V(H_{a_i})$. It follows that there is a component of $H_{a_i} - V(C_i)$ that sends edges to three vertices of $C_i$ which are neither all in $V(P_i)$ nor all in $V(Q_i)$; otherwise $H_{a_i} - \{x_{i-1}, x_i\}$ or $H_{a_i} - \{y_{i-1}, y_i\}$ would be disconnected.  
        Thus, $H_{a_i}$ has a good $K_4$-model whose vertex images are a single component of $H_{a_i} - V(C_i)$ and three disjoint connected subgraphs of $C_i$.
        
        We say that a good $K_4$-model in $H_{a_i}$ is \emph{type-$0$} if $x_{i-1}, x_i, y_{i-1}$, and $y_i$ are in distinct vertex images, \emph{type-$1$} if $x_{i-1}$ and $x_i$ are in the same vertex image, and \emph{type-$2$} if $y_{i-1}$ and $y_i$ are in the same vertex image.  
        We pick a good $K_4$-model in each even $i \in [m-1]$.  Since $m \geq 6k$, at least $k$ of these good $K_4$-models are of the same type, say type-$t$ for some $t \in \{0,1,2\}$.
        
        We obtain the required minor of $G$ as follows. First, for each even $i \in [m-1]$ such that $H_{a_i}$ contains a type-$t$ good $K_4$-model, we contract the vertex images of the $K_4$-model and delete the vertices not belonging to any vertex image. Second, for each index $i \in [m-1]$ not yet considered,  we contract the edges in $E(P_i) \cup E(Q_i)$ and delete all other vertices of $H_{a_i}$. Note that this second step has the effect of $2$-summing the type-$t$ good $K_4$-models. Therefore, we obtain a $\kp_k$ minor in $G$, if $t = 0$, and a $\kf_k$ minor in $G$ in the other two cases.
    \end{proof}
    
    \subsection{Extending flat sets in $2$-connected graphs}\label{sec:extendFlatSet}    
    We now develop some more tools to handle $2$-separations in graphs. Assume that $G = G_1\oplus_eG_2$ with $e = vw$.
    The goal is to improve the bounds for $f_\infty(G)$ given in Lemma~\ref{lem:gluing}. 
    Recall that the proof of Lemma~\ref{lem:gluing} relies on the fact that it is possible to merge a flat set $F_1$ of $(G_1,d_1)$ and a flat set $F_2$ of $(G_2,d_2)$ into one flat set $F_1\cup F_2$ of $(G,d)$ whenever $(v,w) \in F_1 \cap F_2$.
    
    Here is another proof of this fact. Let $(D,l)$, $(D_1,l_1)$ and $(D_2,l_2)$ denote the weighted digraphs obtained by bidirecting $(G,d)$, $(G_1,d_1)$ and $(G_2,d_2)$ respectively. For $i \in [2]$, consider a potential $p_i$ on $(D_i,l_i)$ such that $p_i(x) - p_i(y) = d(xy)$ for all $(x,y) \in F_i$. Since $(v,w) \in F_1 \cap F_2$, we have $p_1(v)-p_1(w) = p_2(v)-p_2(w) = d(vw)$. Hence, it is possible to shift one of the potentials in order to satisfy $p_1(v) = p_2(v)$ and $p_1(w) = p_2(w)$. The potential $p_1\cup p_2 :V(G)\to \RR$ on $(D,l)$ such that $(p_1 \cup p_2)(u) = p_i(u)$ if $u \in V(G_i)$ for $i\in[2]$ witnesses that $F_1\cup F_2$ is a flat set.
     
    Suppose now that the flat sets $F_1$, $F_2$ of $(G_1,d_1)$ and $(G_2,d_2)$ are such that $(v,w)\in F_1$ but $(v,w), (w,v) \notin F_2$. The previous idea does not work anymore since we could have $|p_2(v)-p_2(w)| < d(vw)$. Hence, we can no longer combine the potentials $p_1$ and $p_2$. However, there possibly exists a potential $p_1'$ for $F_1\delete\{(v,w)\}$  such that $p'_1(v)-p'_1(w) = p_2(v)-p_2(w)$. In that case, $p'_1\cup p_2$ is a potential for $(F_1\cup F_2)\delete \{(v,w)\}$ on $(D,l)$. It follows that in this case $(F_1\cup F_2) \delete \{(v,w)\}$ is a flat set. 
    
    We now introduce the notion of \emph{compressible edges}, which are edges for which we can apply the idea of the previous paragraph. In this context, it is helpful to switch from directed notions to undirected notions. We call a set $F$ of edges of $G$ \emph{flattenable} (in $(G,d)$) if some orientation of $F$ is a flat set in $(G,d)$, that is, if there exists a potential $p$ on $(D,l)$ such that $|p(v) - p(w)| = d(vw)$ for all $vw \in F$. Let $F \subseteq E(G)$ be flattenable in $(G,d)$. An edge subset $\Gamma \subseteq F$ is said to be \emph{compressible} in $F$ if for all $\lambda \in [0,1]^\Gamma$ there exists a potential $p$ on $(D,l)$ such that $|p(v) - p(w)| = \lambda(vw) \cdot d(vw)$ for all $vw \in \Gamma$ and $|p(v) - p(w)| = d(vw)$ for all $vw \in F \delete \Gamma$. We define a \emph{frame} in $(G,d)$ as a pair $(\Gamma,F)$ where $\Gamma \subseteq F \subseteq E(G)$, $F$ is flattenable in $(G,d)$ and $\Gamma$ is compressible in $F$.  
    
    Notice that subsets of flattenable sets are flattenable, and that $f_\infty(G)$ is the least integer $k$ such that for every distance function $d$ the edges of the metric graph $(G,d)$ can be partitioned into $k$ flattenable sets.
    
    The next lemma follows directly from the formal definition of compressible edges.
    
    \begin{lemma}\label{lem:special2sum}
        Let $G = G_1\oplus_{vw}G_2$, and let $d$ be a distance function on $G$. For $i \in [2]$, let $d_i$ be the restriction of $d$ to $G_i$ and let $(\Gamma_i,F_i)$ be a frame in $(G_i,d_i)$.
        
        \begin{enumerate}[(i)]
            \item If $vw \in (F_1 \delete \Gamma_1) \cap (F_2 \delete \Gamma_2)$ then $(\Gamma_1 \cup \Gamma_2, F_1 \cup F_2)$ is a frame in $(G,d)$.
            \item If $vw \in \Gamma_1 \cup \Gamma_2$ then $((\Gamma_1 \cup \Gamma_2) \delete \{vw\},(F_1 \cup F_2) \delete \{vw\})$ is a frame in $(G,d)$.
        \end{enumerate}
    \end{lemma}

        We will now use this lemma to improve some bounds given by Lemma~\ref{lem:gluing}. 
        For simplicity, we call {\em gluing} the $2$-sum operation where the edge involved in the $2$-sum is kept. 
        Let $H$ be a graph obtained by gluing graphs $G_1,\dots, G_m$ on distinct edges of a graph $G$. 
        That is, there are distinct edges $e_1, \ldots, e_m$ such that $H = G \oplus_{e_1} G_1 \cdots \oplus_{e_m} G_m$. The bound obtained by applying Lemma~\ref{lem:gluing} is $f_\infty(H)\leq f_\infty(G)+\sum_{i \in [m]}{(f_\infty(G_i)-1)}$. We provide better bounds in the following cases. First, when $G$ is a $2$-connected outerplanar graph  and all $G_i$ are glued on edges of its outer cycle. 
        Second, when $G$ is a $2$-connected treewidth-$2$ graph and $H$ has no $\ks_k$ minor. 
        
    \begin{figure}
        \centering
        \begin{tikzpicture} [x=0.6cm, y=0.6cm]
            \tikzstyle{snake}=[decorate, decoration={snake}]
            \tikzstyle{vtx}=[circle,draw,thick,fill=black!5]
            \begin{scriptsize}
            
            \node[vtx] (a) at (0,0){};
            \node[vtx] (c) at (0,3){};
            
            \node[] (a') at (0.25, -0.5){};
            \node[] (c'') at (0.25, 3.25){};
            
            \node[vtx] (d) at (-2,3){};
            \node[vtx] (e) at (-3,0){};
            
            \draw[thick, green] (e)--(a);
            \draw[thick, blue] (c)--(a) (d)--(a) (d)--(e);
            \draw[thick, red] (d)--(c);
            
            \draw[thick, snake, red] (-0.25,-0.5)--(-3,-0.5) (-3.4, 0.15)--(-2.5,2.85);
            \draw[thick, snake, green] (0,3.5)--(-2,3.5) (0.5,0)--(0.5,3);
            
            \end{scriptsize}
        \end{tikzpicture} \vspace{5mm}
        \begin{tikzpicture} [x=0.6cm, y=0.6cm]
            \tikzstyle{snake}=[decorate, decoration={snake}]
            \tikzstyle{vtx}=[circle,draw,thick,fill=black!5]
            \begin{scriptsize}
            
            \node[vtx] (a) at (0,0){};
            \node[vtx] (b) at (4,0){};
            \node[vtx] (c) at (0,3){};
            
            \node[] (a') at (-0., -0.5){};
            \node[] (a'') at (-0.5, -0.){};
            
            \node[] (b'') at (4.25, 0.25){};
            \node[] (b') at (4, -0.5){};
            
            \node[] (c') at (-0.5, 3){};
            \node[] (c'') at (0.25, 3.25){};
            \draw[thick, olive] (a)--(b)--(c);
            \draw[thick, cyan] (c)--(a);
            \draw[thick, snake, magenta] (a')--(b') (c')--(a'');
            \draw[thick, snake, cyan](b'')--(c'');
            
            \end{scriptsize}
        \end{tikzpicture}
        
        \hspace{5mm}
        \begin{tikzpicture} [x=0.6cm, y=0.6cm]
            \tikzstyle{snake}=[decorate, decoration={snake}]
            \tikzstyle{vtx}=[circle,draw,thick,fill=black!5]
            \begin{scriptsize}
            
            \node[vtx] (a) at (0,0){};
            \node[vtx] (b) at (4,0){};
            \node[vtx] (c) at (0,3){};
            
            \node[] (a') at (0.25, -0.5){};
            
            \node[] (b'') at (4.25, 0.25){};
            \node[] (b') at (4, -0.5){};
            \node[] (c'') at (0.25, 3.25){};
            
            \node[vtx] (d) at (-2,3){};
            \node[vtx] (e) at (-3,0){};
            
            \draw[thick, green] (a)--(b)--(c) (e)--(a);
            \draw[thick, blue] (c)--(a) (d)--(a) (d)--(e);
            \draw[thick, red] (d)--(c);
            
            \draw[thick, snake, red] (0.25,-0.5)--(4,-0.5) (-0.25,-0.5)--(-3,-0.5) (-3.4, 0.15)--(-2.5,2.85);
            \draw[thick, snake, blue] (b'')--(c'');
            \draw[thick, snake, green] (0,3.5)--(-2,3.5);
            
            \end{scriptsize}
        \end{tikzpicture}
        \caption{Illustration of the proof of Lemma~\ref{lem:outerplanarBound}: $G$ is a $2$-sum of $G' = K_4 - e$ and $K_3$. Each color defines a frame $(\Gamma,F)$ in the corresponding graph. Edges of $F \delete \Gamma$ are straight and edges of $\Gamma$ are wavy. The distance function is defined by taking the corresponding Euclidean distance in the figure.\label{fig:fan_2-sum_K_3}}
    \end{figure}
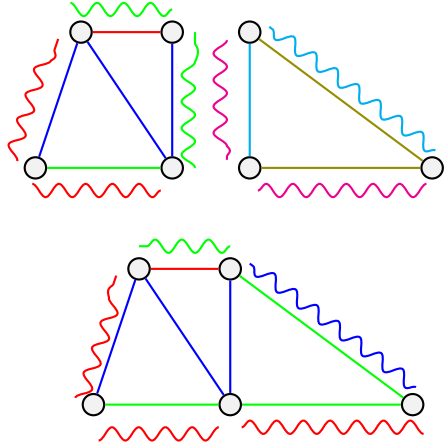
    
    \begin{lemma} \label{lem:outerplanarBound}
        Let $G$ be a $2$-connected outerplanar graph drawn in the plane with outer cycle $C$. Let $H$ be obtained from $G$ by gluing graphs $G_1,\dots, G_m$ on distinct edges of $C$. Let $M = \max_{i \in [m]}f_\infty(G_i)$. Then $f_\infty(H) \leq 3M$.
    \end{lemma}
    \begin{proof}
        We will show that $G$ satisfies the following property:
        \begin{quote} \em $(\star)$ For every distance function $d$ on $G$, there exist three frames $(\Gamma_j,F_j)$, $j \in [3]$, in $(G,d)$ such that each edge of $G$ is in at least one flattenable set $F_j$, and each edge of its outer cycle $C$ is in exactly two flattenable sets $F_j$ and in exactly one compressible set $\Gamma_j$.
        \end{quote}
        For $i \in [m]$, let $\{v_i,w_i\} = V(G_i) \cap V(G)$. Thus, $v_iw_i$ is an edge of $C$. Without loss of generality, we may assume that $v_iw_i$ is an edge of $H$. 
        
        Now let $d$ be some distance function on $H$. We will slightly abuse notation and let $d$ also denote the restriction of this distance function to $G$. For $i \in [m]$, let $d_i$ denote the restriction of $d$ to $G_i$.
        
        Assuming $(\star)$, we can find three frames $(\Gamma_j,F_j)$, $j \in [3]$, in $(G,d)$ as above. 
        For each $i \in [m]$, let $F^i_1$, \ldots, $F^i_M$ be a partition of the edges of $(G_i,d_i)$ into flattenable set. 
        By Lemma~\ref{lem:special2sum}, for every $j \in [3]$ and $k \in [M]$, 
        $$
        \left(F_{j} \cup \bigcup_{i \in I_j} F^i_k\right) \delete \{v_iw_i \mid i \in I_j\}
        $$
        is a flattenable set in $(H,d)$, where $I_j = \{i \in [m] \mid v_iw_i \in \Gamma_j\}$. These $3M$ flattenable sets cover the edges of $(H,d)$, which implies $f_\infty(G) \leq 3M$.
        
        To prove the lemma, it remains to show that the claimed frames $(F_j,\Gamma_j)$, $j \in [3]$ exist in $(G,d)$. We can assume that all inner faces of the drawing of $G$ are triangular faces (if not, add extra edges). We show the result by induction on the number of vertices.
        
        The base case is given by $G = K_3$. Let $V(K_3) = \{v_1,v_2,v_3\}$. Without loss of generality, we can assume $d(v_1v_2) \leq d(v_1v_3) \leq d(v_2v_3)$. It is easy to show that $(\Gamma_1,F_1) = (\{v_1v_2, v_1v_3\}, \{v_1v_2, v_1v_3\})$, $(\Gamma_2,F_2) = (\{v_2v_3\},\{v_2v_1, v_2v_3\})$, and $(\Gamma_3,F_3)  = (\emptyset,\{v_3v_1, v_3v_2\})$ are frames in $(G,d)$. For instance, one can use Lemma~\ref{lem:flatstar} to see that each $F_j$ is flattenable, and a direct verification to see that each $\Gamma_j$ is compressible in $F_j$. Thus $K_3$ satisfies $(\star)$.
        
        Now for the inductive case, suppose that $G$ has at least four vertices. Let $v$ be a degree-$2$ vertex of $G$ (which exists since $G$ is outerplanar and $2$-connected), and consider the graph $G' = G - v$. Let $v_1, v_2$ be the two neighbors of $v$ in $G$, with $d(vv_1)\geq d(vv_2)$. Let $C'$ be the cycle obtained from the outer cycle $C$ in $G$ by shortcutting the path $v_1vv_2$ to $v_1v_2$. 
        
        By induction, $(\star)$ holds for $G'$. Let $(\Gamma'_j,F'_j)$, $j \in [3]$ denote the corresponding frames. Consider three frames $(\Gamma''_j,F''_j)$, $j \in [3]$ for the triangle $vv_1v_2v$, as described in the base case of the induction.
        
         By permuting the indices if necessary, we may assume that $v_1v_2$ is in $(F'_1 \delete \Gamma'_1) \cap (F''_1 \delete \Gamma''_1)$, $\Gamma'_2$ and $\Gamma''_3$. By Lemma~\ref{lem:special2sum}, $(\Gamma_1,F_1) =  (\Gamma'_1 \cup \Gamma''_1, F'_1 \cup F''_1)$ and, for $j \in \{2,3\}$, $(\Gamma_j,F_j) = ((\Gamma'_j \cup \Gamma''_j) \delete \{v_1v_2\},(F'_j \cup F''_j) \delete \{v_1v_2\})$ are all frames in $(G,d)$. See Figure~\ref{fig:fan_2-sum_K_3} for an illustration. It is straightforward to check that these frames satisfy the required condition for $G$.
    \end{proof}

    \subsection{Handling several $2$-cutsets simultaneously}\label{sec:glumpkin}
    Before proceeding, we require the following easy lemma. Let $K_4-e$ be the graph obtained from $K_4$ by deleting an edge.  
    
    \begin{lemma}[\cite{FHJV17}] \label{lem:k4-e}
Let $G$ be a  $2$-connected graph with distinct vertices $u$ and $v$ such that $\deg_G (w) \geq 3$ for all $w \in V(G) \delete \{u,v\}$.  Then $G$ has a $K_4-e$ minor where $u$ and $v$ are contracted to the ends  of $e$.  
\end{lemma}

   Let $G$ be a graph together with a subset of $E(G)$ called \emph{glued edges}.   We say that $G$ has a \emph{$k$-glumpkin} minor if $G$ contains $k$ glued edges in parallel as a minor, that is, if there is a way of choosing a connected subgraph $H$ of $G$ containing at least $k$ glued edges, and of contracting all but $k$ edges of $H$ in such a way that the resulting minor consists of $k$ parallel glued edges. A $k$-glumpkin minor is \emph{rooted} at a glued edge $r$ if it contains $r$.   If $H$ is obtained by gluing graphs $G_1,\dots, G_m$ on distinct edges of $G$, an edge $e \in E(G)$ is a \emph{glued edge} if $e \in E(G) \cap E(G_i)$ for some $i \in [m]$.  The parameter we are really interested in is the largest $\ks_k$ minor in $H$.  However, the next lemma relates $\ks_k$ minors in $H$ to  $k$-glumpkin minors in $G$.  
   
   \begin{lemma} \label{lem:glumstar}
     Let $H$ be obtained by gluing $2$-connected graphs $G_1,\dots, G_m$ on distinct edges of a graph $G$ such that $H$ has minimum degree at least $3$.  If $G$ has a $k$-glumpkin minor, then $H$ has an $\ks_k$-minor.
   \end{lemma}
   \begin{proof}
         Let $u_iv_i$ be the glued edge of $G_i$.  Since $H$ has minimum degree at least $3$, $\deg_{G_i}(w) \geq 3$ for all $w \in V(G_i) \delete \{u_i,v_i\}$. By Lemma~\ref{lem:k4-e}, $G_i$ has a $K_4$ minor containing the glued edge $u_iv_i$, for all $i \in [m]$.  Therefore, since $G$ has a $k$-glumpkin minor, $H$ has an $\ks_k$-minor.  
   \end{proof}
   
   \begin{lemma} \label{lem:outerplanarBound2}
        For all $k, M \in \NN$, let $g_{\ref{lem:outerplanarBound2}}(k,M)=3^kM$.  
        Let $H$ be a graph obtained from a $2$-connected outerplanar graph $G$ by gluing $2$-connected graphs $G_1,\dots, G_m$ on distinct edges of $G$.  
        Let $C$ be the outercycle of $G$ and let $M = \max_{i \in [m]}f_\infty(G_i)$. If there exists a glued edge $r \in E(C)$ such that $G$ does not contain a $k$-glumpkin minor rooted at $r$,
        then $f_\infty(H)\leq g_{\ref{lem:outerplanarBound2}}(k,M)$.
    \end{lemma}
    \begin{proof}
    We proceed by induction on $k$.  The case $k=1$ is vacuous. If $k=2$, then by $2$-connectivity, $r$ is the only glued edge of $G$.  
    Since $G$ is outerplanar, $f_\infty(G) \leq 2$ and so by Lemma~\ref{lem:gluing}, $f_\infty(H) \leq M+1 \leq g_{\ref{lem:outerplanarBound2}}(2,M)$.  Therefore, we may assume $k \geq 3$.  A subpath of $C - r$ is \emph{good} if its ends are connected by a glued edge.   
    Let $P_1, \dots P_p$ be the maximal (under inclusion) good subpaths of $C - r$. Since $G$ is outerplanar, $P_i$ and $P_j$ are internally-disjoint for $i \neq j$. By maximality, every glued edge has both of its ends on some $P_i$.
    
    Let $G_i'$ be the subgraph of $G$ induced by $V(P_i)$.  Let $e_i$ be the glued edge connecting the ends of $P_i$.  Since $G$ does not contain a $k$-glumpkin minor rooted at $r$, $G_i'$ does not contain a $(k-1)$-glumpkin minor rooted at $e_i$. 
    Let $H_i$ be the subgraph of $H$ induced by $G'_i$ and all the graphs $G_j$ that are glued to some edge of $G'_i$. By induction,  $f_\infty(H_i) \leq 3^{k-1}M$ for all $i \in [p]$. Let $C'$ be the cycle obtained from $C$ by replacing $P_i$ with $e_i$ for each $i \in [p]$.  
    Let $G'$ be the subgraph of $G$ induced by the vertices of $C'$.  Notice that $G'$ is a $2$-connected outerplanar graph with outer cycle $C'$, 
    and $H$ can be obtained from $G'$ by gluing the graphs $H_i$ on edges of $C'$. By Lemma~\ref{lem:outerplanarBound},
    \[
    f_\infty(H) \leq 3\cdot \max_{i \in [p]} f_\infty(H_i) \leq 3 \cdot 3^{k-1}M=g_{\ref{lem:outerplanarBound2}}(k,M).  \qedhere
    \]
    \end{proof}  

    We now generalize Lemma~\ref{lem:outerplanarBound2} to $2$-connected treewidth-$2$ graphs.  

  \begin{lemma} \label{lem:gluength}
         For all $k, M \in \NN$, let $g_{\ref{lem:gluength}}(k,M)=3^{k^2}M$.  
        Let $G$ be a $2$-connected treewidth-$2$ graph and let $H$ be obtained by gluing $2$-connected graphs $G_1,\dots, G_m$ on distinct edges of $G$.  Let $M = \max_{i \in [m]}f_\infty(G_i)$.  If for some glued edge $r$, $G$ does not contain a $k$-glumpkin minor rooted at $r$, then $f_\infty(H)\leq g_{\ref{lem:gluength}}(k,M)$.
    \end{lemma}
\begin{proof}
      We proceed by lexicographic induction on $(k, |V(H)|)$.    Let $r$ be a glued edge such that $G$ does not contain a $k$-glumpkin minor rooted at $r$.  
      
      The case $k=1$ is vacuous.  Suppose $k=2$.  Since $G$ is $2$-connected and does not have a $2$-glumpkin minor rooted at $r$, edge $r$ must be the only glued edge of $G$.  Since $G$ is $2$-connected and has treewidth $2$, $f_\infty(G) \leq 2$. By Lemma~\ref{lem:gluing}, $f_\infty(H) \leq M+1 \leq g_{\ref{lem:gluength}}(2,M) $.  Therefore, we may assume $k \geq 3$. 
    If $\deg_H(w)=2$ for some vertex $w \in V(H)$, then we can suppress $w$ by Lemma~\ref{lem:suppressDeg2} and apply induction.  Therefore, we may assume $H$ has minimum degree at least $3$. 
    
      Since $G$ is $2$-connected, there is a cycle in $G$ containing $r$.  
      Let $C$ be a longest cycle in $G$ such that $r \in E(C)$.
      Let $\mathcal{E}$ be an ear decomposition of $G$ beginning with $C$. (See for instance~\cite{Diestel} for background about ear decompositions.)
      The \emph{ear-decomposition tree} $T(\mathcal{E})$ of $\mathcal{E}$ is the rooted tree, whose vertices are the ears in $\mathcal{E}$, defined recursively as follows. The root of $T(\mathcal{E})$ is $C$.   The parent of an ear $P$ is the closest ear $Q$ to $C$ (in $T(\mathcal E)$) such that both ends of $P$ are on $Q$. (Such an ear $Q$ is guaranteed to exist since $G$ has treewidth $2$ and is $2$-connected.) 
      
      Let $P_1, \dots, P_\ell$ be the set of $C$-ears of $\mathcal E$.  Let $T_1, \dots, T_\ell$ be the subtrees of $T(\mathcal E)$ rooted at $P_1, \dots, P_\ell$, respectively.  For each $i \in [\ell]$, let $x_i$ and $y_i$ be the ends of $P_i$ on $C$.  Let $R_i$ be the $x_i$--$y_i$ path in $C$ containing $r$ and let $S_i$ be the other $x_i$--$y_i$ path in $C$. Notice that $|E(S_i)| \geq |E(P_i)|$, by maximality of $C$. If $P_i$ is an edge, then since $G$ is simple, $|E(S_i)| \geq 2$. 
      Otherwise, $|E(S_i)| \geq |E(P_i)| \geq 2$. Therefore, for all $i \in [\ell]$,  $|E(S_i)| \geq 2$.  
      
      We claim that for all $i \in [\ell]$, $V(S_i)$ contains the ends of a glued edge.  Suppose not.  Among all $S_i$ such that $V(S_i)$ does not contain the ends of a glued edge, choose $S_j$ so that $S_j$ is inclusion-wise minimal. 
      Since $G$ has treewidth $2$ and is $2$-connected, for all $i \neq j$, $S_i \subseteq S_j$, $S_j \subseteq S_i$, or $S_i$ and $S_j$ are internally-disjoint.  By the minimality of $S_j$, each internal vertex of $S_j$ has degree $2$ in $H$. 
      However, this contradicts that $H$ has minimum degree at least $3$.  
      
      For each $i \in [\ell]$, let $G_i'$ be the union of all ears in $T_i$ together with the edge $e_i=x_iy_i$, which we declare to be glued.  Since $V(S_i)$ contains the ends of a glued edge and $R_i$ contains $r$, 
      the graph $G_i'$ does not contain a $(k-1)$-glumpkin minor rooted at $e_i$; otherwise, $G$ contains a $k$-glumpkin minor rooted at $r$.  Note that each $G_i'$ contains at least one glued edge other than $e_i$ since $H$ has minimum degree at least $3$.  Let $H_i$ be the graph obtained from $G_i'$  by gluing all $G_j$ such that the glued edge of $G_j$ belongs to $G_i'$.  By induction, $f_\infty(H_i) \leq g_{\ref{lem:gluength}}(k-1,M)$, for all $i \in [\ell]$.  Let $e_{i+1}, \dots, e_{L}$ be the glued edges in $E(C)$.
      
      Observe that $H$ is obtained by gluing graphs $H_1, \dots H_L$ onto  edges of an outerplanar graph $G'$ with outercycle $C$, where $M'=\max_{i \in [L]} f_\infty(H_i)= \max\{M, g_{\ref{lem:gluength}}(k-1,M)\}= g_{\ref{lem:gluength}}(k-1,M)$.  Since $G$ does not contain a $k$-glumpkin minor rooted at $r$, neither does $G'$. Applying Lemma~\ref{lem:outerplanarBound2} to $G'$ gives 
     \[
     f_\infty(H) \leq g_{\ref{lem:outerplanarBound2}}(k, g_{\ref{lem:gluength}}(k-1,M))=3^k(3^{{(k-1)}^2}M) \leq g(k, M). \qedhere
     \] 
     \end{proof}

Lemma~\ref{lem:gluength} yields the following corollary.

 \begin{lemma} \label{lem:gluength2}
         For all $k, M \in \NN$, let $g_{\ref{lem:gluength2}}(k,M)=3^{k^2}M$.  
        Let $G$ be a $2$-connected treewidth-$2$ graph and let $H$ be obtained by gluing $2$-connected graphs $G_1,\dots, G_m$ on distinct edges of $G$. If $H$ does not contain an $\ks_k$ minor and $M = \max_{i \in [m]}f_\infty(G_i)$, then $f_\infty(H)\leq g_{\ref{lem:gluength2}}(k,M)$.
    \end{lemma}
    \begin{proof}
    We proceed by induction on $|V(H)|$.  If $\deg_H(w)=2$ for some $w \in V(H)$, then by Lemma~\ref{lem:suppressDeg2}, we can suppress $w$ and apply induction.  Since $H$ does not contain an $\ks_k$ minor, $G$ does not contain a $k$-glumpkin minor, by Lemma~\ref{lem:glumstar}.  In particular, for each glued edge $r$, $G$ does not contain a $k$-glumpkin minor rooted at $r$.  By Lemma~\ref{lem:gluength}, $f_\infty(H) \leq g_{\ref{lem:gluength}}(k,M)=g_{\ref{lem:gluength2}}(k,M)$.  
    \end{proof}
    
    The following is the main result of this section. 
    
    \begin{lemma}\label{lem:main2conn} 
        Suppose there exist computable functions $g_{\ref{lem:main3con}} : \NN \to \RR$ and $g_{\ref{lem:mainthmaux}} :\NN \times \NN \to \RR$ satisfying the two following conditions.
        \begin{enumerate}
            \item $f_\infty(G)\leq g_{\ref{lem:main3con}}(k)$ for every $3$-connected graph $G$ not containing a $\kall^k$ minor.
            
            \item $f_\infty(H)\leq g_{\ref{lem:mainthmaux}}(k, M)$ for every graph $H$ containing no $\kall^k$ minor, obtained by gluing $2$-connected graphs $G_1,\dots, G_m$ on distinct edges of a $3$-connected graph $G_0$, where $M = \max_{i\in [m]}f_\infty(G_i)$.
        \end{enumerate}
        Then there exists a computable function $g_{\ref{thm:main}}:\NN\to \RR$ such that $f_\infty(G)\leq g_{\ref{thm:main}}(k)$ for all graphs $G$ without a $\kall^k$ minor. 
    \end{lemma}
    \begin{proof}
        We define $g_{\ref{thm:main}}(k)$ as follows.  For all $k,M \in \NN$, let $\alpha(k, M)$ be the maximum of  $g_{\ref{lem:gluength2}}(k,M)$ and $g_{\ref{lem:mainthmaux}}(k,M)$.  Define $\gamma_0(k)=g_{\ref{lem:main3con}}(k)$.  For all $i,k \in \NN$ recursively define $\gamma_i(k)=\alpha(k, \gamma_{i-1}(k))$.  Finally, let $g_{\ref{thm:main}}(k)=\gamma_{6k}(k)$.  
    
        Let $G$ be a graph without a $\kall^k$ minor.  By Lemma~\ref{lem:gluing}, we may assume that $G$ is $2$-connected. By Lemma~\ref{lem:suppressDeg2}, we can assume that $G$ has no degree-$2$ vertices. Let $T_G$ be the SPQR tree of $G$ and let $T = T'_G$ be the contracted SPQR tree, see Lemma~\ref{lem:spqrSimplify}. 
        
        Pick an arbitrary root node $r$ in $T$.  For each node $b$ of $T$, we denote by $T_{b}$ the subtree of $T$ rooted at $b$ and by $H_{b}$ the minor of $G$ corresponding to that subtree. Note that $G=H_r$.  By Lemma~\ref{lem:spqrSimplify}, every leaf of $T$ is an $R$-node. Hence, each leaf $u$ of $T$ corresponds to a $3$-connected minor $H_u$ of $G$.  By our first assumption, $f_\infty(H_u) \leq g_{\ref{lem:main3con}}(k)=\gamma_0(k)$.
        Let $a$ be some inner node of $T$ and let $a_1, \ldots, a_\ell$ denote its children.  Let $M_a = \max_{j\in [\ell]}{f_\infty(H_{a_j})}$. If $a$ is an $O$-node, then by Lemma~\ref{lem:gluength2}, $f_\infty(H_a) \leq g_{\ref{lem:gluength2}}(k,M_a)$. If $a$ is a $R$-node, then $f_\infty(H_a)\leq g_{\ref{lem:mainthmaux}}(k,M_a)$ by our second assumption. 
        In either case, $f_\infty(H_a) \leq \alpha(k, M_a)$.  It follows that if $i$ is the maximum length of an $a$ to leaf path of $T$, then $f_\infty(H_a) \leq \gamma_{i}(k)$.  By Lemma~\ref{lem:spqrSimplify}, the height of $T$ is at most $6k$.  Therefore, $f_\infty(G) = f_\infty(H_r) \leq \gamma_{6k}(k)=g_{\ref{thm:main}}(k)$.  
        \end{proof}
    
    We will establish the existence of $g_{\ref{lem:main3con}}$ and $g_{\ref{lem:mainthmaux}}$ in Lemmas~\ref{lem:main3con} and~\ref{lem:mainthmaux}, respectively. 
    Lemmas~\ref{lem:main2conn}, \ref{lem:main3con}, and~\ref{lem:mainthmaux} and the results from Section~\ref{sec:certificates} together establish  Theorem~\ref{thm:main}, which we now restate: 

    \begin{customthm}{1} 
            There exists a computable function $g_{\ref{thm:main}}: \NN \to \RR$ such that for every $k\in \NN$, every
            graph $G$ with $f_\infty(G) > g_{\ref{thm:main}}(k)$ contains a $\kall^k$ minor. Moreover, every graph $G$ that contains a $\kall^k$ minor has $f_\infty(G) > k$.
    \end{customthm}

    \begin{proof}
    For the first part of the theorem, by Lemmas~\ref{lem:main2conn}, \ref{lem:main3con}, and~\ref{lem:mainthmaux}, there exists a computable function $g_{\ref{thm:main}}:\NN\to \RR$ such that $f_\infty(G)\leq g_{\ref{thm:main}}(k)$ for all graphs $G$ without a $\kall^k$ minor. Thus, every graph $G$ satisfying  $f_\infty(G) > g_{\ref{thm:main}}(k)$ contains a $\kall^k$ minor. 

    For the second part of the theorem, it is shown in Section~\ref{sec:certificates} that each of the four graphs $G$ in $\kall^k$ satisfies  $f_\infty(G) > k$. Since $f_\infty(G)$ is monotone w.r.t.\ minors, it follows that $f_\infty(G) > k$ for every graph $G$ containing a $\kall^k$ minor. 
    \end{proof}

    \section{$3$-connected graphs} \label{sec:3connected}
    
    The results in this section are purely graph theoretical and may be of independent interest.  In particular, we prove several lemmas which give sufficient conditions under which a graph contains some specific graphs as minors.  We also introduce a reduction operation, called \emph{fan-reduction}.  The main result of the section is that if $G$ is a $3$-connected, fan-reduced graph having no $\kall^k$ minor, then the vertex cover number of $G$, $\tau(G)$, is bounded by a function of $k$.  
    
    Before proceeding, we quickly review some graph theoretical terminology.  Let $A, B$ be subsets of vertices of a graph $G$.  An \emph{$A$--$B$ path} is a path $P$ in $G$ such that the ends of $P$ are in $A$ and $B$ respectively, and no internal vertex of $P$ is in $A\cup B$. 
    If $H$ is a subgraph of $G$ then an {\em $H$-path} is a path $P$ in $G$ such that the ends of $P$ are in $H$ but no other vertex nor edge of $P$ is in $H$.
    
    The \emph{$n$-ladder $\kl_n$} is the graph on $2n$ vertices with vertex set $V = \{v_i \mid i \in [n]\} \cup \{w_i \mid i \in [n]\}$ and edge set $E = \{v_iw_i \mid i \in [n]\} \cup \{v_iv_{i+1}, w_iw_{i+1} \mid i \in [n-1]\}$ (see Figure~\ref{fig:ladder}). By repeatedly suppressing degree-$2$ vertices, we can reduce $\kl_n$ to the graph $K_3$. This implies that $f_\infty(\kl_n) = 2$ for all $n\geq 2$ by Lemma~\ref{lem:suppressDeg2}.  
    
      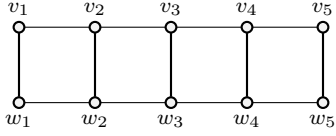
\begin{figure}
        \centering
        \begin{tikzpicture}[x=1cm, y=0.5cm]
            \tikzstyle{vtx}=[circle,draw,thick,fill=black!5]
            \begin{scriptsize}
            \node[vtx, scale=0.5] (x1) at (0,1){};
            \node[vtx, scale=0.5] (x2) at (1,1){};
            \node[vtx, scale=0.5] (x3) at (2,1){};
            \node[vtx, scale=0.5] (x4) at (3,1){};
            \node[vtx, scale=0.5] (x5) at (4,1){};
            
            \node[above] at (0,1.1) {$v_1$};
            \node[above] at (1,1.1) {$v_2$};
            \node[above] at (2,1.1) {$v_3$};
            \node[above] at (3,1.1) {$v_4$};
            \node[above] at (4,1.1) {$v_5$};
            
            \node[vtx, scale=0.5] (y1) at (0,-1){};
            \node[vtx, scale=0.5] (y2) at (1,-1){};
            \node[vtx, scale=0.5] (y3) at (2,-1){};
            \node[vtx, scale=0.5] (y4) at (3,-1){};
            \node[vtx, scale=0.5] (y5) at (4,-1){};
            
            \node[below] at (0,-1.1) {$w_1$};
            \node[below] at (1,-1.1) {$w_2$};
            \node[below] at (2,-1.1) {$w_3$};
            \node[below] at (3,-1.1) {$w_4$};
            \node[below] at (4,-1.1) {$w_5$};
            
            \end{scriptsize}
            
            \draw[thick] (x1)--(y1);
            \draw[thick] (x2)--(y2);
            \draw[thick] (x3)--(y3);
            \draw[thick] (x4)--(y4);
            \draw[thick] (x5)--(y5);

            \draw (x1)--(x2);
            \draw (x2)--(x3);
            \draw (x3)--(x4);
            \draw (x4)--(x5);
            
            \draw (y1)--(y2);
            \draw (y2)--(y3);
            \draw (y3)--(y4);
            \draw (y4)--(y5);
            
        \end{tikzpicture}
        \caption{The ladder $\kl_5$.
        \label{fig:ladder}}
    \end{figure}

    \begin{lemma} \label{lem:ladder}  
        For all $k \in \mathbb{N}$, let $g_{\ref{lem:ladder}}(k) = 12k^2+7k$. If $G$ is a $3$-connected graph containing a $g_{\ref{lem:ladder}}(k)$-ladder as a minor, then $G$ contains $\kn_k$, $\kp_k$, or $\kf_k$ as a minor.
    \end{lemma}
    \begin{proof} 
        Since $\kl_n$ has maximum degree $3$, every graph with an $\kl_n$ minor also contains an $\kl_n$ subdivision.   
        Let $S$ be a subgraph of $G$ isomorphic to a subdivision of $\kl_n$  with $n = g_{\ref{lem:ladder}}(k)$. We say that the vertices of $S$ that do not correspond to internal vertices of a subdivided edge are {\em branch vertices}.  We name these branch vertices $\{v_i \mid i \in [n]\} \cup \{w_i \mid i \in [n]\}$ as in the definition of $\kl_n$ given above.  
        A \emph{rung} is a path in $S$ corresponding to an edge of $\kl_n$ of the form $v_iw_i$, for some $i\in [n]$. 
        We say that an $S$-path $P$ \emph{crosses a rung} $R$, if the ends of $P$ are in different components of $S - V(R)$. A rung is \emph{crossed} if it is crossed by some $S$-path, and is \emph{uncrossed} otherwise. 
        
        If there exists an $S$-path in $G$ that crosses at least $2k+1$ rungs, then $G$ contains an $\kn_{k}$ minor, and we are done. Hence, we may assume that each $S$-path crosses at most $2k$ rungs of $S$.
        
         We say that the path in $S$ from $v_1$ to $v_n$ avoiding all $w_i$ for $i\in [n]$ is the \emph{upper path} of $S$. Similarly the \emph{lower path} is the path in $S$ from $w_1$ to $w_n$ avoiding all vertices $v_i$ for $i\in [n]$.
        For each $i \in \{2, \dots, n-1\}$, let $S_\ell^i$ and $S_r^i$ be the components of $S -\{v_i, w_i\}$ that contain $v_1$ and $v_n$, respectively.  
        
        Suppose there are $8k+1$ uncrossed rungs $R_1,\dots, R_{8k+1}$. For each $i \in [8k+1]$, let $v_{i'}$ and $w_{i'}$ be the ends of $R_i$. We may assume that $i'< j'$ for all $i < j$.  
        Since $G$ is $3$-connected, $G - \{v_{i'}, w_{i'}\}$ is connected.  Therefore, there is a path $P$ in $G - \{v_{i'}, w_{i'}\}$ from $V(S_\ell^{i'})$ to $V(S_r^{i'})$. Since $R_i$ is uncrossed, $P$ must use an internal vertex of $R_i$.  Thus, there exists a vertex $y_i \in V(R_i)\delete \{v_{i'}, w_{i'}\}$ that is connected by an $S$-path $P_i$ to some vertex $z_i \notin V(R_i)$. 
        
        By symmetry and pigeonhole, there is a subset $I$ of size $k$ of $\{2, 4, \dots, 8k\}$ such that $z_i \in V(S_r^{i'})$ and $z_i$ is not on the lower path of $S$, for all $i \in I$. Since $R_i$ is uncrossed for all $i \in [8k+1]$ it follows that $z_i  \in V(S_\ell^{(i+1)'}) \cup V(R_{i+1})$.  For the same reason, $P_i$ and $P_j$ are vertex-disjoint for all distinct $i,j \in I$.  Therefore, $S \cup \bigcup_{i \in I} P_i$ contains an $\kf_k$ minor.  
        
        We may hence assume that $S$ contains at most $8k$ uncrossed rungs.  Thus, $S$ contains at least $n-8k=12k^2-k$ crossed rungs.  Since $12k^2-k=1+(4k+1)(3k-1)$, there is a subset $J$ of $[n]$ of size $3k$ such that for all distinct $i,j \in J$, $|i-j| \geq 4k+1$ and $R_i$ is crossed. For each $i \in J$, let $P_i$ be an $S$-path crossing $R_i$.  Let $\ell_i$ and $r_i$ be the ends of $P_i$ in $S_\ell^i$ and $S_r^i$, respectively.  
        
        We say that $P_i$ is of \emph{type $v$}  if $\ell_i$ and $r_i$ are both on the upper path, \emph{type $w$} if $\ell_i$ and $r_i$ are both on the lower path, and \emph{type $p$} otherwise.  Since $|J|=3k$, there is a subset $J'$ of $J$ of size $k$ such that $P_i$ is of the same type $\mathsf{T}$ for all $i \in J'$.  Recall that each $S$-path crosses at most $2k$ rungs and $|i-j| \geq 4k+1$ for all distinct $i,j \in J'$.  Therefore, if $i,j \in J'$ and $i<j$,
        then $r_i$ is to the left of $\ell_{j}$.  Moreover, for the same reason, $P_i$ and $P_j$ are vertex-disjoint for all distinct $i,j \in J'$.  Therefore, $S \cup \bigcup_{i \in J'} P_i$ contains an $\kf_k$ minor if $\mathsf{T} \in \{v,w\}$ and $S \cup \bigcup_{i \in J'} P_i$ contains a $\kp_k$ minor if $\mathsf{T}=p$.  
        \end{proof}
    
    For each $k \in \NN$, the \emph{$k$-fan} is the graph consisting of a $k$-vertex path called its \emph{outer path},  plus a universal vertex called its {\em center}.  The edges connecting the center to the ends of the $k$-vertex path are called the {\em boundary edges} of the $k$-fan. 
    A {\em fan} is a graph isomorphic to a $k$-fan for some $k$. 
  
    Let $H$ be a fan, and assume that $G$ has an $H$-model. We say that the $H$-model is {\em rooted at} $x,y$ if $x$ and $y$ are contained in the vertex images of vertices $a$ and $b$ of $H$, respectively, and $ab$ is a boundary edge of the fan.  
    
    \begin{lemma} \label{claim}
    For all $k,q \in \NN$, let $g_{\ref{claim}}(k, q) = 3(8k^3)^q$. 
    Let $G$ be a graph and let $P=p_1 \cdots p_r$ be a path in $G$ of length at least $g_{\ref{claim}}(k, q)$ such that $V(G) \delete V(P)$ is a stable set. 
    Then at least one of the following holds: 
    \begin{enumerate}
    \item $G$ has a $k$-fan minor;
    \item there is a model of the $q$-fan
    in $G$ rooted at $p_2, p_{r-1}$ and avoiding $p_1, p_r$; 
    \item there are non-consecutive indices $s, t$ with $1 < s < t < r$ such that $\{p_s,p_t\}$ separates in $G$ the $p_s$--$p_t$ subpath of $P$ from the other vertices of $P$. 
    \end{enumerate}
    \end{lemma} 
    \begin{proof}
    The proof is by induction on $q$. 
    For the base case $q = 1$, observe $g_{\ref{claim}}(k, 1) \geq 24$, for all $k \in \NN$. 
    Thus, it suffices to take $p_2$ and the $p_3$--$p_{r-1}$ subpath of $P$ as the two vertex images to obtain a model of the $1$-fan rooted at $p_2, p_{r-1}$ and avoiding $p_1, p_r$. 
    
    For the inductive step, assume $q > 1$. Let $S=V(G) \delete V(P)$.
    We may assume that every vertex in $S$ has degree at most $k-1$ in $G$, since otherwise there is a $k$-fan minor in $G$. 
    Note that $g_{\ref{claim}}(k, q) = 8k^3 \cdot g_{\ref{claim}}(k, q-1)$. 
    A {\em jump} is a pair $(a, b)$ of indices $a,b\in [r]$ with $b \geq a+2$ such that either $p_{a}p_{b} \in E(G)$ ({\em type~1}) or $p_{a}$ and $p_{b}$ have a common neighbor in $S$ ({\em type~2}). 
    For definiteness, if both conditions are satisfied then $(a, b)$ is considered to be of type 1. 
    To each jump $(a, b)$ of type~2 we associate a corresponding {\em middle vertex} $w\in S$ adjacent to both $a$ and $b$, that is chosen arbitrarily.
    A jump $(a, b)$ is called an {\em outer jump} if $a = 1$ or $b = r$; otherwise, $(a,b)$ is an {\em inner jump}. 
    In what follows we will be mostly interested in inner jumps. \\

    {\bf Case~1: There exists an inner jump $(a, b)$ with $b-a \geq k \cdot g_{\ref{claim}}(k, q-1)$.} 
    Let $(a, b)$ be such a jump. 
    If $(a, b)$ is of type~2, we first modify it as follows. 
    Let $w$ be the middle vertex of $(a, b)$.  
    Since $w$ has degree at most $k-1$, it follows that there exists a jump $(a', b')$ with $b'-a' \geq k \cdot g_{\ref{claim}}(k, q-1) / (k-2) \geq g_{\ref{claim}}(k, q-1)$ such that $w$ is adjacent to $p_{a'}$ and $p_{b'}$ but to no vertex lying strictly in between them on $P$. 
    We rename $(a',b')$ to $(a, b)$. 

    Let $G'$ be the minor of $G$ obtained by contracting the $p_1$--$p_a$ subpath of $P$ into $p_a$ and the $p_b$--$p_r$ subpath of $P$ into $p_b$. 
    Let $P'$ be the path obtained from $P$ by performing these contractions. We regard $p_a$ and $p_b$ as the ends of $P'$.   
    Note that $V(G') \delete V(P')$ is a stable set in $G'$. 
    Since $P'$ has length $b-a \geq g_{\ref{claim}}(k, q-1)$, by induction at least one of the following holds: 
    \begin{enumerate}
    \item $G'$ has a $k$-fan minor; 
    \item there is a model $\mathcal M'$ of the $(q-1)$-fan in $G'$
    rooted at $p_{a+1}, p_{b-1}$ and avoiding $p_a, p_b$;
    \item there are non-consecutive indices $s, t$ with $a < s < t < b$ such that $\{p_s,p_t\}$ separates in $G'$ the $p_s$--$p_t$ subpath of $P'$ from the other vertices of $P'$.  
    \end{enumerate}
    
    In the first case, we are done since $G'$ is a minor of $G$. 
    In the second case, $\mathcal M'$ is also such a model in $G$ since the two subpaths that were contracted in the definition of $G'$ resulted in vertices $p_a, p_b$. By symmetry, we may assume that the vertex image $V_0$ corresponding to the center of the fan contains $p_{a+1}$.

    Recall that $2 \leq a < b \leq r-1$, since $(a,b)$ is an inner jump. Let $L$ and $R$ be the 
    $p_2$--$p_a$ and $p_b$--$p_{r-1}$ subpaths of $P$, respectively.  Let $w$ be the middle vertex of $(a,b)$ if $(a,b)$ is type~2. Let $R'=R$ if $R$ is type~1, and $R'=R \cup \{w\}$ if $(a,b)$ is type~2. In either case, observe that $L$ and $R'$ are connected by an edge. By construction, $V(L) \cup V(R)$ is disjoint from all vertex images of $\mathcal M'$.  Since $w$ is not adjacent to any internal vertex of $P'$, $\{w\}$ is also disjoint from all vertex images of $\mathcal M'$. 
    Finally, the edges $p_ap_{a+1}$ and $p_{b-1}p_{b}$ connect $V(L)$ and $V(R)$ to the vertex images of $\mathcal M'$ containing $p_{a+1}$ and $p_{b-1}$, respectively.  Therefore, $(\mathcal M' \delete \{V_0\}) \cup \{V_0 \cup L, R'\}$ is a model of the $q$-fan in $G$ rooted at $p_2, p_{r-1}$ and avoiding $p_1, p_r$, as desired. 
    
    It remains to consider the third case. Suppose $s, t$ are non-consecutive indices with $a < s < t < b$ such that $\{p_s,p_t\}$ separates in $G'$ the $p_s$--$p_t$ subpath of $P'$ from the other vertices of $P'$.  
    Given how $G'$ was obtained from $G$, this is also true in $G$. That is, $\{p_s,p_t\}$ separates in $G$ the $p_s$--$p_t$ subpath of $P$ from the other vertices of $P$, as desired.   \\
    
    {\bf Case~2: $b-a < k \cdot g_{\ref{claim}}(k, q-1)$ for all inner jumps $(a, b)$.} 
    Let us introduce one more definition. 
    A {\em jump sequence} is a sequence $(a_{1}, b_{1}), \dots, (a_{\ell}, b_{\ell})$ of inner jumps 
    with $\ell \geq 1$ satisfying $a_{i} < a_{i+1} < b_{i} < b_{i+1}$ for each $i\in [\ell-1]$, and $b_{i} \leq a_{i+2}$ for each $i\in [\ell-2]$. 
    Its {\em length} is $\ell$ and its {\em spread} is $b_{\ell} - a_{1}$. \\
    
    {\bf Case~2.1: There exists a jump sequence of spread at least $2k^2 \cdot g_{\ref{claim}}(k, q-1)$.} 
    Let $(a_{1}, b_{1}), \dots, (a_{\ell}, b_{\ell})$ be a jump sequence of spread at least $2k^2 \cdot g_{\ref{claim}}(k, q-1)$ and with $\ell$ {\em minimum}. 
    For each $i\in [\ell]$, if $(a_{i}, b_{i})$ is of type~2 let $w_{i} \in S$ be the middle vertex of $(a_{i}, b_{i})$. 
    
    We claim that all middle vertices $w_{i}$ defined above are distinct. 
    Indeed, assume $w_{i} = w_{j}$ for some $i, j \in [\ell]$ with $i < j$. Then $(a_{i}, b_{j})$ is also an inner jump, and $(a_{1}, b_{1}), \dots, (a_{i-1}, b_{i-1}), (a_{i}, b_{j}), (a_{j+1}, b_{j+1}), \dots, (a_{\ell}, b_{\ell})$ is a jump sequence, as the reader can easily check. 
    But the latter jump sequence has length at most $\ell-1$ and yet its spread is also $b_{\ell}-a_{1}$, contradicting our choice of the original jump sequence. 
    
    Since $b_{i} - a_{i} \leq k \cdot g_{\ref{claim}}(k, q-1)$ for each $i\in [\ell]$, we have 
    \[
    2k^2 \cdot g_{\ref{claim}}(k, q-1) \leq b_{\ell} - a_{1} \leq \sum_{i\in [\ell]} (b_{i} - a_{i}) \leq \ell k \cdot g_{\ref{claim}}(k, q-1), 
    \]
    implying $\ell \geq 2k$. 
    Now, one can obtain a $k$-fan-model using the jump sequence $(a_{1}, b_{1}), \dots, (a_{2k}, b_{2k})$ as illustrated in Figure~\ref{fig:jumps}. 
    
    \begin{figure}
        \centering
        \includegraphics[width=\textwidth]{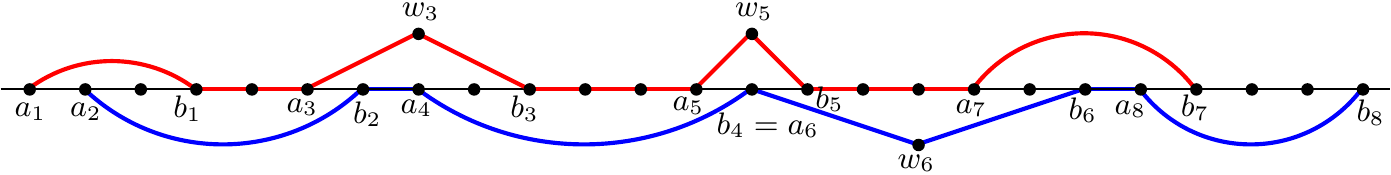}
        \caption{Illustration of a $k$-fan-model obtained from a jump sequence $(a_{1}, b_{1}), \dots, (a_{2k}, b_{2k})$ for $k = 4$. The blue path is the vertex image for the center of the fan, and the red path corresponds to the outer path. Edges incident to the center of the fan map to the first edge of the subpath of $P$ from $a_{2i}$ to $b_{2i-1}$.\label{fig:jumps}}
    \end{figure}
    
    {\bf Case~2.2: All jump sequences have spread less than $2k^2 \cdot g_{\ref{claim}}(k, q-1)$.} 
    Let 
    \[
    M = \{2, r-1\} \cup \{i\in [r]\mid (1,i) \textrm{ is an outer jump}\} \cup \{i\in [r]\mid (i,r) \textrm{ is an outer jump}\}. 
    \]
    If there are $k$ outer jumps of the form $(1,i)$ then $G$ has a $k$-fan minor, and the same is true for those of the form $(i,r)$. 
    Thus we may assume that $|M| \leq 2k$. 
    By the pigeonhole principle, there are two indices $i,j\in M$ with $i < j$ and $M \cap [i+1, j-1] = \emptyset$ such that 
    \[
    j-i \geq \frac{r - 1}{|M|-1} 
    \geq \frac{g_{\ref{claim}}(k, q)}{2k} 
    = 4k^2 \cdot g_{\ref{claim}}(k, q-1).
    \] 
    If there exists an inner jump $(a, b)$ with $a < i < b$, let $(a_{1}, b_{1}), \dots, (a_{\ell}, b_{\ell})$ be a jump sequence such that $a_{1} < i < b_{1}$ and maximizing its spread, and let $s = b_{\ell}$. 
    If no such jump exists, simply let $s = i$. 
    
    We claim that there is no inner jump $(a, b)$ with $a < s < b$. 
    This is obviously true if $s = i$, so assume $s \neq i$, and consider the corresponding jump sequence $(a_{1}, b_{1}), \dots, (a_{\ell}, b_{\ell})$ defined above. 
    Arguing by contradiction, suppose that there is an inner jump $(a, b)$ with $a < s < b$. 
    If $a \leq a_{1}$ then $(a, b)$ is a jump sequence with $a < i < b$ and spread $b - a > b_{\ell} - a_{1}$, contradicting our choice of the jump sequence. 
    If $a_{1} < a$ then letting $\ell' \in [\ell]$ be the smallest index such that $a < b_{\ell'}$ (which is well defined since $a < b_{\ell}$), we deduce that $(a_{1}, b_{1}), \dots,  (a_{\ell'}, b_{\ell'}), (a,b)$ is a jump sequence with $a_{1} < i < b_{1}$ and of spread $b - a_{1} > b_{\ell} - a_{1}$, again a contradiction. 
    Hence, no inner jump $(a, b)$ with $a < s < b$ exists, as claimed. 
    
    Next, if there exists an inner jump $(a, b)$ with $a < j < b$, let $(a'_{1}, b'_{1}), \dots, (a'_{\ell'}, b'_{\ell'})$ be a jump sequence such that $a'_{\ell'} < j < b'_{\ell'}$ and maximizing its spread, and let $t = a'_{1}$. If no such jump exists, simply let $t = j$. 
    By a symmetric argument, there is no inner jump $(a, b)$ with $a < t < b$. 
    
    Recall that every jump sequence has spread strictly less than $2k^2 \cdot g_{\ref{claim}}(k, q-1)$. 
    Thus, $s - i \leq 2k^2 \cdot g_{\ref{claim}}(k, q-1) - 1$ and $j - t \leq 2k^2 \cdot g_{\ref{claim}}(k, q-1) - 1$. 
    It follows that 
    \[ 
    t - s \geq j-i - 4k^2 \cdot g_{\ref{claim}}(k, q-1) + 2 \geq 2. 
    \]
    In other words, $[s+1, t-1]$ is not empty. 
    Since $[s+1, t-1] \subseteq [i+1, j-1]$ and $M \cap [i+1, j-1] = \emptyset$, there is no outer jump $(1, b)$ with $b\in [s+1, t-1]$ and  there is no outer jump $(a, r)$ with $a\in [s+1, t-1]$. 
    Since we already established that there is no inner jump $(a, b)$ with $a < s < b$ or $a < t < b$, we deduce that the two indices $s, t$ satisfy the third outcome of the claim.  That is,  $s$ and  $t$ are non-consecutive indices with $1 < s < t < r$ such that $\{p_s,p_t\}$ separates in $G$ the $p_s$--$p_t$ subpath of $P$ from the other vertices of $P$.
    \end{proof}

      As an easy corollary of Lemma~\ref{claim}, we obtain the following strengthening of Lemma~4.7 in~\cite{JPSST14}.\footnote{The latter lemma works under the assumption that $G$ does not have the graph consisting of two vertices linked by $k$ parallel edges as a minor, which is more restrictive than just forbidding a $k$-fan minor.
    Nevertheless, the two proofs are based on a similar strategy. } 

    \begin{lemma} \label{lem:boundeddegladder}
    For all $k \in \NN$, let $g_{\ref{lem:boundeddegladder}}(k) = 3(8k^3)^{k}$.
    Let $G$ be a graph with no $k$-fan minor.  Let $P$ be a path in $G$ of length at least $g_{\ref{lem:boundeddegladder}}(k)$ such that $V(G) \delete V(P)$ is a stable set. 
    Then there exist two non-consecutive internal vertices $u,v$ of $P$ such that $\{u,v\}$ separates in $G$ the $u$--$v$ subpath of $P$ from the other vertices of $P$. 
    \end{lemma}
    \begin{proof}
    Note that $g_{\ref{lem:boundeddegladder}}(k)=g_{\ref{claim}}(k,k)$.  The lemma follows by applying Lemma~\ref{claim} to $G$ and $P$, and noting that the first two outcomes of Lemma~\ref{claim} are impossible since $G$ has no $k$-fan minor.  
    \end{proof}
    
    Next, we introduce two lemmas about $3$-connected graphs containing subdivisions of large fans as subgraphs. 
    Given a graph $G$, we say that $F$ is a {\em fan subdivision in $G$} if $F$ is a subgraph of $G$ isomorphic to a subdivision of a fan. 
    Moreover, we say that $F$ is a \emph{maximal} fan subdivision in $G$ if $F$ is maximal with respect to subgraph inclusion. That is, for every fan subdivision $F'$ in $G$ such that $F \subseteq F' \subseteq G$, we have $F = F'$.
    
    \begin{lemma}\label{lem:subdivided_fan}
        For all $k \in \mathbb N$, let
       $g_{\ref{lem:subdivided_fan}}(k) = 8k^4+4k^3+10k$.
        If $G$ is a $3$-connected graph and $F$ is a maximal fan subdivision in $G$ such that at least $g_{\ref{lem:subdivided_fan}}(k)$ of the edges of the fan are subdivided, then 
        $G$ has an  $\kl_k$, $\ks_k$ or $\kf_k$ minor. 
    \end{lemma}
    \begin{proof}
        Let $F^*$ denote the $m$-fan such that $F$ is a subdivision of $F^*$, where $v_0$ is the center of $F^*$ and $v_1 \cdots v_m$ is the outer path of $F^*$.
        
        In the following we consider the graph $H$ obtained from $G$ by performing the following two operations. First, we contract each component of $G - V(F)$ into a vertex. Second, for each edge $e$ of $F^*$ that is subdivided at least once in $F$, we contract the corresponding path $P$ of $F$ into a $2$-edge path, that is, we leave just one subdivision vertex. 
        We call this subdivision vertex $v_i^1$ if $e=v_0v_i$ for some $i\in [m]$, and $v_i^2$ if $e=v_iv_{i+1}$ for some $i\in [m-1]$.
        
        Hence, each vertex of $H$ is of the form $v_i, v_i^1, v_i^2$, or results from the contraction of a component of $G - V(F)$. We denote by $F'$ the fan subdivision in $H$ that is the image of $F$, that is, which is obtained from $F$ by the above contractions.  Observe that $F'$ is a {\em maximal} fan subdivision in $H$. 
        Indeed, if some fan subdivision in $H$ strictly contained $F'$ then that fan subdivision could be mapped to a fan subdivision in $G$ strictly containing $F$, contradicting the maximality of $F$.
        
        We will establish the following key property of $H$:
       \begin{quote} \em $(\star)$ If $u_i$ is a vertex of $H$ of the form $v_i^1$ or $v_i^2$, then there is an $F'$-path $P_i$ in $H$ of length at most $2$ connecting $u_i$ to another vertex $u'_i$ of $F'$ distinct from its two neighbors in $F'$ and from $v_0$. 
       \end{quote}
        
        Suppose $(\star)$ does not hold for some $v_i^1$.
        Then $\{v_0, v_i\}$ is a size-$2$ cutset of $H$ separating $v_i^1$ from every vertex $v_j$ with $j \notin \{0, i\}$ (here we implicitly use that $m\geq 2$, since $F^*$ has at least $g_{\ref{lem:subdivided_fan}}(k) \geq 2$ edges). 
        By the construction of $H$, the set $\{v_0, v_i\}$ is also a cutset of $G$ separating $v_i^1$ from every vertex $v_j$ with $j \notin \{0, i\}$. 
        However, this contradicts the fact that $G$ is $3$-connected. 
        
        The remaining case is if $(\star)$ does not hold for some $v_i^2$. 
        Here we first observe that $v_i^2$ is not adjacent to $v_0$ in $H$, because otherwise this would contradict the maximality of $F'$ in $H$. 
        For the same reason, there is no length-$2$ path from $v_i^2$ to $v_0$ in $H$ going through a vertex in $V(H)\delete V(F')$. 
        Using these two observations, we can proceed similarly as in the proof for $v_i^1$. 
        This concludes the proof of $(\star)$.

        Now, we color each edge of $F'$ blue, and each remaining edge of $H$ red. 
        Consider the graph $H^*$ obtained from $H$ as follows.  
        Every edge of the form $v_i^1v_i$ is contracted to the vertex $v_i$, every edge of the form $v_i^2v_i$ is contracted to the vertex $v_i$, and finally, for every vertex $w \in V(H)\delete V(F')$, we select a neighbor of $w$ distinct from $v_0$ in the current graph (which exists) and contract the corresponding edge. 
        Finally, we delete all red edges incident to $v_0$. 
        Loops and parallel edges resulting from edge contractions are deleted as always, but if a red edge parallel to a blue edge is created, we keep the blue edge and delete the red edge. 
        Thus, the blue subgraph of $H^*$ is exactly the fan $F^*$. 
        Let $R^*$ denote the red subgraph of $H^*$. We regard $R^*$ as a spanning subgraph of $H^*$, and thus $R^*$ may have isolated vertices.   
        
        If $R^*$ has a vertex of degree at least $2k+1$, then that vertex is not $v_0$ (since $v_0$ is not incident to any red edge), and it is then easily seen that $H^*$ has an $\ks_k$ minor. 
        Thus we may assume that the maximum degree of $R^*$ is at most $2k$.

        If $R^*$ has a matching of size $k^3$, then by Pigeonhole and Erd\H{o}s-Szekeres~\cite{ES35}, $R^*$ has a matching $M=\{v_{a_i}v_{b_i}: i \in [k]\}$ of size $k$ that satisfies one of the following three conditions:
        \begin{enumerate}
            \item $a_1 < a_2 < \cdots < a_k < b_1 < b_2 < \cdots < b_k$, or 
            \item $a_1 < a_2 < \cdots < a_k < b_k < b_{k-1} < \cdots < b_1$, or 
            \item $a_1 < b_1 < a_2 < b_2 < \cdots < a_k < b_k$.
        \end{enumerate}
        In the first two cases, we see that $H^*$ has an $\kl_k$ minor (obtained by combining $M$ with the $v_{a_1}$--$v_{a_k}$ and $v_{b_1}$-- $v_{b_k}$ subpaths of the outer path of $H^*$). 
        In the third case, we see that $H^*$ has an $\kf_k$ minor. 
        Hence we may assume that $R^*$ has no matching of size $k^3$. 
        
        It follows that $R^*$ has a vertex cover of size at most $2k^3$. 
        However, since $R^*$ has maximum degree at most $2k$, it follows in turn that at most $2k^3(2k+1)$ vertices of $R^*$ have non-zero degrees in $R^*$. 
        
        Recall that $v_i^1$ and $v_i^2$ (if they exist) are the only $2$ vertices of $F'$ that are contracted to $v_i$ in $F^*$.  Since $F^*$ has at least $g_{\ref{lem:subdivided_fan}}(k)$ edges that are subdivided in $F'$ and $g_{\ref{lem:subdivided_fan}}(k)/2 - 2k^3(2k+1) = 5k$, there exists $I \subseteq [m]$ with $|I|=k$ such that the following holds: 
        \begin{itemize}
            \item there is a vertex $u_i$ of the form $v_i^1$ or $v_i^2$ in $H$, for each $i\in I$; 
            \item $v_i$ has degree $0$ in $R^*$ for all $i\in I$, and 
            \item $|i-j|\geq 5$ for all $i, j\in I$ with $i\neq j$. 
        \end{itemize}
        Now, consider an index $i\in I$ and its associated subdivision vertex $u_i$ in $H$. 
        By $(\star)$, there is an $F'$-path $P_i$ in $H$ of length at most $2$ connecting $u_i$ to another vertex $u'_i$ of $F'$ distinct from its two neighbors in $F'$ and from $v_0$. 
        The (one or two) edges of $P_i$ are red and are not incident to $v_0$, and they disappeared in the edge contraction operations leading to the graph $H^*$. 
        It follows that $u'_i$ is very close to $u_i$ in $F'- v_0$, namely $u'_i$ must be one of $v_{i-1}, v_{i+1}$, or one of the subdivision vertices $v_{i-1}^1, v_{i+1}^1, v_{i-1}^2, v_{i}^2,  v_{i+1}^2$ (if they exist).  
        
        Since the paths $P_i$ and $P_j$ are vertex disjoint for all $i, j\in I$ with $i\neq j$ (which follows from the fact that $v_i$ and $v_j$ have degree $0$ in $R^*$), and since $|i-j|\geq 5$, combining $F'$ with these $k$ paths we can see that $H$ contains an  $\kf_k$ minor. 
        \end{proof}
    
    Let $F$ be an $m$-fan with center $v_0$ and outer path $v_1 \cdots v_m$.  
    Suppose that $F$ is a subgraph of a graph $G$. 
    We say that $F$ is \emph{reducible in $G$} if $m \geq 5$ and all vertices $v_2,\dots, v_{m-1}$ have degree exactly $3$ in $G$. The \emph{$F$-reduction} of $G$ is the minor of $G$ obtained by contracting the edges of the path $v_3 \cdots v_{m-1}$.  
    Thus, the resulting graph has $m-4$ fewer vertices than $G$.  
    
    A reducible fan subgraph in $G$ is said to be {\em maximal in $G$} if it is not a proper subgraph of any other reducible fan subgraph of $G$. 
    Observe that if $F_1$ and $F_2$ are two distinct maximal reducible fan subgraphs of $G$ then $F_1$ and $F_2$ are almost vertex disjoint in the following sense: $F_2$ contains none of the internal vertices of the outer path of $F_1$, and vice versa.  
    We define the \emph{fan-reduction} of $G$ as the minor of $G$ obtained by simultaneously performing all $F$-reductions for all maximal reducible fan subgraphs $F$ of $G$.  
    By the previous observation, this minor is well-defined. 
    We say that $G$ is \emph{fan-reduced} if $G$ does not contain a reducible fan subgraph. 
    Observe that the fan-reduction of $G$ is fan-reduced.  
    
    \begin{lemma}\label{lem:non_subdivided_fan} 
        For all $k \in \mathbb N$, let $g_{\ref{lem:non_subdivided_fan}}(k) = 20k^5 + 14 k^4 + 2k^3 + 5k$.
        If $G$ is a $3$-connected fan-reduced graph containing a  $g_{\ref{lem:non_subdivided_fan}}(k)$-fan as a subgraph, then $G$ contains an $\ks_k, \kf_k$ or $\kl_k$ minor. 
    \end{lemma}
    \begin{proof}
        Consider an $m$-fan subgraph $F$ in $G$ with center $v_0$, outer path $v_1 \cdots v_m$, and $m = g_{\ref{lem:non_subdivided_fan}}(k)$. Let $H$ be obtained from $G$ by contracting each component of $G - V(F)$ into a vertex. 
        We color the edges of $F$ blue and the remaining edges of $H$ red as in the proof of Lemma~\ref{lem:subdivided_fan}, and define $H^*$ in exactly the same way. 
        The only difference here is that no edge of $F$ needs to be contracted since $F$ is already a fan. In the notation used in the proof of Lemma~\ref{lem:subdivided_fan}, here we have $F=F'=F^*$.
        Let $R^*$ denote the red spanning subgraph of $H^*$. 
        
        If $R^*$ has a vertex of degree at least $2k+1$ or a matching of size $k^3$, then we find one of our target minors, exactly as in the proof of Lemma~\ref{lem:subdivided_fan}. 
        Thus we may assume that this does not happen, implying that at most $2k^3(2k+1)$ vertices of $R^*$ have non-zero degrees in $R^*$. 
        
        Since $(m-2k^3(2k+1)) / (2k^3(2k+1) + 1) \geq 5k$  there is an index $i\in [m-5k]$ such that none of $v_{i+1}, \dots, v_{i+5k}$ is incident to a red edge in $H^*$. 
        For each $\ell \in [k]$, 
        there must be an index $j \in \{i+5(\ell-1)+2,i+5(\ell-1)+3,i+5(\ell-1)+4\}$ such that $v_j$ is incident to a red edge of $H$.  Otherwise, $v_{i+5(\ell-1)+1}, \dots, v_{i+5(\ell-1)+5}$ together with $v_0$ form a reducible fan in $G$. 
        Since all red edges incident to $v_j$ in $H$ disappeared when constructing $H^*$, it follows that $v_j$ is adjacent in $H$ to a vertex $w_{\ell} \in V(H) \delete V(F)$ such that the neighbors of $w_{\ell}$ in $H$ are a subset of $\{v_0, v_{j-1}, v_j, v_{j+1}\}$. 
        Furthermore, $w_{\ell}$ must be adjacent to at least three of these four vertices, since otherwise $G$ would not be $3$-connected. 
        Now, combining $F$ with the $k$ vertices $w_1, \dots, w_k$ we see that $H$ contains an $\kf_k$ minor. 
        \end{proof}
    
    Combining the two previous lemmas, we obtain the following lemma.
    
     \begin{lemma}
        \label{lem:no_subdivision_of_big_fan}
        For all $k \in \NN$, let $g_{\ref{lem:no_subdivision_of_big_fan}}(k)  = g_{\ref{lem:non_subdivided_fan}}(k)(g_{\ref{lem:subdivided_fan}}(k) + 1) + g_{\ref{lem:subdivided_fan}}(k)$.
        If $G$ is a $3$-connected, fan-reduced graph containing a subdivision of a $g_{\ref{lem:no_subdivision_of_big_fan}}(k)$-fan as a subgraph, then $G$ has an  $\ks_k, \kf_k$ or $\kl_k$ minor.  
    \end{lemma}
    \begin{proof}
    Since $G$ contains a $g_{\ref{lem:no_subdivision_of_big_fan}}(k)$-fan subdivision, $G$ contains a maximal $m$-fan subdivision $F$ with $m\geq g_{\ref{lem:no_subdivision_of_big_fan}}(k)$.
    If at least $g_{\ref{lem:subdivided_fan}}(k)$ edges of the $m$-fan are subdivided in $F$, then, by Lemma~\ref{lem:subdivided_fan}, $G$ contains an $\kl_k, \ks_k$ or $\kf_k$ minor.
    Otherwise, $F$ contains an $m'$-fan as a subgraph with $m' \geq (g_{\ref{lem:no_subdivision_of_big_fan}}(k) - g_{\ref{lem:subdivided_fan}}(k)) / (g_{\ref{lem:subdivided_fan}}(k) + 1) = g_{\ref{lem:non_subdivided_fan}}(k)$, and by Lemma~\ref{lem:non_subdivided_fan}, $G$ contains an $\kl_k, \ks_k$ or $\kf_k$ minor.
    \end{proof}
    
    The next lemma is standard, we include the proof nevertheless for completeness. 
    
     \begin{lemma}
            \label{lem:fans_ladders}
            For all $k \in \NN$, let $g_{\ref{lem:fans_ladders}}(k) = k^{k^2+2}$. 
            If $G$ is a graph with a $g_{\ref{lem:fans_ladders}}(k)$-fan 
            minor, then $G$ contains a subdivision of a $k$-fan as a subgraph, or $G$ contains an $\kl_k$ minor. 
    \end{lemma}
    \begin{proof}
    Let $G$ be a graph containing an $m$-fan $F$ as 
    minor with $m = g_{\ref{lem:fans_ladders}}(k)$. 
    Let  $v_0$ be the center of $F$ and  $v_1 \cdots  v_m$ be the outer path. Let $\{X_{i} \mid i \in \{0,1,\dots, m\}\}$  denote an $F$-model in $G$, with $X_{i}$ denoting the vertex image of $v_i$.  
    
    For every edge $v_iv_j$ of $F$ we choose vertices $x_i^j, x_j^i$ of $X_i, X_j$, respectively, such that $x_i^jx_j^i\in E(G)$. 
    Let $T$ be a subtree of $G[X_0\cup\{x_i^0\mid i\in [m]\}]$ such that the leaves of $T$ are exactly the vertices $x_i^0$ for $i\in [m]$. 
    If $T$ contains a vertex of degree at least $k$, then $G$ contains a subdivision of a $k$-fan. 
    Thus we may assume that $T$ has maximum degree less than $k$. 
    
    Now, suppress all degree-$2$ vertices in $T$, giving a tree $T'$. 
    Thus every non-leaf vertex of $T'$ has degree between $3$ and $k-1$ in $T'$. 
    In particular, $k \geq 4$.
    Choose an arbitrary non-leaf vertex $r$ of $T'$.  
    Since $T'$ has $m \geq (k-1)^{k^2+2}$ leaves and maximum degree at most $k-1$, it follows that there is a leaf of $T'$ at distance at least $\log_{k-1} |T'| - 1 \geq \log_{k-1}{(k-1)^{k^2+2}} - 1= k^2 +1$ from $r$ in $T'$.   
    
    Consider the path $P'$ of $T'$ from $r$ to that leaf, minus the leaf, and let $P$ denote the corresponding path of $T$.  
    By construction, there are $k^2$ vertex-disjoint   $V(P)$--$\{x_i^0 \mid i\in [m]\}$ paths in the graph $G[X_0\cup\{x_i^0\mid i\in [m]\}]$. 
    Applying Erd\H{o}s-Szekeres we then find an $\kl_k$ minor in $G$.    
    \end{proof}
    
    \begin{lemma}\label{lem:shortpath}
            For all $k \in \NN$, let $g_{\ref{lem:shortpath}}(k) = g_{\ref{lem:boundeddegladder}}(g_{\ref{lem:fans_ladders}}(g_{\ref{lem:no_subdivision_of_big_fan}}(g_{\ref{lem:ladder}}(k))))$.
            If $G$ is a $3$-connected, fan-reduced graph with no $\kall^k$ minor, then the maximum length of a path in $G$ is at most $g_{\ref{lem:shortpath}}(k)$. 
    \end{lemma}
    \begin{proof} 
        By Lemmas~\ref{lem:fans_ladders},~\ref{lem:no_subdivision_of_big_fan} and~\ref{lem:ladder}, we deduce that $G$ has no $m$-fan minor, where $m = g_{\ref{lem:fans_ladders}}(g_{\ref{lem:no_subdivision_of_big_fan}}(g_{\ref{lem:ladder}}(k)))$.
        Arguing by contradiction, suppose $G$ has a path $P$ of length more than $g_{\ref{lem:shortpath}}(k) = g_{\ref{lem:boundeddegladder}}(m)$.
        
        Let $C_1, \dots, C_p$ denote the components of $G-V(P)$. 
        Let $H$ be the graph obtained from $G$ by contracting each component $C_i$ into a vertex $c_i$. 
        Note that $H$ has no $m$-fan minor, since $H$ is a minor of $G$. 
        By Lemma~\ref{lem:boundeddegladder}, applied to the graph $H$ and path $P$, there exist two non-consecutive internal vertices $u,v$ of $P$ such that $\{u,v\}$ separates in $H$ the $uv$-subpath of $P$ from the other vertices of $P$. 
        However, the same remains true in $G$, by construction of $H$.
        Therefore, $\{u,v\}$ is a cutset of $G$, contradicting the fact that $G$ is $3$-connected.
    \end{proof}
    
     In the following we will use another reduction operation for $3$-connected graphs. Let $G$ be a $3$-connected graph and let $h\geq 3$ be a fixed integer. Let $T_1,\dots, T_\ell$ be an enumeration of all stable sets of $G$ satisfying the following conditions for each $i \in [\ell]$, 
     \begin{itemize}
         \item $|T_i|\geq h+1$, 
         \item there exists $S_i \subseteq V(G)$ with $|S_i|\leq h$ such that for all $v\in T_i$, the set of neighbors of $v$ in $G$ is exactly $S_i$, 
         \item 
         $T_i$ is inclusion-wise maximal with respect to the above two properties.  
     \end{itemize}
    Observe that by maximality, the sets $T_1,\dots, T_\ell$ are pairwise disjoint.
    Let $G'$ be the graph obtained from $G$ by removing all vertices in $T_i$ except $h+1$ of them, for each $i\in [\ell]$.  Clearly, $G'$ does not depend on which $h+1$ vertices remain in each $T_i$.    
    We call $G'$ the \emph{$h$-reduction} of $G$. Note that, since  $G$ is $3$-connected, $G'$ is also $3$-connected. 
    If $G'$ is the graph $G$ itself, that is, no vertex was removed in the process, then we say that $G$ is {\em $h$-reduced}. 
    
    \begin{lemma}\label{lem:reductionKstTau}
            Let $G$ be a $3$-connected graph, let $h\geq 3$, and let $G'$ be the $h$-reduction of $G$.
            Then $\tau(G') = \tau(G)$.
        \end{lemma}
        \begin{proof}
            
            Since $G'$ is a subgraph of $G$, $\tau(G')\leq \tau(G)$. It remains to show that $\tau(G')\geq \tau(G)$.  
            
            Let $T_1,\dots, T_\ell$ and $S_1,\dots, S_\ell$ be as in the definition of $h$-reduction.
            Let $W$ be a minimum-size vertex cover of $G'$. We claim $\bigcup_{i\in [\ell]}S_i \subseteq  W$. By contradiction, suppose that there exists a vertex $w\in S_i\delete W$ for some $i\in [\ell]$. Then all edges incident to $w$ have to be covered with all $h+1$ vertices of $T_i$ remaining in $G'$. However, $S_i$ has at most $h$ vertices. Hence, replacing these $h+1$ vertices of $T_i$ with the at most $h$ vertices of $S_i$ in $W$ gives a smaller vertex cover, a contradiction.
            
            Now, we note that $W$ is also a vertex cover of $G$, implying that $\tau(G')\geq \tau(G)$. 
            To see this, observe that all edges of $G$ that are not in $G'$ are of the form $vw$ with $v\in T_i$ and $w \in S_i$, and every such edge $vw$ is covered by $w\in S_i \subseteq W$. 
        \end{proof}
         
        Let $G$ be a connected graph and let $T$ be a depth-first search (DFS) tree of $G$ from some vertex $r$ of $G$. 
        We see $T$ as being rooted at $r$, and define the usual notions of ancestors and descendants: $w$ is an {\em ancestor} of $v$ if $w$ is on the $r$--$v$ path in $T$, in which case we say that $v$ is a {\em descendant} of $w$. 
        Note that these relations are not strict: $v$ is both an ancestor and a descendant of itself. 
        By definition of DFS trees, all edges $vw$ of $G$ are such that either $v$ is a strict ancestor of $w$ in $T$  or $v$ is a strict descendant of $w$ in $T$.

        \begin{lemma}\label{lem:tworeduction}
		    For all $k, p \in \NN$, let $g_{\ref{lem:tworeduction}}(k,p)= 
		    ((p+1)2^p+kp^3)^{p+1}$.
		   Let $G$ be a $3$-connected graph such that the longest path in $G$ has length at most $p$, $G$ is $p$-reduced, and $G$ has no $\ks_k$ minor.
		    Then $|V(G)|\leq g_{\ref{lem:tworeduction}}(k,p)$.
		\end{lemma}
		\begin{proof}
		Let $T$ be a DFS tree of $G$ rooted at some vertex $r$ of $G$.  First we claim that for every vertex $v$ of $G$, at most $(p+1)2^{p}$ children of $v$ in $T$ are leaves of $T$. Indeed, for each such leaf $w$, the neighborhood of $w$ in $G$ is a subset of the set $X$ of ancestors of $v$ in $T$. Since $G$ is $p$-reduced, at most $p+1$ of these leaves have the same neighborhood in $G$. Moreover, $|X| \leq p$, since $T$ has no path of length  more than $p$, implying that there are at most $2^{p}$ choices for the neighborhood of $w$. This implies the claim.

		Let 
		\[        
		d = (p+1)2^{p} + k(p-1){p-1 \choose 2} + 1.
		\]
		If $T$ has maximum degree  at most $d$, then since $T$ has at most $p+1$ levels,   
		\[
		|V(G)| = |V(T)| \leq \sum_{i=0}^p d^i = \frac{d^{p+1}-1}{d-1}  \leq d^{p+1} \leq g_{\ref{lem:tworeduction}}(k,p),  
		\]
		as desired. 
		Hence, it is enough to show that $T$ has maximum degree at most $d$. For each $x \in V(T)$, we let $T_x$ be the subtree of $T$ rooted at $x$.  Note that if $x$ has at least two children, then the set of ancestors $A$ of $x$ is a cutset of $G$. Since $G$ is $3$-connected, $|A| \geq 3$.   Partitioning the vertices of $T$ into levels according to their distances from the root, it follows that there is only one vertex on each of the first $3$ levels. 
		We argue by contradiction and suppose that there is a vertex $v$ of $T$ having at least $d$ children in $T$. 
		Since $d \geq 2$, the set $X$ of ancestors of $v$ is a cutset of $G$ with $|X| \geq 3$. 
        This implies that $v$ is at distance at least $2$ from the root $r$ of $T$.

		Let $w$ be the ancestor of $v$ closest to $r$ in $T$ that is adjacent in $G$ to at least one vertex in $T_v$. 
		Let $P$ be the $w$--$v$ path in $T$.  If $w$ has a neighbor in $G$ which is a strict descendant of $v$, we let $v_0$ denote a child of $v$ whose subtree $T_{v_0}$ contains a neighbor of $w$, and let $w_0$ denote such a neighbor.  
		Otherwise, we just let $v_0=w_0=v$. 
		Let $C$ denote the cycle of $G$ obtained by adding the edge $ww_0$ to the $w$--$w_0$ path of $T$. 

		Recall that at most $(p+1)2^{p}$ children of $v$ are leaves of $T$. 
		Enumerate the non-leaf children of $v$ that are distinct from $v_0$ as $v_1, \dots, v_q$; thus, $q \geq d - (p+1)2^{p} - 1=k(p-1){p-1 \choose 2}$. 

		Fix some index $i\in [q]$, and let $x_i$ denote a child of $v_i$ in $T$. 
		We will construct a special $K_4$-model in $G$ using the cycle $C$ and some vertices of the subtree $T_{v_i}$. 
		The four vertex images of this $K_4$-model are denoted $V_i, X'_i, P_i^1, P_i^2$. 
		We proceed with their definitions in the next few paragraphs.  

		First, observe that every edge out of $V(T_{x_i})$ in $G- v_i$ has its other end in $P$, by our choice of $w$. 
		Choose a vertex $x'_i$ in $V(T_{x_i})$ having a neighbor $p_i^2$ in $V(P)$, with $p_i^2$ as close to $v$ on $P$ as possible (thus possibly $p_i^2=v$). 

		Since $G$ is $3$-connected, there is an $\{x'_i\}$--$V(P)$ path $Q_i$ in the graph $G-\{v_i, p_i^2\}$. Let $p_i^1$ denote the end of $Q_i$ in $V(P)$.  
		Note that all vertices of $Q_i-p_i^1$ are in $V(T_{x_i})$. 
		Also, $p_i^1$ is a strict ancestor of $p_i^2$ by our choice of $p_i^2$. 
		
		For a walk $W$ and vertices $a, b$ of $W$, we write $aWb$ to denote the $a$--$b$ subwalk of $W$. If $W_1$ and $W_2$ are walks such that $W_1$ ends at the same vertex that $W_2$ starts, we let $W_1W_2$ denote the concatenation of $W_1$ and $W_2$.  

		Next, let $R_i$ be a $\{v_i\}$--$(V(P)\cup V(Q_i))$ path in the graph $G-\{v, x'_i\}$, and let $y_i$ denote its end distinct from $v_i$.  We choose $R_i$ so that $y_i$ is as close as possible to $V(P)$ in the graph $P\cup Q_i$.   Let $S_i$ denote the $v_i$--$x'_i$ path in $T$. If $s_i$ is the last vertex of $R_i$ contained in $S_i$, we replace $R_i$ by $S_is_iR_i$.            
		The definitions of the four vertex images $V_i, X'_i, P_i^1, P_i^2$ depend on whether $y_i \in V(P)$ or not. 
	
	\begin{figure}
	    \centering
	    \begin{tikzpicture}
	    \begin{scriptsize}
            \tikzstyle{vtx} = [circle,draw,thick,fill=black!5]
            
            \draw[black!6, fill= black!6] (0,1)--(-4., -0.8) --(-5, -3.8)--(6, -3.8)-- (4.5, -0.8)--(0,1);
            \draw[black!12, fill= black!12] (2,-1)--(1, -3.5)--(4, -3.5)--(2,-1);
            \draw[black!12, fill=black!12] (-1,-1)--(-1.8, -2.2)--(-0.7, -2.2)--(-1,-1);
            \draw[black!12, fill=black!12] (0,-1)--(-0.4, -2.2)--(0.4, -2.2)--(0,-1);
            \draw[black!12, fill=black!12] (4,-1)--(3.6, -2.2)--(4.4, -2.2)--(4,-1);
            
            \node[vtx] (c1) at (-3.5, -1){};
            \node[vtx] (cn) at (-2.5, -1){};
            \node[vtx] (v0) at (-1, -1){};
            \node[vtx] (v1) at (0, -1){};
            \node[vtx] (vi) at (2, -1){};
            \node[vtx] (vq) at (4, -1){};
            \node[vtx] (v) at (0, 1){};
            \node[vtx] (w) at (0, 5){};
            \node[vtx] (r) at (0, 6){};
            \node[vtx] (pi1) at (0, 4){};
            \node[vtx] (pi2) at (0, 2){};
            \node[vtx] (xi) at (2.5, -2){};
            \node[vtx] (xi') at (3, -3){};
            \node[vtx] (yi) at (0, 3){};
            \node[vtx] (w0) at (-1.3, -2.){};

			\node[left] at (-1.1, -1) {$v_0$};
			\node[right] at (0.1, -1) {$v_1$};
			\node[right] at (2.1, -1) {$v_i$};
			\node[right] at (4.1, -1) {$v_q$};
			\node[left] at (-0.1, 1) {$v$};
			\node[left] at (-0.1, 5) {$w$};
			\node[left] at (-0.1, 6) {$r$};
			\node[left] at (-0.1, 4) {$p_i^1$};
			\node[left] at (-0.1, 2.) {$p_i^2$};
			\node[right] at (2.5, -2.2) {$x_i$};
			\node[right] at (3.1, -3) {$x_i'$};
			\node[left] at (-0.1, 3) {$y_i$};
			\node[right] at (-1.1, -2.) {$w_0$};
			\node[] at (-3,-1){$\cdots$};
			\node[] at (1,-1){$\cdots$};
			\node[] at (3,-1){$\cdots$};
			\node[below] at (-3, -1.1){$\leq (p+1)2^p$ leaves};
            
            \draw[] (v)--(c1) (v)--(cn) (v)--(v0) (v)--(v1) (v)--(vi) (v)--(vq) (vi)--(xi) (v0)to[bend left](w0)  (r)--(w)--(pi1)--(yi)--(pi2)--(v); 
            
            \draw[] (xi) to[bend right] (xi'); 
            \draw[] (xi') to[bend right] (pi2);
            
            \draw[red, thick] (vi) to[out=260, in=80] (1.9, -1.8) to[out=260, in=40] (1.5, -2.6) to[out=100, in=310] (yi);
            \draw[fill=black!5] (1.5, -2.6) circle (2pt);
            \node[red] at (0.7, 2.5){$R_i$};
            
            \draw[green, thick] (xi') to[out=200, in=45] (2.2, -2.7) to[out=225, in=50] (2., -3.2) to[out=110, in=310] (pi1);
            \draw[fill=black!5] (2., -3.2) circle (2pt);
            \node[green] at (1.2, 2.5){$Q_i$};
            
            \draw[thick, blue,] (w)--(pi1)--(yi)--(pi2)--(v)--(v0);
            \draw[thick, blue,] (v0)to[bend left](w0); 
            \draw[thick, blue] (w0) to[bend left] (w);
            \node[left, blue] at (-1.5, 2.9){$C$};
            
            \draw[magenta, thick] (vi)--(xi); 
            \draw[magenta, thick] (xi) to[bend right] (xi'); 
            \node[right, magenta] at (2.3, -1.5){$S_i$};
            
	    \end{scriptsize}
	    \end{tikzpicture}
	    
	    \vspace{5mm}
	    \begin{tikzpicture}
	    \begin{scriptsize}
            \tikzstyle{vtx} = [circle,draw,thick,fill=black!5]
            
            \draw[black!6, fill= black!6] (0,1)--(-4., -0.8) --(-5, -3.8)--(6, -3.8)-- (4.5, -0.8)--(0,1);
            \draw[black!12, fill= black!12] (2,-1)--(1, -3.5)--(4, -3.5)--(2,-1);
            \draw[black!12, fill=black!12] (-1,-1)--(-1.8, -2.2)--(-0.7, -2.2)--(-1,-1);
            \draw[black!12, fill=black!12] (0,-1)--(-0.4, -2.2)--(0.4, -2.2)--(0,-1);
            \draw[black!12, fill=black!12] (4,-1)--(3.6, -2.2)--(4.4, -2.2)--(4,-1);
            
            \node[vtx] (c1) at (-2.5, -1){};
            \node[vtx] (cn) at (-3.5, -1){};
            \node[circle,draw,thick,fill=cyan] (v0) at (-1, -1){};
            \node[vtx] (v1) at (0, -1){};
            \node[circle,draw,thick,fill=olive] (vi) at (2, -1){};
            \node[vtx] (vq) at (4, -1){};
            \node[circle,draw,thick,fill=cyan] (v) at (0, 1){};
            \node[circle,draw,thick,fill=cyan] (w) at (0, 5){};
            \node[vtx] (r) at (0, 6){};
            \node[circle,draw,thick,fill=orange] (pi1) at (0, 4){};
            \node[circle,draw,thick,fill=cyan] (pi2) at (0, 2){};
            \node[circle,draw,thick,fill=purple] (xi) at (2.5, -2){};
            \node[circle,draw,thick,fill=purple] (xi') at (3, -3){};
            \node[circle,draw,thick,fill=orange] (yi) at (0, 3){};
            \node[circle,draw,thick,fill=cyan] (w0) at (-1.3, -2.){};

			\node[left] at (-1.1, -1) {$v_0$};
			\node[right] at (0.1, -1) {$v_1$};
			\node[right] at (2.1, -1) {$v_i$};
			\node[right] at (4.1, -1) {$v_q$};
			\node[left] at (-0.1, 1) {$v$};
			\node[left] at (-0.1, 5) {$w$};
			\node[left] at (-0.1, 6) {$r$};
			\node[left] at (-0.1, 4) {$p_i^1$};
			\node[left] at (-0.1, 2.) {$p_i^2$};
			\node[right] at (2.5, -2.2) {$x_i$};
			\node[right] at (3.1, -3) {$x_i'$};
			\node[left] at (-0.1, 3) {$y_i$};
			\node[right] at (-1.2, -2.) {$w_0$};
			\node[] at (-3,-1){$\cdots$};
			\node[] at (1,-1){$\cdots$};
			\node[] at (3,-1){$\cdots$};
			\node[below] at (-3, -1.1){$\leq (p+1)2^p$ leaves};
            
            \draw[] (v)--(c1) (v)--(cn) (v)--(v0) (v)--(v1) (v)--(vi) (v)--(vq) (vi)--(xi) (v0)to[bend left](w0)  (r)--(w)--(pi1)--(yi)--(pi2)--(v); 
            
            \draw[] (xi) to[bend right] (xi'); 
            \draw[] (xi') to[bend right] (pi2);
            
            \draw[] (1.5, -2.6) to[out=100, in=310] (yi);
            \draw[olive, thick] (vi) to[out=260, in=80] (1.9, -1.8) to[out=260, in=40] (1.5, -2.6);
            \draw[fill=olive] (1.5, -2.6) circle (2pt);
            \node[olive, above] at (2.1, -2.1){$V_i$};
            
            \draw[] (2., -3.2) to[out=110, in=310] (pi1);
            \draw[purple, thick] (xi') to[out=200, in=45] (2.2, -2.7) to[out=225, in=50] (2., -3.2); 
            \draw[fill=purple] (2., -3.2) circle (2pt);
            \draw[purple, thick] (xi) to[bend right] (xi'); 
            \node[right, purple] at (2.3, -3.3){$X_i'$};
            
            \draw[thick, cyan,] (w)--(pi1) (yi)--(pi2)--(v)--(v0);
            \draw[thick, cyan,] (v0)to[bend left](w0); 
            \draw[thick, cyan] (w0) to[bend left] (w);
            \node[left, cyan] at (-0.5, 4.5){$P_i^2$};
            
            \draw[thick, orange] (pi1)--(yi);
            \node[left, orange] at (-0.1, 3.5){$P_i^1$};
            
	    \end{scriptsize}
	    \end{tikzpicture}
	    \caption{The case $y_i\in V(P)$ of the proof of Lemma~\ref{lem:tworeduction}. 
	    \label{fig:tworeduction1}}
	\end{figure}
		
		First suppose that $y_i \in V(P)$.   We define 
	    $V_i = V(R_i) \delete \{y_i\}$ and 
			$X'_i = (V(S_i) \delete V(R_i)) \cup (V(Q_i) \delete \{p_i^1\})$. 
		Notice that there is an edge $e_i$ of $S_i$ with one end in $V_i$ and the other in $X'_i$. 
		The two sets $P_i^1, P_i^2$ will be a partition of the vertices of the cycle $C$, chosen as follows.  
		If $y_i$ is a strict ancestor of $p_i^2$, let $P_i^1$ be the vertices of the $p_i^1$--$y_i$ path of $T$, and let $P_i^2=V(C) \delete P_i^1$.
		If, on the other hand, $y_i$ is a descendant of $p_i^2$, let $P_i^2$ be the vertices of the $p_i^2$--$y_i$ path of $T$, and let $P_i^1= V(C) \delete P_i^2$.  This case is illustrated in Figure~\ref{fig:tworeduction1}.

		We now argue that the sets $V_i, X'_i, P_i^1, P_i^2$ do form a $K_4$-model in this case.  
		These sets are connected, there is an edge between $P_i^1$ and $P_i^2$ (because of the cycle $C$), there is an edge between $X'_i$ and $P_i^j$ for $j \in [2]$ (because $p_i^j \in P_i^j$), there is an edge between $V_i$ and $X'_i$ (namely, $e_i$), and finally there is an edge between $V_i$ and $P_i^j$ for $j \in [2]$ (because one of $v, y_i$ is in $P_i^1$ and the other is in $P_i^2$).  
		This concludes the case where $y_i \in V(P)$. 
    
	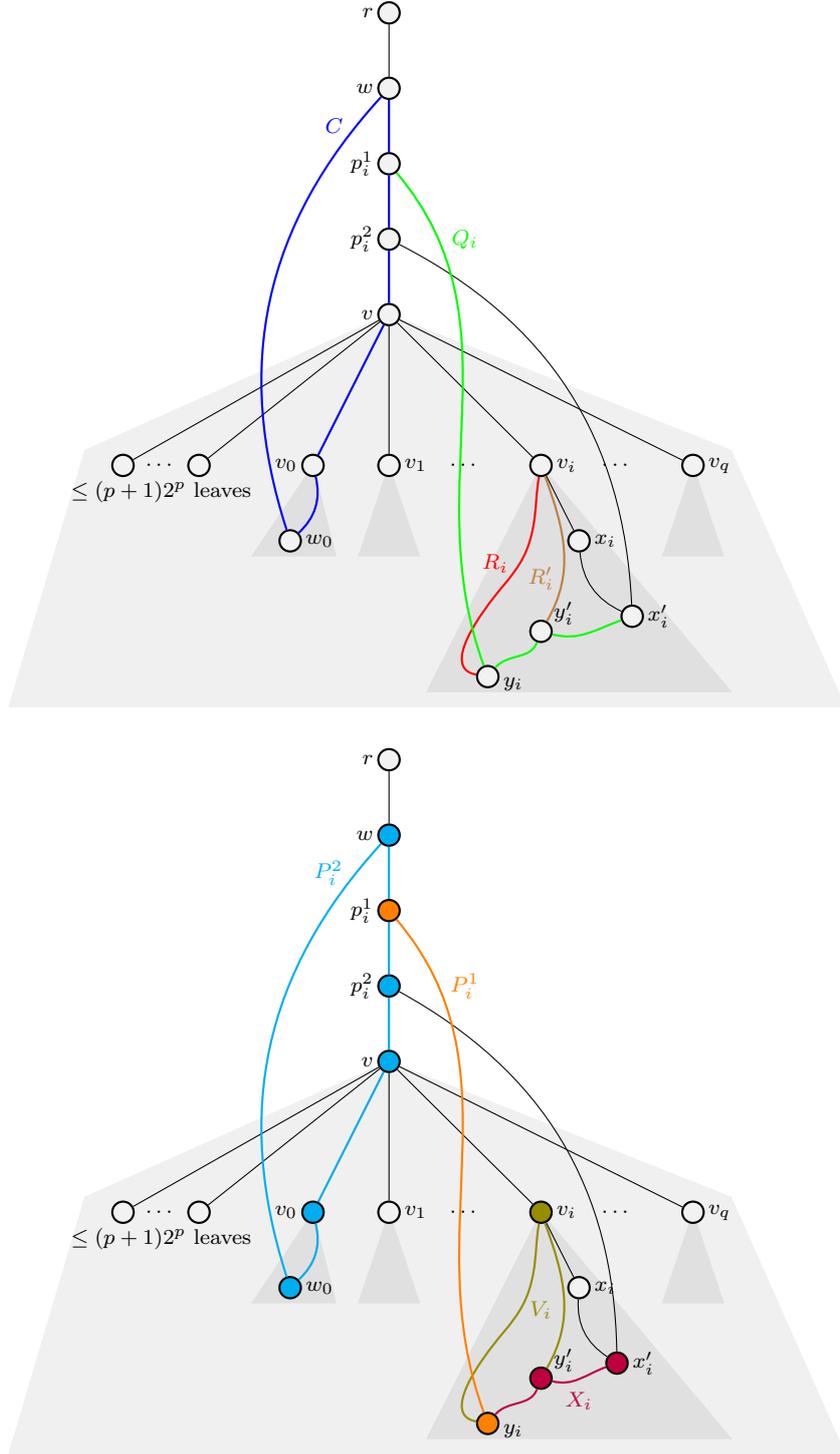
\begin{figure}
	    \centering
	    \begin{tikzpicture}
	    \begin{scriptsize}
            \tikzstyle{vtx} = [circle,draw,thick,fill=black!5]
            
            \draw[black!6, fill= black!6] (0,1)--(-4., -0.8) --(-5, -4.2)--(6, -4.2)-- (4.5, -0.8)--(0,1);
            \draw[black!12, fill= black!12] (2,-1)--(0.5, -4.)--(4.5, -4.)--(2,-1);
            \draw[black!12, fill=black!12] (-1,-1)--(-1.8, -2.2)--(-0.7, -2.2)--(-1,-1);
            \draw[black!12, fill=black!12] (0,-1)--(-0.4, -2.2)--(0.4, -2.2)--(0,-1);
            \draw[black!12, fill=black!12] (4,-1)--(3.6, -2.2)--(4.4, -2.2)--(4,-1);
            
            \node[vtx] (c1) at (-3.5, -1){};
            \node[vtx] (cn) at (-2.5, -1){};
            \node[vtx] (v0) at (-1, -1){};
            \node[vtx] (v1) at (0, -1){};
            \node[vtx] (vi) at (2, -1){};
            \node[vtx] (vq) at (4, -1){};
            \node[vtx] (v) at (0, 1){};
            \node[vtx] (w) at (0, 4){};
            \node[vtx] (r) at (0, 5){};
            \node[vtx] (pi1) at (0, 3){};
            \node[vtx] (pi2) at (0, 2){};
            \node[vtx] (xi) at (2.5, -2){};
            \node[vtx] (xi') at (3.2, -3){};
            \node[vtx] (yi) at (1.3, -3.8){};
            \node[vtx] (yi') at (2., -3.2){};
            \node[vtx] (w0) at (-1.3, -2.){};

			\node[left] at (-1.1, -1) {$v_0$};
			\node[right] at (0.1, -1) {$v_1$};
			\node[right] at (2.1, -1) {$v_i$};
			\node[right] at (4.1, -1) {$v_q$};
			\node[left] at (-0.1, 1) {$v$};
			\node[left] at (-0.1, 4) {$w$};
			\node[left] at (-0.1, 5) {$r$};
			\node[left] at (-0.1, 3) {$p_i^1$};
			\node[left] at (-0.1, 2) {$p_i^2$};
			\node[right] at (2.6, -2) {$x_i$};
			\node[right] at (3.3, -3) {$x_i'$};
			\node[right] at (1.4, -3.9) {$y_i$};
			\node[below] at (2.3, -2.7) {$y_i'$};
			\node[right] at (-1.2, -2.) {$w_0$};
			\node[] at (-3,-1){$\cdots$};
			\node[] at (1,-1){$\cdots$};
			\node[] at (3,-1){$\cdots$};
			\node[below] at (-3, -1.1){$\leq (p+1)2^p$ leaves};

            \draw[] (v)--(c1) (v)--(cn) (v)--(v0) (v)--(v1) (v)--(vi) (v)--(vq) (vi)--(xi) (v0)to[bend left](w0)  (r)--(w)--(pi1)--(pi2)--(v); 
            
            \draw[] (xi) to[bend right] (xi'); 
            \draw[] (xi') to[bend right] (pi2);
            
            \draw[red, thick] (vi) to[out=260, in=80] (1.9, -1.8) to[out=260, in=50] (1.5, -2.6) to[out=230, in=170] (yi);
            \node[red] at (1.4, -2.3){$R_i$};
            
            \draw[green, thick] (xi') to[out=200, in=345] (yi') (yi') to[in=50, out=250] (yi) (yi) to[out=110, in=310] (pi1);
            \node[green] at (1., 2.){$Q_i$};
            
            \draw[brown, thick] (vi) to[out=290, in=60] (yi');
            \node[brown] at (2., -2.5){$R_i'$};
            
            \draw[thick, blue,] (w)--(pi1)--(pi2)--(v)--(v0);
            \draw[thick, blue,] (v0)to[bend left](w0); 
            \draw[thick, blue] (w0) to[bend left] (w);
            \node[left, blue] at (-0.5, 3.5){$C$};
	    \end{scriptsize}
	    \end{tikzpicture}
	    
	    \vspace{5mm}
	    \begin{tikzpicture}
	    \begin{scriptsize}
            \tikzstyle{vtx} = [circle,draw,thick,fill=black!5]
            
            \draw[black!6, fill= black!6] (0,1)--(-4., -0.8) --(-5, -4.2)--(6, -4.2)-- (4.5, -0.8)--(0,1);
            \draw[black!12, fill= black!12] (2,-1)--(0.5, -4.)--(4.5, -4.)--(2,-1);
            \draw[black!12, fill=black!12] (-1,-1)--(-1.8, -2.2)--(-0.7, -2.2)--(-1,-1);
            \draw[black!12, fill=black!12] (0,-1)--(-0.4, -2.2)--(0.4, -2.2)--(0,-1);
            \draw[black!12, fill=black!12] (4,-1)--(3.6, -2.2)--(4.4, -2.2)--(4,-1);
            
            \node[vtx] (c1) at (-3.5, -1){};
            \node[vtx] (cn) at (-2.5, -1){};
            \node[circle,draw,thick,fill=cyan] (v0) at (-1, -1){};
            \node[vtx] (v1) at (0, -1){};
            \node[circle,draw,thick,fill=olive] (vi) at (2, -1){};
            \node[vtx] (vq) at (4, -1){};
            \node[circle,draw,thick,fill=cyan] (v) at (0, 1){};
            \node[circle,draw,thick,fill=cyan] (w) at (0, 4){};
            \node[vtx] (r) at (0, 5){};
            \node[circle,draw,thick,fill=orange] (pi1) at (0, 3){};
            \node[circle,draw,thick,fill=cyan] (pi2) at (0, 2){};
            \node[vtx] (xi) at (2.5, -2){};
            \node[circle,draw,thick,fill=purple] (xi') at (3, -3){};
            \node[circle,draw,thick,fill=orange] (yi) at (1.3, -3.8){};
            \node[circle,draw,thick,fill=purple] (yi') at (2, -3.2){};
            \node[circle,draw,thick,fill=cyan] (w0) at (-1.3, -2.){};

			\node[left] at (-1.1, -1) {$v_0$};
			\node[right] at (0.1, -1) {$v_1$};
			\node[right] at (2.1, -1) {$v_i$};
			\node[right] at (4.1, -1) {$v_q$};
			\node[left] at (-0.1, 1) {$v$};
			\node[left] at (-0.1, 4) {$w$};
			\node[left] at (-0.1, 5) {$r$};
			\node[left] at (-0.1, 3) {$p_i^1$};
			\node[left] at (-0.1, 2) {$p_i^2$};
			\node[right] at (2.6, -2) {$x_i$};
			\node[right] at (3.1, -3) {$x_i'$};
			\node[right] at (1.4, -3.9) {$y_i$};
			\node[below] at (2.3, -2.7) {$y_i'$};
			\node[right] at (-1.2, -2.) {$w_0$};
			\node[] at (-3,-1){$\cdots$};
			\node[] at (1,-1){$\cdots$};
			\node[] at (3,-1){$\cdots$};
			\node[below] at (-3, -1.1){$\leq (p+1)2^p$ leaves};

            \draw[] (v)--(c1) (v)--(cn) (v)--(v0) (v)--(v1) (v)--(vi) (v)--(vq) (vi)--(xi) (v0)to[bend left](w0)  (r)--(w)--(pi1)--(pi2)--(v); 
            
            \draw[] (xi) to[bend right] (xi'); 
            \draw[] (xi') to[bend right] (pi2);
            
            \draw[purple, thick] (xi') to[out=200, in=345] (yi') (yi') to[in=50, out=250] (yi);
            \node[purple] at (2.5, -3.5){$X_i$};
            
            \draw[olive, thick] (vi) to[out=260, in=80] (1.9, -1.8) to[out=260, in=50] (1.5, -2.6) to[out=230, in=170] (yi);
            \draw[olive, thick] (vi) to[out=290, in=60] (yi');
            \node[olive] at (2, -2.3){$V_i$};
            
            \draw[orange, thick] (yi) to[out=110, in=310] (pi1);
            \node[orange] at (1., 2.){$P_i^1$};
            
            \draw[thick, cyan,] (w)--(pi1)--(pi2)--(v)--(v0);
            \draw[thick, cyan,] (v0)to[bend left](w0); 
            \draw[thick, cyan] (w0) to[bend left] (w);
            \node[left, cyan] at (-0.5, 3.5){$P_i^2$};
            
	    \end{scriptsize}
	    \end{tikzpicture}
	    \caption{The case $y_i\in V(Q_i)$ of the proof of Lemma~\ref{lem:tworeduction}.
	    \label{fig:tworeduction2}}
	\end{figure}
	
		Next, suppose that $y_i \notin V(P)$. 
		In this case, $y_i$ is a vertex of $Q_i-p_i^1$. 
		Consider an $\{v_i\}$--$V(Q_i)$ path $R_i'$ in $G-\{v, y_i\}$. 
		Note that, by our choice of $R_i$, the path $R_i'$ avoids $V(P)$, and thus all its vertices are in $V(T_{v_i})$. 
		Furthermore, the end $y'_i$ of $R_i'$ distinct from $v_i$ must be in the subpath $x'_iQ_iy_i - \{y_i\}$, again by our choice of $R_i$. 

		Define
		\begin{align*}
			V_i &= (V(R_i) \delete \{y_i\}) \cup (V(R_i') \delete \{y'_i\})  \\            
			X'_i &= V(x'_iQ_iy_i) \delete \{y_i\} \\
		    P_i^1 &= V(y_iQ_ip_i^1) \\ 
		    P_i^2 &= V(C) \delete \{p_i^1\} 
		\end{align*}
		Using the previous observations, one can check that $V_i, X'_i, P_i^1, P_i^2$ form a $K_4$-model in this case as well.  This case is illustrated in Figure~\ref{fig:tworeduction2}.

		This ends the definitions of the vertex images $V_i, X'_i, P_i^1, P_i^2$. 
		Observe that, in all cases, the only vertices of these sets {\em not} in the subtree $T_{v_i}$ are the vertices of the cycle $C$.

		Now, there are at most ${p-1 \choose 2}$ choices for $p_i^1$ and $p_i^2$. 
		Furthermore, when $y_i \in V(P)$, there are at most $p-2$ choices for vertex $y_i$.         
		Seeing the possibility that $y_i \notin V(P)$ as another `choice', and using that $q \geq k(p-1){p-1 \choose 2}$, we conclude that there is a set $I$ of $k$ distinct indices $i\in [q]$ that have the same pair $(p_i^1, p_i^2)$, that agree on whether $y_i \in V(P)$, and furthermore that have the same vertex $y_i$ in case $y_i \in V(P)$. 
		Letting $P^j=\bigcup_{i\in I}P_i^j$ for $j \in [2]$, we then see that $P^1, P^2$ together with the sets $V_i, X'_i$ for $i\in I$ define an $\ks_k$-model in $G$, a contradiction.
		\end{proof}

    \begin{lemma}\label{lem:boundedTau}
        For all $k \in \NN$, let  $g_{\ref{lem:boundedTau}}(k)= g_{\ref{lem:tworeduction}}(k,g_{\ref{lem:shortpath}}(k))$.  If $G$ is a $3$-connected, fan-reduced graph having no $\kall^k$ minor, then $\tau(G)\leq g_{\ref{lem:boundedTau}}(k)$.
    \end{lemma}
    \begin{proof}
        By Lemma~\ref{lem:shortpath}, the maximum length of a path in $G$ is at most $p = g_{\ref{lem:shortpath}}(k)$ since $G$ is $3$-connected, and does not have a $\kall^k$ minor. Let $G'$ be the $p$-reduction of $G$. Notice that $G'$ is $3$-connected, has no $\ks_k$ minor and the length of a longest path in $G'$ is bounded by $p$. Hence, by Lemma~\ref{lem:tworeduction}, $\tau(G')\leq |V(G')|\leq  g_{\ref{lem:tworeduction}}(k,p)$. Now, by Lemma~\ref{lem:reductionKstTau}, 
        \[
        \tau(G) = \tau(G') \leq  g_{\ref{lem:tworeduction}}(k,p) = g_{\ref{lem:tworeduction}}(k,g_{\ref{lem:shortpath}}(k)) = g_{\ref{lem:boundedTau}}(k). \qedhere
        \]
    \end{proof}
        
    \section{Finishing the proof}\label{sec:finish}

     Recall that to prove our main result, Theorem~\ref{thm:main}, it suffices to establish the existence of the functions $g_{\ref{lem:main3con}}$ and $g_{\ref{lem:mainthmaux}}$ from Lemma~\ref{lem:main2conn}.  We do this in Lemmas~\ref{lem:main3con} and~\ref{lem:mainthmaux} at the end of this section.  Before doing so, we require a few more lemmas. 
     The {\em wheel $W_n$}  is the graph obtained by adding a universal vertex to a cycle of length $n$.

    \begin{lemma}\label{lem:wheels}
        $f_\infty(W_n)\leq 4$, for all $n \geq 3$. 
    \end{lemma} 
    \begin{proof}
    Let $v_0$ be the universal vertex of $W_n$ and $W_n-v_0 = C = v_1 \cdots v_nv_1$.  
    Let  $d$ be an arbitrary distance function on $W_n$.  Define $\mathcal{S}$ to be the set of inclusion-wise minimal subsets $S$ of $E(C)$ such that $S$ is not flattenable in $(W_n, d)$.  Let $d'$ be $d$ restricted to $E(C)$.
    Let $\mathcal{S}_1$ be the sets in $\mathcal{S}$ that are not flattenable in $(C, d')$, and let $\mathcal{S}_2 = \mathcal{S} \delete \mathcal{S}_1$.

    Fix $S \in \mathcal{S}_2$ and let $\vec S$ be an orientation of $S$ such that $\vec S$ is flat in $(C, d')$.  Let the length function of $\ddir{W_n}{d}{\vec S}$ be $l$, and $Z$ be a negative directed cycle in $\ddir{W_n}{d}{\vec S}$. Since $S$ is flattenable in $(C, d')$, $Z$ must use the vertex $v_0$. By renaming vertices, we may assume that $Z$ is of the form $v_0v_1\cdots v_kv_0$. Let $P = v_1 \cdots v_k$ and $Q = v_k \cdots v_nv_1$.  We abuse notation and regard $P,Q$, and $C$ as subsets of edges or arcs whenever convenient.   
    
    Since $\vec S$ is flat in $(C, d')$, $l(C) \geq 0$.  Combining this with $l(Z) < 0$ gives 
    \begin{equation} \label{eq:wheels1}
    d(v_0v_1)+d(v_0v_k) < l(Q) \leq d(Q) \text { and }  d(v_0v_1)+d(v_0v_k) < l(P) \leq d(P).
    \end{equation}

    Let $H_1$ and $H_2$ be the subgraphs of $W_n$ induced by $\{v_0, v_1,\dots, v_k\}$ and $\{v_0,v_k,\dots, v_n, v_1\}$, respectively.  Let $d_i$ be the restriction of $d$ to $H_i$.  
    Clearly, each $(H_i, d_i)$ can be covered by two flat sets $F_i^1, F_i^2$.  By \eqref{eq:wheels1},  every negative directed cycle $W$ in $\ddir{W_n}{d}{F_i^j}$ can be shortened to a negative directed cycle $W'$ in $\ddir{H_i}{d_i}{F_i^j}$ for all $i,j \in [2]$.  Therefore, $F_i^j$ is also flat in $(W_n, d)$ for all $i,j \in [2]$.
    Thus, $(W_n, d)$ has a flat cover of size $4$.
    
    We may therefore assume that $\mathcal{S}_2 = \emptyset$.  That is, every set in $\mathcal{S}$ is not flattenable in $(C, d')$. Let $U$ be the set of edges of $W_n$ incident to $v_0$. Note that $U$ is flattenable in $(W_n, d)$ by Lemma~\ref{lem:flatstar}.  If $\mathcal{S}_1 = \emptyset$, then $E(C)$ is flattenable in $(W_n, d)$, and so $E(W_n)$ is the union of two flattenable sets, $E(C)$ and $U$.  Therefore, we may assume $\mathcal{S}_1 \neq \emptyset$ and choose $T \in \mathcal{S}_1$.  Let $X \subseteq E(C)$.  Observe that if $\sum_{e \in X} d(e) \leq \frac{1}{2} d(C)$, then $X$ is flattenable in $(C, d')$.    
     It follows that for every $X \subseteq E(C)$, at least one of $X$ or $E(C) \delete X$ is flattenable in $(C, d')$.  
    Since $T$ is not flattenable in $(C, d')$, $E(C) \delete T$ is flattenable in $(C, d')$. Since $\mathcal{S}_2 = \emptyset$, $E(C) \delete T$ is flattenable in $(W_n, d)$.  By minimality, $T$ is the union of two flattenable sets $T_1$ and $T_2$ of $(W_n, d)$.  
    Thus, $E(W_n) = (E(C) \delete T) \cup T_1 \cup T_2 \cup U$, as required.  
    \end{proof}

We now generalize Lemma~\ref{lem:wheels}.  This generalization is analagous to Lemma~\ref{lem:gluength2} for $2$-connected treewidth-$2$ graphs.   
 
    \begin{lemma}\label{lem:wheel_glued}
        Let $H$ be a graph obtained by gluing  $2$-connected graphs $G_1, \dots, G_m$ on distinct edges of the wheel $W_n$, such that $H$ has no $\ks_k$ minor.
        Let $M = \max_{i\in [m]}{f_\infty(G_i)}$. 
        Then $f_\infty(H)\leq (k+7)M$.
    \end{lemma}
    \begin{proof}
        Let $W_n-v_0=C=v_1 \cdots v_n$. 
        We proceed by induction on $|V(H)|$.  By Lemma~\ref{lem:suppressDeg2}, we may assume that $H$ has minimum degree at least $3$.
        Let $E_0$ be the set of glued edges incident to $v_0$.  If $|E_0| \geq k$, then $W_n$ has a $k$-glumpkin minor.  By Lemma~\ref{lem:glumstar}, $H$ contains an $\ks_k$ minor, which is a contradiction.  Thus, $|E_0| \leq k-1$.  
        
        Let $d$ be an arbitrary distance function on $H$, and $d_W$ be the restriction of $d$ to $W_n$. By Lemma~\ref{lem:wheels}, $(W_n, d_W)$ has a flat cover of size $4$, say $F_1, F_2, F_3, F_4$. 
        Let $F_0$ be the set of arcs of $D(W_n)$ incident to $v_0$.  For each $i \in [4]$, let 
     $\Gamma_i^+, \Gamma_i^-$ be such that $\Gamma_i^+ \cup \Gamma_i^- = F_i\delete F_0$ and $(v_{j+1},v_{j}) \notin \Gamma_i^+$, $(v_{j},v_{j+1})\notin \Gamma_i^{-}$ for all $j \in \mathbb Z / n \mathbb Z$. Since every two arcs of $\Gamma_i^{\pm}$ are both forward or both backward arcs of every directed cycle of $D(W_n)$, $(\Gamma_i^{\pm}, F_i)$ is a frame of $(W_n, d_W)$ for all $i\in [4]$.  Let $H'$ be the graph obtained from $W_n$ by only gluing along glued edges belonging to $E(C)$.  
        By Lemma~\ref{lem:flatstar} and Lemma~\ref{lem:special2sum}, $f_\infty(H') \leq 1+8M$.   Since $|E_0| \leq k-1$, Lemma~\ref{lem:gluing} implies that 
        \[
        f_\infty(H) \leq f_\infty (H')+(k-1)(M-1) \leq 
        (k+7)M. \qedhere
        \] 
    \end{proof}
    
   We now apply our results about wheels to fan-reduced graphs. Recall that every graph can be obtained from its fan-reduction by replacing fan gadgets by fans.
    
        \begin{lemma}\label{lem:reductionFan} 
            Let $F$ be a reducible fan of a graph $G$, and let $G'$ be the $F$-reduction of $G$. Then $f_\infty(G)\leq f_\infty(G')+4$. 
        \end{lemma}
        \begin{proof}
            Let $v_0$ be the center of $F$, and $v_1 \cdots v_k$ be its outer path.  When performing the $F$-reduction, we rename vertices such that $v_0$ is still the center and $v_1v_2v_{k-1}v_k$ is the outer path of the reduced fan.   Let $W_{k-2}$ be the wheel graph on $k-1$ vertices, where $v_0$ is the universal vertex, and $v_2v_3 \cdots v_{k-1}v_2$ is the outer cycle. Let $H$ be the graph obtained by performing the $3$-sum of $G'$ with $W_{k-2}$ along the clique $v_0v_2v_{k-1}$. Note that $G$ is obtained from $H$ by deleting the edge $v_2v_{k-1}$. Hence, $f_\infty(G)\leq f_\infty(H)$.   By Lemma~\ref{lem:wheels}, $f_\infty(W_{k-2}) \leq 4$. Therefore, applying Lemma~\ref{lem:gluing},
            \[f_\infty(G) \leq f_\infty(H) \leq f_\infty(G') + f_\infty(W_{k-2}) \leq f_\infty(G') + 4. \qedhere 
            \] 
        \end{proof}
        
        \begin{lemma}\label{lem:fanreplacement}
            Let $G$ be a graph, $G'$ be the fan-reduction of $G$, and $t$ be the number of reduced fans in $G'$.  Then, $t \leq \tau(G')$ and $f_\infty(G) \leq 5\tau(G')$.
        \end{lemma}
        \begin{proof}
            Suppose $F'$ is a reduced fan in $G'$, where $v_0$ is the center and $v_1 \cdots v_4$ is  the outer path.  Note that every vertex cover of $G'$ must use at least one of $v_2$ or $v_3$.  Since $\{v_2, v_3\}$ is disjoint from all other reduced fans, we conclude that $t \leq \tau(G')$. 
            For the second part, first observe that 
            $f_\infty(G')\leq \tau(G')$,  by Lemma~\ref{lem:vertexcover}. By repeatedly applying Lemma~\ref{lem:reductionFan} to each maximal reducible fan of $G$,
            \[
            f_\infty(G)\leq f_\infty(G') + 4t \leq 5\tau(G'). \qedhere
            \]
        \end{proof}

    \begin{lemma}\label{lem:main3con}
        For all $k \in \mathbb \NN$, let $g_{\ref{lem:main3con}}(k) = 5g_{\ref{lem:boundedTau}}(k)$.
        If $G$ is a $3$-connected graph with no $\kall^k$ minor, then $f_\infty(G)\leq g_{\ref{lem:main3con}}(k)$. 
    \end{lemma}
    \begin{proof}
        Let $G'$ be the fan-reduction of $G$. By Lemmas~\ref{lem:fanreplacement} and~\ref{lem:boundedTau}, 
        \[
        f_\infty(G)\leq 5\tau(G') \leq 5g_{\ref{lem:boundedTau}}(k)= g_{\ref{lem:main3con}}(k). \qedhere
        \]
    \end{proof}
 
    \begin{lemma}\label{lem:mainthmaux}
        For all $k,M \in \NN$, let $g_{\ref{lem:mainthmaux}}(k,M)=(2k+11)M g_{\ref{lem:boundedTau}}(k)$.   
        Let $G$ be a $3$-connected graph and let $H$ be a graph obtained by gluing $2$-connected graphs $G_1, \dots, G_m$ on distinct edges of $G$ such that $H$ has no $\kall^k$ minor.   
        Let $M = \max_{i\in [m]}{f_\infty(G_i)}$. 
        Then $f_\infty(H)\leq g_{\ref{lem:mainthmaux}}(k, M)$. 
    \end{lemma}
    \begin{proof}
        We proceed by induction on $|E(H)|$.  By Lemma~\ref{lem:suppressDeg2}, we may assume that $H$ has minimum degree at least $3$.  Let $\mathcal F$ be the set of maximal reducible fans in $G$.
        Let $G'$ be the fan-reduction of $G$ and let $\mathcal F'$ be the set of reduced fans in $G'$. If $F$ is a fan with center $v_0$ and outerpath $v_1 \cdots v_m$, we define $I(F)=V(F) \delete \{v_0, v_1, v_m\}$.  Let $X'$ be a vertex cover of $G'$ and set $X = X' \delete \bigcup_{F' \in \mathcal F'} I(F')$. 
        We regard $X$ as a subset of vertices of $G$.  
        Let $\Gamma$ be the set of glued edges of $G$ and $\Gamma_X$ be the set of edges of $\Gamma$ incident to a vertex in $X$. 
        
        If $|\Gamma_X| > (k-1) \tau(G')$, then there is a vertex  $x \in X$ incident to at least $k$ glued edges $xy_1, \dots, xy_k$. Since $G$ is $3$-connected, there is a tree in $G-x$ containing $\{y_1, \dots, y_k\}$.  
        Therefore, $G$ contains a $k$-glumpkin minor that is obtained by contracting the tree to a single vertex.  By Lemma~\ref{lem:glumstar}, $H$ contains an $\ks_k$ minor, which is a contradiction. Hence, $|\Gamma_X| \leq (k-1)\tau(G')$. 
        
        Let $F \in \mathcal F$ with center $v_0$ and outerpath $v_1 \cdots v_m$.  Let $F^+$ be the graph obtained from $F$ by adding the edge $v_1v_m$ (if it is not already present) and gluing all $G_i$ whose glued edge is contained in $E(F)$.

        Let $G^X$ be obtained from $G$ by gluing all $G_i$ whose glued edge belongs to $\Gamma_X$ and replacing each $F \in \mathcal F$ by a triangle, $\Delta_F$. 
        Let $H^+$ be obtained from $G^X$ by simultaneously taking the clique-sum of $F^+$ and $G^X$ along $\Delta_F$ for all $F \in \mathcal F$. Notice that $H$ is a subgraph of $H^+$. 
        
        By Lemma~\ref{lem:fanreplacement}, $f_\infty(G) \leq 5\tau(G')$.  Since $|\Gamma_X| \leq (k-1)\tau(G')$, by Lemma~\ref{lem:gluing} 
        \[
        f_\infty(G^X) \leq f_\infty(G)+(k-1)(M-1)\tau(G') \leq (k+4)M\tau(G').
        \]
        Since $G'$ is a $3$-connected fan-reduced graph not containing a $\kall^k$ minor, by Lemma~\ref{lem:boundedTau}, $\tau(G') \leq g_{\ref{lem:boundedTau}}(k)$. 
        By Lemma~\ref{lem:wheel_glued}, $f_\infty(F^+) \leq (k+7)M$, for all $F \in \mathcal F$. Finally,  $|\mathcal F| \leq \tau(G')$, by Lemma~\ref{lem:fanreplacement}.  Putting this altogether, 
        \begin{align*}
        f_\infty(H) &\leq f_\infty(H^+) \\
                    &\leq f_\infty(G^X) + (k+7)M\tau(G') \\
                    &\leq (k+4)M\tau(G') + (k+7)M \tau(G')\\
                    &= (2k+11)M \tau(G') \\
                    &\leq (2k+11)M g_{\ref{lem:boundedTau}}(k) \\
                    &= g_{\ref{lem:mainthmaux}}(k, M). \qedhere
        \end{align*}
        \end{proof}
 
\section*{Acknowledgements} 

We thank Monique Laurent and Antonios Varvitsiotis for helpful discussions regarding the material in Section 2. 
We also thank an anonymous referee for their helpful comments on an earlier version of the paper.

\bibliographystyle{abbrv}
\bibliography{ReferencesNew}
\include{ReferencesNew.bbl}

\end{document}